\documentclass[a4paper]{amsart}

\usepackage{amsmath,amsfonts,amssymb,amsthm} 
\usepackage{mathrsfs} 
\usepackage[shortlabels]{enumitem} 
\usepackage{pict2e} 
\usepackage{caption} 
\usepackage[T1]{fontenc} 

\usepackage[a4paper,left=2.8cm,right=2.8cm, bottom=4cm, top=3cm]{geometry}

\usepackage{mathtools} 
\mathtoolsset{centercolon} 

\usepackage{tikz}
\usetikzlibrary{cd}

\usepackage{subcaption}

\usepackage{mathabx}
\usepackage{kbordermatrix}

\usepackage[hypertexnames=false,
 	pdfpagemode=UseNone,
 	extension=pdf,
 	colorlinks=true,
 	linkcolor=blue,
 	citecolor=red,
 	urlcolor=blue,
 ]{hyperref}

\theoremstyle{definition}
\newtheorem{definition}{Definition}[section]
\newtheorem{assumption}[definition]{Assumption}
\newtheorem{notation}[definition]{Notation}

\theoremstyle{plain}
\newtheorem{lemma}[definition]{Lemma}
\newtheorem{theorem}[definition]{Theorem}
\newtheorem{proposition}[definition]{Proposition}

\newtheorem{conjecture}[definition]{Conjecture}

\theoremstyle{remark}
\newtheorem{remark}[definition]{Remark}
\newtheorem{example}[definition]{Example}

\newcommand{\bp}{\mathbf{p}}

\newcommand{\bA}{\mathbb{A}}

\newcommand{\bC}{\mathbb{C}}

\newcommand{\bF}{\mathbb{F}}

\newcommand{\bN}{\mathbb{N}}

\newcommand{\bP}{\mathbb{P}}
\newcommand{\bQ}{\mathbb{Q}}
\newcommand{\bR}{\mathbb{R}}

\newcommand{\bZ}{\mathbb{Z}}

\newcommand{\calD}{\mathcal{D}}

\newcommand{\calL}{\mathcal{L}}

\newcommand{\calO}{\mathcal{O}}

\newcommand{\calX}{\mathcal{X}}

\newcommand{\rmE}{\mathrm{E}}

\newcommand{\rmG}{\mathrm{G}}
\newcommand{\rmH}{\mathrm{H}}
\newcommand{\rmI}{\mathrm{I}}

\newcommand{\rmK}{\mathrm{K}}

\newcommand{\rmM}{\mathrm{M}}

\newcommand{\rmV}{\mathrm{V}}

\newcommand{\rmZ}{\mathrm{Z}}

\DeclareMathOperator{\disc}{disc}

\DeclareMathOperator{\divop}{div}
\DeclareMathOperator{\Gr}{Gr}

\DeclareMathOperator{\im}{im}

\DeclareMathOperator{\MW}{MW}

\DeclareMathOperator{\NS}{NS}
\DeclareMathOperator{\Pic}{Pic}

\DeclareMathOperator{\rank}{rank}

\DeclareMathOperator{\Stab}{Stab}

\numberwithin{table}{section}

\begin{document}

\title[Degenerations and Fibrations of K3 Surfaces]{Degenerations and Fibrations of K3 Surfaces: Lattice Polarisations and Mirror Symmetry}



\author[L. Giovenzana]{Luca Giovenzana}
\address{Department of Mathematics, University of Antwerp, Campus Middelheim, Middelheimlaan 1, 2020 Antwerpen, Belgium}
\email{luca.giovenzana@uantwerpen.be}

\author[A. Thompson]{Alan Thompson}
\address{Department of Mathematical Sciences, Loughborough University, Loughborough, Leicestershire, LE11~3TU, United Kingdom}
\email{A.M.Thompson@lboro.ac.uk}
\thanks{The authors were supported by Engineering and Physical Sciences Research Council (EPSRC) New Investigator Award EP/V005545/1.\\
\indent For the purpose of open access, the authors have applied a Creative Commons Attribution (CC BY) licence to any Author Accepted Manuscript version arising.}

\begin{abstract}
Tyurin degenerations of K3 surfaces are degenerations whose central fibre consists of a pair of rational surfaces glued along a smooth elliptic curve. We study the lattice theory of such Tyurin degenerations, establishing a notion of lattice polarisation that is compatible with existing definitions for the general fibre and the rational surfaces comprising the central fibre. We separately consider elliptically fibred K3 surfaces, where the base of the fibration admits a splitting into a pair of discs with specified monodromy around the boundary. In this setting we establish a notion of lattice polarisation for the induced elliptic fibrations over discs, which is compatible with the existing definition for K3 surfaces. Finally, we discuss the mirror symmetric correspondence between these two settings.
\end{abstract}

\date{}
\maketitle

\section{Introduction}



A K3 surface of degree 2 is given by a double cover of the projective plane branched along a sextic curve $(f_6 = 0) \subset \bP^2$; such a K3 surface can be described explicitly as a weighted projective hypersurface $(z^2 = f_6)\subset \bP(1,1,1,3)$. A simple degeneration of such a K3 surface occurs when the sextic curve degenerates to a double cubic. Explicitly, consider the family of K3 surfaces over the punctured disc $\Delta^*$ defined by the equation $(z^2={f_3}^2 + tf_6)\subset \bP(1,1,1,3)\times \Delta^*$, where $t \in \Delta^*$. This family can be completed to a flat family $\mathcal X \to \Delta$ with smooth total space $\calX$ and central fibre $X_0 = V_1\cup_C V_2$ , where $V_1\cong \mathbb P^2$ and $V_2\cong \textrm{Bl}_{18}\mathbb P^2$ are rational surfaces glued along a smooth anticanonical curve isomorphic to $(f_3=0)\subset \mathbb P^2$. 

One may associate a lattice to this degeneration in a canonical way. On $X_0$, the orthogonal complement of the hyperplane section inside $K_{V_1}^{\perp}\oplus K_{V_2}^{\perp} \subset \NS(V_1) \oplus \NS(V_2)$ is generated by classes $(e_i-e_{i+1})$, where $e_1,\ldots,e_{18} \in \NS(V_2)$ are the classes of the exceptional divisors in $V_2\cong \textrm{Bl}_{18}\mathbb P^2$. These classes span the negative definite root lattice $A_{17}$.

Degenerations such as this one, where a K3 surface degenerates into a pair of smooth rational surfaces $V_i$ glued along a smooth curve $C$, are called \emph{Tyurin degenerations}. By adjunction $C$ is necessarily an elliptic curve and is anticanonical in each of the surfaces $V_i$. We call such smooth rational surfaces with smooth anticanonical divisor \emph{quasi del Pezzo surfaces}.

Besides the degeneration described above, a K3 surface of degree $2$ admits three other Tyurin degenerations, which have corresponding lattices $D_{16}A_1$, $E_8E_8A_1$, and $E_7D_{10}$ \cite[Section 5]{npgttk3s}. Dolgachev observed that under mirror symmetry such degenerations correspond to elliptic fibrations on the mirror K3 surfaces, with reducible fibres governed by the lattices above \cite[Remark 7.11]{mslpk3s}. Indeed, in this case we see that mirrors to K3 surfaces of degree $2$ (which are polarised by the lattice $H\oplus E_8 \oplus E_8 \oplus A_1$) admit four different elliptic fibrations, with reducible singular fibres of Kodaira types $\rmI_{18}\ (\widetilde{A}_{17})$, $\rmI_{12}^*\,\rmI_2\ (\widetilde{D}_{16}\widetilde{A}_1)$, $\rmI\rmI^*\,\rmI\rmI^*\,\rmI_2\ (\widetilde{E}_8\widetilde{E}_8\widetilde{A}_1$), and $\rmI\rmI\rmI^*\, \rmI_6^*\ (\widetilde{E}_7\widetilde{D}_{10})$ respectively.

The DHT philosophy, introduced in \cite{mstdfcym}, attempts to link this observation to the theory of mirror symmetry for the surfaces $V_i$. It postulates that the base of each elliptic fibration $Y \to \bP^1$ on the mirror K3 surface should admit a splitting $\bP^1 \cong \Delta_1 \cup_{\gamma} \Delta_2$ into a pair of discs $\Delta_i$ glued along a simple loop $\gamma$, so that the restriction $Y_i \to \Delta_i$ of the fibration to the disc $\Delta_i$ is the Landau-Ginzburg (LG) model for the quasi del Pezzo surface $V_i$.

In our running example of the degeneration of a K3 surface of degree $2$, it is known that the mirror to $V_1 \cong \bP^2$ is a Landau-Ginzburg  model consisting of an elliptic fibration over $\bA^1$ with three singular fibres of Kodaira type $\rmI_1$ and monodromy $\begin{psmallmatrix} 1 & -9 \\ 0 & 1 \end{psmallmatrix}$ around the point at infinity (see, for instance, \cite{msdpsvccs}). This may be equivalently viewed as a fibration over a disc $\Delta_1$, by simply removing a small neighbourhood of the point at infinity. In this setting, the DHT philosophy suggests that there should exist a simple loop $\gamma$ in the base of the mirror fibration $Y \to \bP^1$, such that there are three singular fibres of types $\rmI_1\, \rmI_1\, \rmI_1$ on one side of $\gamma$, four singular fibres of types $\rmI_1\, \rmI_1\, \rmI_1\, \rmI_{18}$ on the other side of $\gamma$, and such that monodromy around $\gamma$ is given by $\begin{psmallmatrix} 1 & -9 \\ 0 & 1 \end{psmallmatrix}$. Moreover, there should be some sense in which an elliptic fibration over a disc with singular fibres of types $\rmI_1\, \rmI_1\, \rmI_1\, \rmI_{18}$ may be thought of as an LG model for $V_2\cong \textrm{Bl}_{18}\mathbb P^2$ (which is not a del Pezzo surface and so does not have an LG model in the traditional sense).

The aim of this paper is to establish a general framework for this theory. There are substantial technical obstacles to doing this, related to establishing compatibility between the various theories of mirror symmetry involved.

At the highest level, mirror symmetry for $n$-dimensional Calabi-Yau manifolds predicts an exchange of Hodge numbers $h^{p,q}(X) = h^{p,n-q}(Y)$ between a Calabi-Yau $X$ and its mirror $Y$. For K3 surfaces, however, this statement is vacuous, as all K3 surfaces have the same Hodge numbers.

A solution to this problem was given by Dolgachev, Nikulin, and Pinkham \cite{sedeask3,fggahfrsg2r,mslpk3s}, who realised that mirror symmetry for K3 surfaces could instead be formulated as an exchange between the algebraic (i.e. N\'{e}ron-Severi) and transcendental parts of $H^{1,1}(X)$. Their theory relies upon the notion of lattice polarisations, which impose constraints on the N\'{e}ron-Severi lattice of a K3 surface.

On the other hand, existing theories of mirror symmetry for quasi del Pezzo surfaces are somewhat less fine-grained. Work of Auroux, Katzarkov, and Orlov \cite{msdpsvccs} has established a version of homological mirror symmetry between del Pezzo surfaces and their LG models, and Harder and the second author \cite{pdpslf} have extended this to a homological mirror symmetry statement for quasi del Pezzo surfaces, using Kuznetsov's theory of pseudolattices \cite{ecslc}. However, these theories are not sensitive to much apart from the del Pezzo degree (i.e. the self-intersection number of the anticanonical divisor); they do not, for instance, draw any distinction between LG models for different weak del Pezzo surfaces of the same degree, such as $\bP^1 \times \bP^1$ and the Hirzebruch surface $\bF_2$.

To remedy this, one may ask for a notion of lattice polarised mirror symmetry for quasi del Pezzo surfaces and their LG models, in the style of Dolgachev's theory for K3 surfaces. However, it is not obvious how such a concept should be correctly defined: one has $H^2(V;\bZ) \cong \NS(V)$ for any quasi del Pezzo surface $V$, so there is no algebraic / transcendental decomposition to work with as there is in the K3 case. An attempt to solve this problem in the weak del Pezzo case, by defining lattice polarisations in terms of effective classes instead of algebraic classes, was made by Doran and the second author in \cite{mslpdps}.

Concretely, the aims of this paper can be summarised as follows.
\begin{enumerate}
\item Rigorously establish a theory of lattice polarisations for Tyurin degenerations of K3 surfaces $X \rightsquigarrow X_0 = V_1 \cup_C V_2$, linking Dolgachev's \cite{mslpk3s} theory of lattice polarisations for the K3 surface $X$ to the (pseudo)lattice theory of the quasi del Pezzo surfaces $V_i$ described in \cite{mslpdps,pdpslf} (Theorems~\ref{thm:polarisedtyurinstable} and \ref{thm:Lpolwdp}).
\item Enhance the theory of pseudolattices for elliptic fibrations over discs from the setting of \cite{pdpslf}, where all singular fibres have Kodaira type $\rmI_1$, to allow (non-multiple) fibres of arbitrary Kodaira types (Theorem \ref{thm:qdpfibration}).
\item Rigorously develop a theory of lattice polarisations for elliptic fibrations over discs (with arbitrary Kodaira fibres), which is compatible both with Dolgachev's \cite{mslpk3s} theory of lattice polarisations for an elliptically fibred K3 surface $Y \to \bP^1$ when there is an appropriate splitting $\bP^1 \cong \Delta_1 \cup_{\gamma} \Delta_2$ of the base, and with the theory of lattice polarised rational elliptic surfaces described in \cite{mslpdps} (Theorems \ref{thm:Gammapolfibration} and \ref{thm:ratellcompatibility}).
\item Formulate a mirror symmetric correspondence between these two settings.
\end{enumerate}

To elaborate further on (4), we will develop a rigorous definition for when a lattice polarised Tyurin degeneration $X \rightsquigarrow X_0 = V_1 \cup_C V_2$ and a lattice polarised elliptically fibred K3 surface $Y \to \bP^1$ with a splitting $\bP^1 \cong \Delta_1 \cup_{\gamma} \Delta_2$ form a \emph{mirror pair} (Definition \ref{def:mirrorpair}). We will show that this definition is compatible with existing theories of mirror symmetry in the following ways.
\begin{itemize}
\item The K3 surfaces $X$ and $Y$ are lattice polarised mirrors in the sense of \cite{mslpk3s} (Section \ref{sec:dolgachevnikulin});
\item $(V_i,C)$ and the restriction $Y_i \to \Delta_i$ form a quasi del Pezzo surface / LG model mirror pair, in the sense of \cite{pdpslf} (Theorem \ref{thm:HMS});
\item We formulate a version of lattice polarised mirror symmetry for weak del Pezzo surfaces $(V,C)$ and elliptic fibrations over discs $W \to \Delta$ (Conjecture~\ref{con:wdpmirror}), which is compatible with the theory above and the formulation of mirror symmetry for weak del Pezzo surfaces from \cite{mslpdps} (Theorem~\ref{thm:wdpcompatibility}).
\end{itemize}

Given this, we conclude by rigorously formulating the DHT philosophy for K3 surfaces as a conjecture (Conjecture \ref{conj:DHT}), as follows: If $X \rightsquigarrow X_0 = V_1 \cup_C V_2$ is a Tyurin degeneration of lattice polarised K3 surfaces, then the mirror K3 surface $Y$ to $X$ should admit an elliptic fibration $Y \to \bP^1$ with a splitting $\bP^1 \cong \Delta_1 \cup_{\gamma}  \Delta_2$, so that $X$ and $Y$ form a mirror pair. 

There remain several obstacles to proving this conjecture, principal amongst which is appropriate choice of the splitting loop $\gamma$. Evidence from \cite{nftdk3spr18l} suggests that this loop cannot be determined by lattice theoretic considerations alone; additional combinatorial data is likely to be required.

\subsection{Outline of the paper} 
This paper is structured as follows. In Section \ref{sec:pseudolattices} we begin by developing the abstract theory of lattices and pseudolattices that we will need for the remainder of the paper. The aim is to develop tools to translate between the lattice theory of the K3 surfaces $X$ and $Y$ and the theory of pseudolattices associated with the quasi del Pezzo surfaces $V_i$ and the fibrations over discs $Y_i \to \Delta_i$ (as described in \cite{pdpslf}). We then describe how to define lattice polarisations in each of these settings, in a compatible way.

In Section \ref{sec:degenerations} we apply these tools to the study of Tyurin degenerations of K3 surfaces. Many of the ideas here are extensions of ideas of Friedman \cite{npgttk3s} to the lattice polarised setting; there is also a close relationship with recent work of Alexeev and Engel on lattice polarised degenerations \cite{cmk3s}. In particular, in Section \ref{sec:degenerationpolarisation} we define lattice polarisations on Tyurin degenerations and on their central fibres, and show (Theorems \ref{thm:polarisedtyurinstable} and \ref{thm:Lpolwdp}) that these notions are compatible with the definition of a lattice polarisation on a weak del Pezzo surface given in \cite{mslpdps}. In Section \ref{sec:primitivity}, we introduce the concept of a \emph{doubly admissible} degeneration, which is a natural extension of Dolgachev's \cite{mslpk3s} $1$-admissible condition to the setting of Tyurin degenerations; in the remainder of Section \ref{sec:degenerations} we show that this assumption allows some major simplifications of the underlying theory and we explore its implications for moduli.

In Section \ref{sec:fibrations} we study elliptically fibred K3 surfaces. In this setting there is much more work to do. The first serious obstacle is that the results of \cite{pdpslf} only hold for genus $1$ Lefschetz fibrations over discs, which necessarily have all singular fibres of Kodaira type $\rmI_1$, whereas elliptically fibred K3 surfaces may have singular fibres of many different Kodaira types. Our first main result of this section, Theorem \ref{thm:qdpfibration}, shows that the results of \cite{pdpslf} may be extended to arbitrary elliptic fibrations over discs (without multiple fibres).

Using this, in Section \ref{sec:ellipticK3} we proceed to study elliptic fibrations on lattice polarised K3 surfaces $Y \to \bP^1$. We describe when a loop $\gamma \subset \bP^1$ splits the base of such a fibration into two LG models and show that the theory of Section \ref{sec:pseudolattices} applies in this setting. This allows us, in Section \ref{sec:fibrationpolarisation}, to describe a compatibility condition between the lattice polarisation on $Y$ and the splitting loop $\gamma \subset \bP^1$, under which we obtain induced lattice polarisations on the pieces $Y_i \to \Delta_i$ (Theorem \ref{thm:Gammapolfibration}). In Section \ref{sec:primitivity2} we explore the analogue of the doubly admissible condition in this setting and show that it approximately corresponds to the existence of a section for the elliptic fibration $Y \to \bP^1$ along with a $1$-admissible $0$-cusp in the Baily-Borel compactification of the moduli space of $Y$. 

To conclude our treatment of fibrations, in Section \ref{sec:rationalelliptic} we explore the compatibility between the notions of lattice polarisation defined here and in \cite{mslpdps}. This is complicated by the fact that our lattice polarisations are defined for elliptic fibrations over discs, whilst those in \cite{mslpdps} are defined for rational elliptic surfaces. Compatibility between the two definitions is established by Theorem \ref{thm:ratellcompatibility}.

In Section \ref{sec:mirror} we discuss how mirror symmetry relates the degeneration and fibration pictures presented in the previous sections. We begin by defining when a Tyurin degeneration of lattice polarised K3 surfaces $X \rightsquigarrow X_0 = V_1 \cup_C V_2$ and a lattice polarised elliptically fibred K3 surface $Y \to \bP^1$ with a splitting $\bP^1 \cong \Delta_1 \cup_{\gamma} \Delta_2$ form a \emph{mirror pair} (Definition \ref{def:mirrorpair}). We give the most general form of this definition first, which is quite complicated to state, then show that adding the doubly admissible assumption allows a dramatic simplification (Proposition \ref{prop:mirrordoublyadmissible}).

In Section \ref{sec:mirrorcompatibility} we prove compatibility between this definition and some established formulations of mirror symmetry. We demonstrate that our mirror pairs are compatible with lattice polarised mirror symmetry for the K3 surfaces $X$ and $Y$, in the sense of \cite{mslpk3s}, and homological mirror symmetry for the quasi del Pezzo surfaces $V_i$ and the pieces $Y_i \to \Delta_i$, in the sense of \cite{pdpslf} (Theorem \ref{thm:HMS}). 

Compatibility between our mirror pairs and the notion of lattice polarised mirror symmetry for lattice polarised del Pezzo surfaces (in the sense of \cite{mslpdps}) is more complicated. Firstly, one must address the fact that the mirror correspondence in \cite{mslpdps} is stated for rational elliptic surfaces, not elliptic fibrations over discs; this is dealt with by Conjecture \ref{con:wdpmirror}, which reformulates the main conjecture of \cite{mslpdps} in our setting. Secondly, and more seriously, one must also account for the fact that effective curves on $X$ (resp. $Y$) may not give rise to effective curves supported on just one component $V_i$ (resp. one piece $Y_i \to \Delta_i$); the failure of this is controlled by an object that we call the \emph{coupling group}. When this coupling group is torsion, Theorem \ref{thm:wdpcompatibility} shows that our mirror pairs are compatible with the theory from \cite{mslpdps}. On the other hand, if the coupling group is not torsion, we do not expect lattice polarised mirror symmetry for the K3 surfaces $X$ and $Y$ (in the sense of \cite{mslpk3s}) to induce lattice polarised mirror symmetry for $V_i$ and $Y_i \to \Delta_i$ (in the sense of \cite{mslpdps}).

Finally, Section \ref{sec:dht} explores the consequences of this work for the DHT philosophy. In particular, we give a precise conjectural statement (Conjecture \ref{conj:DHT}) about the mirror symmetric relationship between Tyurin degenerations of K3 surfaces and elliptically fibred K3 surfaces.
\medskip

\noindent\textbf{Data access statement:} No datasets were generated or analysed during this study.
\medskip

\noindent\textbf{Acknowledgments:} LG would like to thank Giulia Gugiatti and Franco Rota for stimulating discussions. He also wishes to thank Franco Giovenzana for engaging in fruitful discussions throughout the duration of this project. AT would like to thank Charles Doran and Andrew Harder for helpful discussions. The authors would also like to thank Klaus Hulek for helpful comments on an earlier version.  Part of this project was completed whilst AT was attending the thematic programme on \emph{K-theory, algebraic cycles and motivic homotopy theory} at the Isaac Newton Institute in the summer of 2022; he would like to thank the Isaac Newton Institute for their excellent hospitality over this period. 

\section{Pseudolattices} \label{sec:pseudolattices}

\subsection{Background}\label{sec:background}

We begin by briefly recapping some relevant background from the theory of pseudolattices, based on \cite{ecslc,pdpslf}; a much more detailed treatment of these concepts may be found in those two papers. We start with some basic definitions.

\begin{definition} \cite[Definition 2.1]{ecslc}
A \emph{pseudolattice} $(\mathrm{G},\langle \cdot, \cdot \rangle_\mathrm{G})$ is a finitely generated free abelian group $\mathrm{G}$ equipped with a (not necessarily symmetric) nondegenerate bilinear form $\langle \cdot, \cdot \rangle_\mathrm{G} \colon \mathrm{G} \times \mathrm{G} \rightarrow \mathbb{Z}$.
The \emph{rank} of $\rmG$ is its rank as a free abelian group and $\rmG$ is called \emph{unimodular} if the form $\langle \cdot, \cdot \rangle_\mathrm{G}$ induces an isomorphism between $\mathrm{G}$ and $\mathrm{Hom}_\mathbb{Z}(\mathrm{G},\mathbb{Z})$. To simplify notation, we will usually refer to $(\mathrm{G},\langle \cdot, \cdot \rangle_\mathrm{G})$ by its underlying group $\mathrm{G}$.
\end{definition}

A simple example of a pseudolattice, that we will use repeatedly, is the \emph{elliptic curve pseudolattice}.

\begin{definition} \label{def:ellipticcurvepseudolattice} \cite[Definition 2.14]{pdpslf} The \emph{elliptic curve pseudolattice}, henceforth denoted by $\mathrm{E}$, is the unimodular pseudolattice of rank $2$ generated by primitive elements $a,b$, with bilinear form defined by
\[
\langle a,b \rangle_\mathrm{E} = -1, \quad\quad \langle b,a \rangle_\mathrm{E} = 1, \quad\quad \langle a,a \rangle_\mathrm{E} = \langle b,b \rangle_\mathrm{E} = 0,
\]
extended to $\bZ^2$ by linearity.
\end{definition}

The pseudolattice $\mathrm{E}$ arises naturally from an elliptic curve $C$ in two ways (see \cite[Example 2.15]{pdpslf}).
\begin{enumerate}
\item The numerical Grothendieck group $\mathrm{K}_0^{\mathrm{num}}(\mathbf{D}(C))$ associated to the bounded derived category $\mathbf{D}(C)$ of coherent sheaves on $C$, with its Euler pairing, is isomorphic to $\mathrm{E}$. The element $a$ is identified with the class of the structure sheaf $\calO_p$ of a point $p \in C$, and $b$ may be identified with the class of $\calO_C$.
\item The first integral homology $\mathrm{H}_1(C;\bZ)$ is isomorphic to $\rmE$. In this case $(a,b)$ is taken to be a symplectic basis for $\mathrm{H}_1(C;\bZ)$ and the bilinear form $\langle \cdot,\cdot \rangle_{\rmE}$ is the negative of the standard symplectic form\footnote{\cite[Example 2.15]{pdpslf} does not specify the orientation of the basis $(a,b)$ or the sign convention for the bilinear form; this is the correct choice for compatibility with the results of \cite[Section 4]{pdpslf}.}. This may also be interpreted as the numerical Grothendieck group associated to a triangulated category, in this case the \emph{derived Fukaya category} $\mathbf{F}(C)$ associated to $C$.
\end{enumerate}
These two interpretations can be thought of as mirror to one another \cite[Section 5]{pdpslf}; the first will be important for us in Section \ref{sec:degenerations} and the second in Section \ref{sec:fibrations}.

We now define some important structures on pseudolattices: \emph{Serre operators} and \emph{exceptional bases}.

\begin{definition} \cite[Section 2.1]{ecslc}
Let $\mathrm{G}$ be a pseudolattice. A \emph{Serre operator} $S_\mathrm{G}$ is an automorphism of $\mathrm{G}$ so that $\langle u_1 , u_2 \rangle_\mathrm{G} = \langle u_2, S_\mathrm{G}(u_1) \rangle_\mathrm{G}$ for all $u_1,u_2 \in \mathrm{G}$. 
\end{definition}

We note that if a Serre operator exists then it is unique, and that every unimodular pseudolattice has a Serre operator.

\begin{definition} \cite[Definition 2.2]{ecslc}
An element $e \in\mathrm{G}$ is called \emph{exceptional} if $\langle e,e\rangle_{\mathrm{G}} = 1$. A sequence of elements $(e_1,\ldots,e_n)$ is called an \emph{exceptional sequence} if $\langle e_i,e_j \rangle_{\mathrm{G}} = 0$ for all $i > j$. An \emph{exceptional basis} for $\mathrm{G}$ is an exceptional sequence whose elements form a basis for $\mathrm{G}$.
\end{definition}

A subclass of pseudolattices, that will be important for this paper, is the \emph{surface-like} pseudolattices.

\begin{definition}\label{def:surfacelike} \cite[Definition 3.1]{ecslc}
A pseudolattice $\mathrm{G}$ is \emph{surface-like} if there is a primitive element $\mathbf{p} \in\mathrm{G}$ such that
\begin{enumerate}
\item $\langle \mathbf{p},\mathbf{p} \rangle_\mathrm{G} = 0$.
\item $\langle \mathbf{p} ,u \rangle_\mathrm{G} = \langle u , \mathbf{p}\rangle_\mathrm{G}$ for all $u \in \mathrm{G}$.
\item If $u_1,u_2 \in \mathrm{G}$ satisfy $\langle u_1,\mathbf{p} \rangle_\mathrm{G}= \langle u_2 ,\mathbf{p} \rangle_\mathrm{G} = 0$, then $\langle u_1,u_2 \rangle_\mathrm{G} = \langle u_2,u_1 \rangle_\mathrm{G}$.
\end{enumerate}
An element $\bp$ with the above properties is called a \emph{point-like} element in $\mathrm{G}$.
\end{definition}

Using a point-like element, one may define the \emph{N\'{e}ron-Severi lattice} of a surface-like pseudolattice $\rm{G}$.

\begin{definition} \cite[Section 3.2]{ecslc}
Let $\mathrm{G}$ be a surface-like pseudolattice. Define the \emph{N\'eron-Severi group} of $\mathrm{G}$ to be the group $\NS(\mathrm{G}) := {\mathbf{p}^{\perp}} / \mathbf{p}$. It is a finitely generated free abelian group, equipped with a nondegenerate integral symmetric bilinear form $q(\cdot, \cdot)$, defined as follows: if $u_1,u_2 \in {\mathbf{p}}^{\perp}$ and $[u_1], [u_2]$ are their classes in $\NS(\mathrm{G})$, then $q\left([u_1], [u_2]\right) := -\langle u_1, u_2 \rangle_\mathrm{G}$, and this definition is independent of the choice of representatives for $[u_1]$, $[u_2]$. The pair $(\NS(\mathrm{G}), q(\cdot,\cdot))$ is called the \emph{N\'eron-Severi lattice} of $\mathrm{G}$. As before, to simplify notation, we will usually refer to $(\NS(\mathrm{G}), q(\cdot,\cdot))$ by its underlying group $\NS(\mathrm{G})$. 
\end{definition}

Following \cite[Section 3.2]{ecslc}, one may define a rank function on a surface-like pseudolattice $\rmG$ as follows. For $u \in \rmG$, we have $\rank(u) := \langle  u, \mathbf{p} \rangle_\mathrm{G} = \langle \bp, u \rangle_{\mathrm{G}}$.
Using this, one may define a map $\lambda\colon \bigwedge^2 \mathrm{G}\to \mathrm{NS}(\mathrm{G})$ by
\[
\lambda( u_1 \wedge u_2) := [\rank(u_1) u_2 - \rank(u_2) u_1].
\]
This map can be used to define a distinguished class $K_\mathrm{G}$  in $\NS(\mathrm{G}) \otimes \bQ$ , called the \emph{canonical class} of $\mathrm{G}$. Existence and uniqueness of the canonical class is established in \cite[Lemma 3.12]{ecslc}.

\begin{definition}\label{def:canonical}
The \emph{canonical class of $\mathrm{G}$} is the unique class $K_\mathrm{G} \in \mathrm{NS}(\mathrm{G}) \otimes \bQ$ such that for every $u_1,u_2 \in \mathrm{G}$,
\[
\langle u_1,u_2 \rangle_\mathrm{G} - \langle u_2,u_1 \rangle_\mathrm{G} = - q(K_\mathrm{G},\lambda(u_1 \wedge u_2) ).
\]
\end{definition}

By \cite[Lemma 3.12]{ecslc}, if $\mathrm{G}$ is unimodular, then $K_\mathrm{G}$ is integral (i.e. $K_{\mathrm{G}} \in \mathrm{NS}(\mathrm{G})$).

We next introduce a special class of maps between pseudolattices.

\begin{definition} \cite[Definition 2.16]{pdpslf}
Let $\mathrm{G}$ and $\mathrm{H}$ be pseudolattices. A \emph{spherical homomorphism} from $\mathrm{G}$ to $\mathrm{H}$ is a homomorphism of abelian groups $f\colon\mathrm{G} \to \mathrm{H}$ with the following properties:
\begin{enumerate}
\item The homomorphism $f$ has a right adjoint $r\colon \mathrm{H} \to \mathrm{G}$, in the sense that 
\[
\langle f(u), v \rangle_{\mathrm{H}} = \langle u, r(v) \rangle_{\mathrm{G}}
\]
for any $u \in \mathrm{G}$ and $v \in \mathrm{H}$. 
\item The \emph{twist} and \emph{cotwist} endomorphisms, which are defined as
\[
{T}_f := \mathrm{id}_{\mathrm{H}} - fr, \quad \quad {C}_f := \mathrm{id}_{\mathrm{G}} - rf,
\]
respectively, are invertible. In fact, if ${T}_f$ is invertible with inverse $T_f^{-1}$, then ${C}_f$ is invertible with inverse $C_f^{-1} = \mathrm{id}_{\mathrm{G}} + rT_f^{-1}f$, and vice versa.
\end{enumerate}
\end{definition}

Note that, by \cite[Remark 2.17]{pdpslf}, if a right adjoint for a homomorphism $f$ exists, then it must be unique, and a right adjoint always exists if $\rmG$ is unimodular.

\begin{definition}\cite[Definition 2.18]{pdpslf}
A spherical homomorphism $f \colon \mathrm{G} \rightarrow \mathrm{H}$ is \emph{relative $(-1)^n$-Calabi-Yau} (usually abbreviated to \emph{relative $(-1)^n$-CY}) if $\mathrm{G}$ has a Serre operator $S_{\mathrm{G}}$ and ${C}_f = (-1)^nS_{\mathrm{G}}$.
\end{definition}

Note that this definition only depends upon the parity of $n$; consequently we will usually take $n$ to be an element of $\bZ/2\bZ$. An example of a relative $(-1)^0$-CY spherical homomorphism that will be important for this paper is as follows.

\begin{example}\cite[Example 2.20]{pdpslf} \label{ex:Z(v)}
Let $\mathrm{E}$ be the elliptic curve pseudolattice and let $v \in \mathrm{E}$ be a primitive vector. Note that $\langle v, v\rangle_\mathrm{E} = 0$, as the bilinear form on $\rmE$ is antisymmetric. Let $\mathrm{Z}(v)$ be the unimodular pseudolattice generated by a single element $z$, with $\langle z, z \rangle_{\mathrm{Z}(v)}  =1$. Note that the Serre operator on $\rmZ(v)$ is the identity and $(z)$ is an exceptional basis for $\mathrm{Z}(v)$.

There is a spherical homomorphism $\zeta\colon \mathrm{Z}(v) \rightarrow \mathrm{E}$ sending $z$ to $v$. The right adjoint of this homomorphism is $\rho\colon \mathrm{E} \to \mathrm{Z}(v)$ defined by $w \mapsto \langle v, w \rangle_\mathrm{E} z$.
The cotwist of $\zeta$ is the identity on $\bZ(v)$, so it follows that $\zeta$ is relative $(-1)^0$-CY. The twist of $\zeta$ acts on $w \in \mathrm{E}$ by $T_{\zeta}(w) = w - \langle v, w \rangle_\mathrm{E} v$. If we identify $\rmE$ with $\mathrm{H}_1(C;\bZ)$ for an elliptic curve $C$, and $v \in \mathrm{H}_1(C;\bZ)$ is the class of a simple oriented cycle, then $T_{\zeta}$ is the action on $\mathrm{H}_1(C;\bZ)$ induced by the Dehn twist associated to $v$.
\end{example}

We can use spherical homomorphisms to glue pseudolattices together.

\begin{definition}\cite[Definition 2.22]{pdpslf} \label{def:gluing}
Let $\mathrm{G}_1, \mathrm{G}_2$ and  $\mathrm{H}$ be pseudolattices. Assume that 
\[
f_1 \colon \mathrm{G}_1 \longrightarrow \mathrm{H}, \quad\quad f_2 \colon \mathrm{G}_2 \longrightarrow \mathrm{H}
\]
are homomorphisms of the underlying groups. Define $\mathrm{G}_1 \oright_{\mathrm{H}} \mathrm{G}_2$ to be the pseudolattice with underlying group $\mathrm{G}_1 \oplus \mathrm{G}_2$ and bilinear form $\langle \cdot , \cdot \rangle_{\mathrm{G}_1 \oright_{\mathrm{H}} \mathrm{G}_2}$ given as follows: for $u_i \in \mathrm{G}_i$ and $v_j \in \mathrm{G}_j$, define
\[
\langle u_i, v_j \rangle_{\mathrm{G}_1 \oright_{\mathrm{H}}\mathrm{G}_2}
=\begin{cases}
\langle u_i, v_j \rangle_{\mathrm{G}_i} & \text{if }  i = j\\
\langle f_i(u_i), f_j(v_j) \rangle_\mathrm{H} & \text{if } i=1,\ j = 2 \\
0 & \text{if }  i=2,\ j=1 
\end{cases}
\]
and extend to $\mathrm{G}_1 \oplus \mathrm{G}_2$ by linearity. By \cite[Proposition 2.23]{pdpslf}, there is a natural spherical homomorphism $f_1 \oright f_2 \colon \mathrm{G}_1 \oright_{\rmH} \mathrm{G}_2 \rightarrow \mathrm{H}$, sending $(u_1,u_2) \in \mathrm{G}_1 \oright_{\rmH} \mathrm{G}_2$ to $f_1(u_1) + f_2(u_2)$, which has the property that ${T}_{f_1 \oright f_2} = {T}_{f_1} \cdot {T}_{f_2}$.
\end{definition}

We will mostly be interested in the case of spherical homomorphisms where the target is the elliptic curve pseudolattice $\rmE$. In this case we have the following result, which is a special case of \cite[Proposition 2.23]{pdpslf}.

\begin{lemma} \label{lem:sumisCY}
Given two pseudolattices $\mathrm{G}_1$ and $\mathrm{G}_2$ with relative $(-1)^0$-CY spherical homomorphisms $f_i\colon \mathrm{G}_i \rightarrow \mathrm{E}$, then $f_1 \oright f_2$ is also relative $(-1)^0$-CY.
\end{lemma}

As an important example of this gluing process, we can glue the spherical homomorphisms $\rmZ(v) \to \rmE$ from Example \ref{ex:Z(v)} to build pseudolattices of higher rank.

\begin{example}\label{ex:multiZ} \cite[Example 2.24]{pdpslf}
Suppose we have an ordered $n$-tuple of elements $(v_1,\ldots,v_n)$ in $\rmE$. For each $i \in \{1,\ldots,n\}$, define rank one pseudolattices $\mathrm{Z}(v_i)$ generated by elements $z_i$ with $\langle z_i,z_i \rangle_{\mathrm{Z}(v_i)} = 1$, as in Example \ref{ex:Z(v)}. By Example \ref{ex:Z(v)}, we have relative $(-1)^0$-CY spherical homomorphisms $\zeta_i \colon \mathrm{Z}(v_i) \to \mathrm{E}$ taking $z_i$ to $v_i$.

Let $\mathrm{Z}(v_1,\ldots,v_n)$ denote the pseudolattice 
\[\mathrm{Z}(v_1,\ldots,v_n) := \mathrm{Z}(v_1) \oright_{\mathrm{E}} \mathrm{Z}(v_2) \oright_{\mathrm{E}}  \cdots \oright_{\mathrm{E}}  \rmZ(v_n).\] 
Then $\mathrm{Z}(v_1,\ldots,v_n)$ is a unimodular pseudolattice of rank $n$ generated by $z_1,\ldots,z_n$, with bilinear form given by
\[
\langle z_i, z_j \rangle_{\mathrm{Z}(v_1,\ldots,v_{n})} = \begin{cases}1 & \text{ if } i = j \\
\langle v_i, v_j \rangle_\mathrm{E} & \text{ if } i < j \\
0 & \text{ if } i > j.
\end{cases}
\]
$(z_1,\ldots,z_n)$ thus form an exceptional basis for $\rmZ(v_1,\ldots,v_n)$.

The spherical homomorphism $\zeta = \zeta_1 \oright \cdots \oright \zeta_n \colon \mathrm{Z}(v_1,\ldots,v_{n}) \rightarrow \mathrm{E}$ sending $z_i$ to $v_i$ is relative $(-1)^{0}$-CY and the twist ${T}_{\zeta} = {T}_{\zeta_{1}}  \cdots {T}_{\zeta_{n}}$.
\end{example}

Spherical homomorphisms to the elliptic curve pseudolattice give a convenient way to construct surface-like pseudolattices. The following result is a combination of the results of \cite[Propositions 3.1, 3.2, 3.3]{pdpslf}.

\begin{proposition} \label{prop:sphericalsurfacelike} Suppose that $(a,b)$ is a basis for $\rmE$ as in Definition \ref{def:ellipticcurvepseudolattice}. Let $f \colon \rmG \to \rmE$ be a relative $(-1)^0$-CY spherical homomorphism with right adjoint $r$. Then $\rmG$ is surface-like with point-like vector $\bp = r(a)$ if and only if $r(a)$ is primitive and the twist $T_f$ satisfies $T_f(a) = a$. Moreover, if this holds, then:
\begin{enumerate}
\item $r(b) \in \bp^{\perp}$ and $[r(b)] = -K_{\rmG}$, where the square brackets $[-]$ denote the class in $\NS(\rmG)$.
\item In the basis $(a,b)$, the twist $T_f$ has matrix  $\begin{psmallmatrix} 1 & -d \\ 0 & 1\end{psmallmatrix}$, where $d := q(K_{\mathrm{G}},K_{\mathrm{G}})$.
\end{enumerate}
\end{proposition}

Now we come to the central definition of \cite{pdpslf}.

\begin{definition}\cite[Definition 3.7]{pdpslf} \label{def:qdp}
Let $\rmG$ be a pseudolattice of rank $n$. A spherical homomorphism of pseudolattices $f\colon \mathrm{G} \to \mathrm{E}$ is called a \emph{quasi del Pezzo homomorphism} if there exists a basis $(a,b)$ for $\mathrm{E}$, as in Definition \ref{def:ellipticcurvepseudolattice}, so that all of the following conditions hold. 
\begin{enumerate}[(1)]
\item $\mathrm{G}$ is surface-like with point-like element $\mathbf{p} = r(a)$.
\item $f$ is relative $(-1)^0$-Calabi-Yau.
\item $\mathrm{G}$ admits an exceptional basis $e_{\bullet} = (e_1,\ldots,e_n)$ such that $f(e_i) \in \mathrm{E}$ is primitive for each $1 \leq i \leq n$.
\item $\NS(\mathrm{G})$ has signature $(1,n-3)$.
\end{enumerate}
\end{definition}

Note that conditions (1) and (2) from this definition imply that Proposition \ref{prop:sphericalsurfacelike} holds for all quasi del Pezzo homomorphisms. The motivating example behind this definition is as follows.

\begin{example} \cite[Example 3.10]{pdpslf} \label{ex:qdp} Let $V$ be a nonsingular rational surface and suppose that the anticanonical linear system $|-K_V|$ contains a smooth member $C$; such surfaces are called \emph{quasi del Pezzo surfaces}. Note that, by adjunction, $C$ is an elliptic curve. Let $i^*\colon \mathrm{K}_0^{\mathrm{num}}(\mathbf{D}(V)) \to \mathrm{K}_0^{\mathrm{num}}(\mathbf{D}(C))$ be the map between the numerical Groethendieck groups of the bounded derived categories of coherent sheaves $\mathbf{D}(V)$ and $\mathbf{D}(C)$ on $V$ and $C$, respectively, given by the derived pull-back under the inclusion $i\colon C \hookrightarrow V$. Then $i^*$ is a relative $(-1)^0$-CY spherical homomorphism, with right adjoint $i_*$ given by derived push-forward. 

Identify $\mathrm{K}_0^{\mathrm{num}}(\mathbf{D}(C)) \cong \mathrm{E}$ as above: identify $a$ with the class $[\calO_p]$ of the structure sheaf $\calO_p$ of a point $p \in C$ and $b$ with the class $[\calO_C]$ of $\calO_C$. Then $i^*$ is a quasi del Pezzo homomorphism of pseudolattices.

The point-like element $\mathbf{p} = i_*([\calO_p])$ is the class of the structure sheaf of a point in $\mathrm{K}_0^{\mathrm{num}}(\mathbf{D}(V))$. With respect to this point-like element we have an isomorphism of lattices 
\[\bp^{\perp}/\bp =: \NS(\mathrm{K}_0^{\mathrm{num}}(\mathbf{D}(V))) \cong \NS(V) \cong H^2(V;\bZ),\] which takes the canonical class in $\NS(\mathrm{K}_0^{\mathrm{num}}(\mathbf{D}(V)))$ (Definition \ref{def:canonical}) to the usual canonical class in $\NS(V)$.
\end{example}

The main result of \cite{pdpslf} is a classification of quasi del Pezzo homomorphisms. In order to state it, we first need to say what it means for two quasi del Pezzo homomorphisms to be isomorphic.

\begin{definition}\cite[Definition 3.9]{pdpslf} \label{def:qdpiso} Two quasi del Pezzo homomorphisms $f_1\colon \mathrm{G}_1 \to \mathrm{E}$ and $f_2\colon \mathrm{G}_2 \to \mathrm{E}$ are \emph{isomorphic} if
\begin{itemize}
\item there are bases $(a_1,b_1)$ and $(a_2,b_2)$ for $\mathrm{E}$, so that $(a_i,b_i)$ satisfies the conditions of Definition \ref{def:qdp} for $f_i\colon \mathrm{G}_i \to \mathrm{E}$, and an automorphism $\varphi \colon \mathrm{E} \to \mathrm{E}$ taking $(a_1,b_1)$ to $(a_2,b_2)$,
\item there is an isomorphism of pseudolattices $\psi\colon \mathrm{G}_1 \to \mathrm{G}_2$ (i.e. an isomorphism of the underlying abelian groups which preserves the bilinear form), and
\item we have $\varphi f_1 = f_2 \psi$.
\end{itemize}
It is easy to check that the last part of this definition implies that $\psi r_1 = r_2 \varphi$, where $r_i$ is the right adjoint to $f_i$.
\end{definition}

The next theorem is the main result of \cite{pdpslf}, which classifies quasi del Pezzo homomorphisms up to isomorphism. The version stated here is a slight simplification of the original.

\begin{theorem} \label{thm:qdpclassification} \cite[Theorem 3.25]{pdpslf} Let $f\colon \mathrm{G} \to \mathrm{E}$ be a quasi del Pezzo homomorphism and let $n:=\rank(\mathrm{G})$. Then $f\colon \mathrm{G} \to \mathrm{E}$ is isomorphic to exactly one of:
\begin{enumerate}
\item $\mathrm{Z}(b,3a+b,6a+b,a,\ldots,a) \to \mathrm{E}$, for any $n \geq 3$; or
\item $\mathrm{Z}(b,2a+b,2a+b,4a+b) \to \mathrm{E}$ for $n = 4$.
\end{enumerate} 
\end{theorem}

It is not difficult to show (c.f. \cite{sdp}, \cite[Proposition 0.4]{asdps}) that if $V$ is a quasi del Pezzo surface and $C \subset V$ is a smooth anticanonical divisor, then $(V,C)$ is either 
\begin{itemize}
\item a blow-up of a smooth cubic curve in $\bP^2$ in $k \geq 0$ points in almost general position (a set of points is in \emph{almost general position} if no stage of the blowing-up involves blowing up a point which lies on a rational $(-2)$-curve; infinitely near points are allowed as long as this condition is not violated),  or
\item a smooth curve of bidegree $(2,2)$ in $\bP^1 \times \bP^1$, or 
\item a smooth anticanonical curve in the Hirzebruch surface $\bF_2$.
\end{itemize}
It then follows from Theorem \ref{thm:qdpclassification}, \cite[Example 3.10]{pdpslf}, and the analysis in \cite[Section 3.3]{pdpslf}, that every quasi del Pezzo homomorphism is isomorphic to $i^*\colon \mathrm{K}_0^{\mathrm{num}}(\mathbf{D}(V)) \to \mathrm{K}_0^{\mathrm{num}}(\mathbf{D}(C))$ for $V$ a quasi del Pezzo surface and $i\colon C \hookrightarrow V$ a smooth anticanonical divisor. Explicitly, case (1) of Theorem \ref{thm:qdpclassification} corresponds to the case where $V$ is a blow-up of a smooth cubic curve in $\bP^2$ in $n-3$ points in almost general position, and case (2) of Theorem \ref{thm:qdpclassification} corresponds to the case where $V \cong \bP^1 \times \bP^1$ or $V \cong \bF_2$.

The main observation of \cite{pdpslf} is that quasi del Pezzo homomorphisms also arise naturally from genus $1$ Lefschetz fibrations over closed discs $Y \to \Delta$. In this setting they occur as boundary maps $H_2(Y,F;\bZ) \to H_1(F;\bZ)$, where $F$ is a smooth fibre over a point on the boundary of $\Delta$ and we place pseudolattice structures on $H_2(Y,F;\bZ)$ and $H_1(F;\bZ)$ via the \emph{Seifert pairing} and the usual intersection form respectively; we refer the reader to Section \ref{sec:fibrationsondiscs} for a more detailed treatment.

The fact that quasi del Pezzo homomorphisms arise in both of these settings should be thought of as a manifestation of Fano-LG mirror symmetry for quasi del Pezzo surfaces; see \cite[Section 5]{pdpslf}.

\subsection{Further properties of quasi del Pezzo homomorphisms}

In this subsection we will collect some additional properties of quasi del Pezzo homomorphisms, which do not come from \cite{ecslc,pdpslf}, that will be useful to us in the rest of the paper.

\begin{lemma}\label{lem:riproperties}
Let $f \colon \rmG \to \rmE$ be a quasi del Pezzo homomorphism of pseudolattices and let $r\colon \rmE \to \rmG$ denote the right adjoint of $f$. Let $(a,b)$ be a basis for $\rmE$ as in Definition \ref{def:qdp}. Then $r$ is injective and satisfies $fr(a) = 0$.
\end{lemma}
\begin{proof}
Injectivity of $r$ is an easy consequence of the explicit classification of quasi del Pezzo homomorphisms (Theorem \ref{thm:qdpclassification}). Moreover, as $f$ is a quasi del Pezzo homomorphism, Proposition \ref{prop:sphericalsurfacelike} gives $T_{f}(a) = a$. But $T_{f}(a) = a - fr(a)$ by definition, so we must have $fr(a) = 0$. 
\end{proof}

We next collect some information about the N\'{e}ron-Severi lattice associated to a quasi del Pezzo homomorphism.

\begin{lemma} \label{lem:NSproperties} Let $f\colon \mathrm{G} \to \mathrm{E}$ be a quasi del Pezzo homomorphism and let $n:=\rank(\mathrm{G})$.
\begin{enumerate}
\item If $f\colon \mathrm{G} \to \mathrm{E}$ is isomorphic to case (1) in Theorem \ref{thm:qdpclassification}, then $\NS(\rmG)$ is the odd unimodular lattice $\rmI_{1,n-3}$ of signature $(1,n-3)$. The canonical class $K_{\rmG}$ has $q(K_{\rmG},K_{\rmG}) = 12-n$ and its orthogonal complement $K_{\rmG}^{\perp}$ in $\NS(\rmG)$ is given in Table \ref{tab:NSlattices}. This orthogonal complement is negative definite if $n < 12$, degenerate if $n = 12$, and indefinite of signature $(1,n-4)$ if $n > 12$. If $n \notin \{3,12\}$, the discriminant group $\disc(K_{\rmG}^{\perp}) \cong \bZ\,/\,|12-n|\bZ$. The quadratic form on $\disc(K_{\rmG}^{\perp})$ is generated by an element $e$ with $e^2 = \frac{1}{n-12}$, with values taken modulo $\bZ$ if $n$ is odd and modulo $2\bZ$ if $n$ is even. 
\item If $f\colon \mathrm{G} \to \mathrm{E}$ is isomorphic to case (2) in Theorem \ref{thm:qdpclassification}, then $\NS(\rmG)$ is the even unimodular lattice $H = \rmI\rmI_{1,1}$ of signature $(1,1)$ (the hyperbolic plane).  The canonical class $K_{\rmG}$ has $q(K_{\rmG},K_{\rmG}) = 8$ and the orthogonal complement $K_{\rmG}^{\perp}$ is the root lattice $A_1$, which is negative definite with $\disc(K_{\rmG}^{\perp}) \cong \bZ/2\bZ$. The quadratic form on $\disc(K_{\rmG}^{\perp})$ is generated by an element $e$ with $e^2 = -\frac{1}{2}$, with values taken modulo $2\bZ$.
\end{enumerate}
\end{lemma}

\begin{table}
\begin{tabular}{|c|c|c|c|c|c|c|c|}
\hline
$n$ & $3$ & $4$ & $5$ & $6$ & $7$ & $8$ & $\geq 9$ \\
\hline
\rule{0pt}{3ex} $K_{\rmG}^{\perp}$ & $\{0\}$ & $(-8)$ & $\begin{psmallmatrix} -2 & 1 \\ 1 & -4\end{psmallmatrix}$ & $A_2 \oplus A_1$ & $A_4$ & $D_5$ & $E_{n-3}$\\
\hline
\end{tabular}
\caption{N\'{e}ron-Severi lattices in case (1) of Lemma \ref{lem:NSproperties}. Here $A_n$, $D_n$, $E_6$, $E_7$, and $E_8$ are the usual negative definite root lattices, and we extend the definition of the $E_n$ series of lattices to $n \geq 9$ in the obvious way.}
\label{tab:NSlattices}
\end{table}

\begin{proof} This follows from the isomorphism between $\NS(\rmG)$ and $\NS(V)$ for a quasi del Pezzo surface $V$, along with the well-known lattice theory for quasi del Pezzo surfaces (see, for example, \cite[Chapter 8]{cagmv}).
\end{proof}

Now we come to the main result of this section, which characterises isomorphisms between quasi del Pezzo homomorphisms in terms of their action on the N\'{e}ron-Severi lattices.

\begin{proposition} \label{prop:qdpisomorphism} Suppose that $f_1\colon \rmG_1 \to \rmE$ and $f_2\colon \rmG_2 \to \rmE$ are two quasi del Pezzo homomorphisms. 

If $f_1\colon \rmG_1 \to \rmE$ and $f_2\colon \rmG_2 \to \rmE$ are isomorphic, in the sense of Definition \ref{def:qdpiso}, then the isomorphism $\psi \colon \rmG_1 \to \rmG_2$ induces an isometry $\hat{\psi} \colon \NS(\rmG_1) \to \NS(\rmG_2)$ such that $\hat{\psi}(K_{\rmG_1}) = K_{\rmG_2}$.

Conversely, if $\hat{\psi} \colon \NS(\rmG_1) \to \NS(\rmG_2)$ is an isometry such that $\hat{\psi}(K_{\rmG_1}) = K_{\rmG_2}$, then there exists an isomorphism between $f_1\colon \rmG_1 \to \rmE$ and $f_2\colon \rmG_2 \to \rmE$ that induces $\hat{\psi}$.
\end{proposition}

\begin{proof}
Let $r_i$ denote the right adjoint to $f_i$. Suppose first that there exists an isomorphism between $f_1\colon \rmG_1 \to \rmE$ and $f_2\colon \rmG_2 \to \rmE$, in the sense of Definition \ref{def:qdpiso}. Then there are bases $(a_1,b_1)$ and $(a_2,b_2)$  for $\mathrm{E}$ so that $(a_i,b_i)$ satisfies the conditions of Definition \ref{def:qdp} for $f_i\colon \mathrm{G}_i \to \mathrm{E}$, and an automorphism $\varphi \colon \mathrm{E} \to \mathrm{E}$ taking $(a_1,b_1)$ to $(a_2,b_2)$, along with an isomorphism $\psi \colon \rmG_1 \to \rmG_2$ which commutes with $\varphi$.

As $\psi r_1 = r_2 \varphi$, we have $\psi(r_1(a_1)) = r_2(\varphi(a_1)) = r_2(a_2)$, so $\psi$ maps the point-like vector in $\rmG_1$ to the point-like vector in $\rmG_2$. It follows that $\psi$ induces an isometry $\hat{\psi} \colon \NS(\rmG_1) \to \NS(\rmG_2)$. Moreover, as $\psi(r_1(b_1)) = r_2(\varphi(b_1)) = r_2(b_2)$ and $K_{\rmG_i} = -[r_i(b_i)]$ for each $i \in \{1,2\}$, we have $\hat{\psi}(K_{\rmG_1}) = K_{\rmG_2}$.

Next, suppose that $\hat{\psi} \colon \NS(\rmG_1) \to \NS(\rmG_2)$ is an isometry such that $\hat{\psi}(K_{\rmG_1}) = K_{\rmG_2}$. 

By the classification of quasi del Pezzo homomorphisms, the homomorphism $f_j \colon \rmG_j \to \rmE$ is isomorphic to $i^*\colon \mathrm{K}_0^{\mathrm{num}}(\mathbf{D}(V_j)) \to \mathrm{K}_0^{\mathrm{num}}(\mathbf{D}(C_j))$ for $V_j$ a quasi del Pezzo surface and $i\colon C_j \hookrightarrow V_j$ a smooth anticanonical divisor. We may identify $\mathrm{K}_0^{\mathrm{num}}(\mathbf{D}(C_j)) \cong \mathrm{E}$ in the usual way: identify $a_j$ with the class of a point $p \in C_j$ and $b_j$ with the class of $\calO_{C_j}$. Moreover, by \cite[Example 3.5]{ecslc}, there are isomorphisms $\mathrm{K}_0^{\mathrm{num}}(\mathbf{D}(V_j))\cong \bZ \oplus \NS(V_j) \oplus \bZ$ given by $v \mapsto (\rank(v),c_1(v),\chi(v))$,  i.e. the rank, first Chern class, and Euler characteristic of an element $v \in \mathrm{K}_0^{\mathrm{num}}(\mathbf{D}(V_j))$.

Choose bases for $\NS(\rmG_j)$ as free $\bZ$-modules for each $j \in \{1,2\}$. Let $Q_j$ denote the matrix of the intersection form and let $K_j$ denote the column vector representing $K_{\rmG_j}$ with respect to this basis. Then in the realisation $\rmG_j \cong \bZ \oplus \NS(V_j) \oplus \bZ$ given above, one may compute that the bilinear form on $\rmG_j$ is given explicitly by
\[\begin{pmatrix}
-1 & 0 & 1 \\ Q_jK_j & -Q_j & 0 \\ 1 & 0 & 0
\end{pmatrix}\]
and the map $f_j \colon \rmG_j \to \rmE$ is given explicitly by
\[\begin{pmatrix}
0 & -K_j^TQ_j & 0 \\ 1 & 0 & 0
\end{pmatrix}\]
with respect to the basis $a_j = \begin{psmallmatrix} 1\\0 \end{psmallmatrix}$ and $b_j = \begin{psmallmatrix} 0\\1 \end{psmallmatrix}$ of $\rmE$.

Let $P$ be the matrix of the isometry $\NS(\rmG_1) \to \NS(\rmG_2)$ with respect to the chosen bases of these two lattices. Then $P^TQ_2P = Q_1$ and $PK_1 = K_2$. Define a map $\psi \colon \rmG_1 \to \rmG_2$ by the matrix
\[\begin{pmatrix}
1 & 0 & 0 \\ 0 & P & 0   \\ 0 & 0 & 1
\end{pmatrix}\]
and let $\varphi$ be the automorphism $\varphi \colon \mathrm{E} \to \mathrm{E}$ taking $(a_1,b_1)$ to $(a_2,b_2)$. Then it is a straightforward exercise in linear algebra to show that $\psi$ defines an isomorphism of pseudolattices $\rmG_1 \to \rmG_2$ and that $\varphi f_1 = f_2 \psi$, so we obtain an isomorphism of quasi del Pezzo homomorphisms between  $f_1\colon \rmG_1 \to \rmE$ and $f_2\colon \rmG_2 \to \rmE$.
\end{proof}

\subsection{From pseudolattices to lattices}\label{sec:pseudolattice}

With this background in place, we will explore what happens when one glues pairs of quasi del Pezzo homomorphisms. In particular, we will show that if this is done in the correct way, it gives rise to an exact sequence of lattices.

Let $f_i \colon \mathrm{G}_i \to \mathrm{E}$, for $i \in \{1,2\}$ be two quasi del Pezzo homomorphisms of pseudolattices, as in Definition \ref{def:qdp}, and suppose further that the basis $(a,b)$ for $\mathrm{E}$ appearing in this definition can be taken to be the same for both $i \in \{1,2\}$. Let $r_i$ denote the right adjoint of $f_i$. Finally, assume that $q(K_{\mathrm{G}_1},K_{\mathrm{G}_1}) = - q(K_{\mathrm{G}_2},K_{\mathrm{G}_2})$. 

The homomorphism $(-f_2) \colon \mathrm{G}_2 \to \mathrm{E}$ is clearly spherical, with right adjoint $(-r_2)$ and twist and cotwist $T_{(-f_2)} = T_{f_2}$ and $C_{(-f_2)} = C_{f_2}$ respectively. From this it follows immediately that $(-f_2)$ is also relatively $(-1)^0$-CY.

Now define $\mathrm{G} := \rmG_1 \oright_{\mathrm{E}} \mathrm{G}_2$ glued along the homomorphisms $f_1$ and $(-f_2)$, and let $f\colon\mathrm{G} \to \mathrm{E}$ be the natural homomorphism
\begin{align*}&f\colon \mathrm{G}_1 \oright_{\mathrm{E}} \mathrm{G}_2 \longrightarrow \mathrm{E}.\\
&f(u_1,u_2)  = f_1(u_1)-f_2(u_2).
\end{align*}
By Lemma \ref{lem:sumisCY}, $f$ is a relative $(-1)^0$-CY spherical homomorphism. We denote the right adjoint of $f$ by $r$. By the proof of  \cite[Proposition 2.23]{pdpslf}, the map $r$ is given explicitly by
\begin{equation} \label{eq:r} r(v) = (r_1T_{f_2}(v),-r_2(v)) \in \rmG_1 \oright_{\mathrm{E}} \mathrm{G}_2\end{equation}
for any $v \in \rmE$.  Let $\mathrm{K} \subset \mathrm{G}$ denote the kernel of $f$.

\begin{lemma} The bilinear form $\langle \cdot , \cdot \rangle_{\mathrm{G}}$ is symmetric on $\mathrm{K}$, so $\mathrm{K}$ is a lattice \textup{(}in the classical sense\textup{)}.
\end{lemma}
\begin{proof} Let $S_{\mathrm{G}}$ denote the Serre operator of $\mathrm{G}$; as $f$ is relatively $(-1)^0$-CY, we have $S_{\mathrm{G}} = C_f := \mathrm{id}_{\mathrm{G}} - rf$. Then for any $u,v \in \mathrm{K}$, we have
\[\langle u,v \rangle_{\mathrm{G}} = \langle v, S_{\mathrm{G}}(u)\rangle_{\mathrm{G}} = \langle v, u-rf(u) \rangle_{\mathrm{G}} = \langle v,u\rangle_{\mathrm{G}}\]
as $f(u) = 0$ by assumption. 
\end{proof}

\begin{lemma}\label{lem:r(E)isotrivial} Let $u \in \mathrm{E}$. Then $r(u) \in \mathrm{K}$ and $\langle v, r(u) \rangle_{\mathrm{G}} = 0$ for all $v \in \mathrm{K}$, so $\im(r)$ is a totally degenerate sublattice of $\rmK$.
\end{lemma}
\begin{proof} By Definition \ref{def:gluing}, the twist $T_f = T_{f_1}\cdot T_{-f_2} = T_{f_1}\cdot T_{f_2}$. Moreover, by Proposition \ref{prop:sphericalsurfacelike}, in the basis $(a,b)$ for $\mathrm{E}$ we have $T_{f_i} = \begin{psmallmatrix} 1 & -d_i \\ 0 & 1 \end{psmallmatrix}$, where $d_i := q(K_{\mathrm{G}_i},K_{\mathrm{G}_i})$. Thus, since $d_1 = -d_2$ by assumption, we have $T_f = \mathrm{id}_{\mathrm{E}}$. But $T_f = \mathrm{id}_{\mathrm{E}} - fr$ by definition, so $fr(u) = 0$ for all $u \in \mathrm{E}$ and hence $r(u) \in \mathrm{K}$.

To prove the second part of the statement, note that by adjunction $\langle v, r(u) \rangle_{\mathrm{G}} = \langle f(v), u \rangle_{\mathrm{E}} = 0$,
as $f(v) = 0$ by assumption.
\end{proof}

Let $\overline{\mathrm{E}} \subset \mathrm{K}$ denote the saturation of the $\bZ$-submodule $\im(r)$, equipped with the bilinear form induced by $\langle \cdot , \cdot \rangle_{\mathrm{G}}$; by the results above we see that $\overline{\mathrm{E}}$ is isomorphic to $\bZ^2$ with the zero form. By Lemma \ref{lem:r(E)isotrivial}, we may  take the quotient $\mathrm{K}/\overline{\mathrm{E}}$ to obtain a new lattice, which we denote by $\rmM$. We thus have an exact sequence of lattices (in the classical sense)
\begin{equation} \label{eq:ses} 0 \longrightarrow \overline{\mathrm{E}} \longrightarrow \mathrm{K} \longrightarrow \rmM \longrightarrow 0.\end{equation}

\begin{remark}
Note that if $(u_1,u_2)$ and $(v_1,v_2)$ are elements of $\rmK$, for $u_i,v_i \in \rmG_i$, then as $\rmK = \ker(f)$ we have $f_1(u_1) = f_2(u_2)$ and $f_1(v_1) = f_2(v_2)$. Consequently
\begin{align*}
\langle (u_1,u_2),\,(v_1,v_2)\rangle_{\rmG_1\oright_{\rmE}\rmG_2} &= \langle u_1,v_1\rangle_{\rmG_1} - \langle f_1(u_1),f_2(v_2)\rangle_{\rmE} + \langle u_2,v_2\rangle_{\rmG_2}\\
&=\langle u_2,v_2\rangle_{\rmG_2} - \langle f_2(u_2),f_1(v_1)\rangle_{\rmE} + \langle u_1,v_1\rangle_{\rmG_1}\\
&= \langle (u_2,u_1),\,(v_2,v_1)\rangle_{\rmG_2\oright_{\rmE}\rmG_1}
\end{align*}
It follows that switching the order of the factors in the definition of $\rmG$ does not change the lattice $\rmK$: we obtain the same $\rmK$ if we define $\rmG$ to be $\rmG_1\oright_{\rmE}\rmG_2$ or $\rmG_2\oright_{\rmE}\rmG_1$. From this it is easy to see that the remaining results of this section, particularly the exact sequence \eqref{eq:ses}, are also independent of the ordering of the factors in $\rmG$.
\end{remark}

As a consequence of this remark, we may freely switch the labels on $\rmG_1$ and $\rmG_2$ without affecting the subsequent development. For concreteness, we will henceforth assume that labels have been chosen so that $q(K_{\mathrm{G}_1},K_{\mathrm{G}_1}) \geq 0$.

\begin{definition}
Define the \emph{degree} of $\rmG$ to be $\deg(\rmG) = q(K_{\rmG_1},K_{\rmG_1})$, whenever $q(K_{\rmG_1},K_{\rmG_1}) \neq 8$. If $q(K_{\rmG_1},K_{\rmG_1}) = 8$ then, by Theorem \ref{thm:qdpclassification} and Lemma \ref{lem:NSproperties}, there are two possibilities for the quasi del Pezzo homomorphism $f_1 \colon\rmG_1 \to \rmE$: we say that $\deg(\rmG) = 8$ if $\NS(\rmG_1) \cong \rmI_{1,1}$ (case (1) of Theorem \ref{thm:qdpclassification}) and $\deg(\rmG) = 8'$ if $\NS(\rmG_1) \cong \rmI\rmI_{1,1}$ (case (2) of Theorem \ref{thm:qdpclassification}).
\end{definition}

As a consequence of the classification of quasi del Pezzo homomorphisms (Theorem \ref{thm:qdpclassification}), we note that $f\colon \rmG \to \rmE$ is determined up to isomorphism by $\deg(\rmG)$, and that $0 \leq \deg(\rmG) \leq 9$. 

\subsection{Point-like vectors and the N\'{e}ron-Severi lattice}\label{sec:pointsandNS}

Next we consider the vectors $(r_1(a),0)$ and $(0,r_2(a))$ in $\rmG_1 \oright_{\mathrm{E}} \mathrm{G}_2 = \rmG$. Define $\Psi$ to be the $\bZ$-submodule of $\rmG$ generated by these two elements.

\begin{lemma} \label{lem:psiproperties} $\Psi$ has the following properties.
\begin{enumerate}
\item $\Psi$ is primitive in $\rmG$.
\item If $p \in \Psi$, then $\langle p,v\rangle_{\rmG} = \langle v, p\rangle_{\rmG}$ for any $v \in \rmG$.
\item $\Psi$ is contained in $\rmK$.
\item The restriction of the bilinear form $\langle \cdot,\cdot\rangle_{\rmG}$ to $\Psi$ is identically zero.
\item $\Psi \cap \overline{\rmE}$ is generated by the primitive element 
\[r(a) = (r_1(a),-r_2(a)) \in \rmG_1 \oright_{\mathrm{E}} \mathrm{G}_2 = \rmG.\]
\end{enumerate}
\end{lemma}
\begin{proof}
To prove (1), note that $r_i(a)$ is primitive in $\rmG_i$ for each $i \in \{1,2\}$, by Proposition \ref{prop:sphericalsurfacelike}. As $\rmG$ is isomorphic to $\rmG_1 \oplus \rmG_2$ as a $\bZ$-module, it follows that $\Psi$ is primitive in $\rmG$.

To prove (2), it suffices to show that $\langle r_i(a),v\rangle_{\rmG} = \langle v, r_i(a)\rangle_{\rmG}$ for each $i \in \{1,2\}$ and any $v \in \rmG$. So let $v \in \rmG$. Writing $v = (v_1,v_2)$ for $v_i \in \rmG_i$, by Definition \ref{def:gluing} we have
\begin{align*}
\langle r_1(a), v\rangle_{\rmG} &= \langle r_1(a),v_1\rangle_{\rmG_1} + \langle f_1r_1(a),-f_2(v_2) \rangle_{\rmE},\\
\langle v, r_1(a)\rangle_{\rmG} &= \langle v_1,r_1(a)\rangle_{\rmG_1}.
\end{align*}
But $f_1r_1(a) = 0$ by Lemma \ref{lem:riproperties}, so $\langle r_1(a), v\rangle_{\rmG} = \langle r_1(a),v_1\rangle_{\rmG_1}$.
To complete the proof, we use the fact that, by Definition \ref{def:qdp}, $r_1(a)$ is point-like in $\rmG_1$, so $\langle r_1(a),v_1\rangle_{\rmG_1} = \langle v_1,r_1(a)\rangle_{\rmG_1}$;
we conclude that $\langle r_1(a),v\rangle_{\rmG} = \langle v, r_1(a)\rangle_{\rmG}$ as required. The proof for $r_2$ is completely analogous.

To prove (3), we have $f(r_1(a),0) = f_1r_1(a)$ and $f(0,r_2(a)) = -f_2r_2(a)$, and both of these are equal to $0$ by Lemma \ref{lem:riproperties}. So $(r_1(a),0), (0,r_2(a)) \in \ker(f) = \rmK$, which implies $\Psi \subset \rmK$.

To prove (4) we note that, as $r_i(a)$ is point-like in $\rmG_i$ for $i \in \{1,2\}$, we must have $\langle r_i(a),r_i(a)\rangle_{\rmG_i} = 0$. So, by Definition \ref{def:gluing},
\[\langle (r_1(a),0), (r_1(a),0) \rangle_{\rmG} = \langle (0,r_2(a)), (0,r_2(a)) \rangle_{\rmG} = 0\]
Moreover, $\langle (0,r_2(a)),(r_1(a),0) \rangle_{\rmG} = 0$ by definition and
\[\langle (r_1(a),0), (0,r_2(a)) \rangle_{\rmG} = \langle f_1r_1(a),-f_2r_2(a) \rangle_{\rmE} = 0,\]
as $f_ir_i(a) = 0$ for each $i \in \{1,2\}$ by Lemma \ref{lem:riproperties}. We conclude that the restriction of the bilinear form $\langle \cdot,\cdot\rangle_{\rmG}$ to $\Psi$ is identically zero.

Finally, to prove (5) we observe that, by Equation \eqref{eq:r}, the map $r$ is given by
\[r(v) = (r_1T_{f_2}(v),-r_2(v)) = (r_1(v)-r_1f_2r_2(v),-r_2(v)).\]
Applying this to $a \in \rmE$, we see that $r(a) = (r_1(a),-r_2(a))$, where we have used the fact that $f_2r_2(a) = 0$. Thus $(r_1(a),-r_2(a)) \in \im(r) \subset \overline{\rmE}$, and this element is primitive by the same argument that we used to prove (1).

To complete the proof of (5), as both $\Psi$ and $\overline{\rmE}$ have rank $2$, it suffices to find an element of $\Psi$ that is not in $\overline{\rmE}$. Consider $(0,-r_2(a)) \in \Psi$. Assume for a contradiction that some multiple of this is equal to $r(v)$, for $v \in \rmE$. By the formula above, this can only be true if $r_2(v) = kr_2(a)$, for $k \in \bZ \setminus \{0\}$. But $r_2$ is linear and injective by Lemma \ref{lem:riproperties}, so we must have $v = ka$. It then follows from our computation above that $r(ka) = (kr_1(a),-kr_2(a))$; this is a contradiction.
\end{proof}

As a consequence of Lemma \ref{lem:psiproperties}(2), we see that the left and right orthogonal complements of $\Psi$ in $\rmG$ agree. We denote this orthogonal complement by $\Psi^{\perp}_{\rmG}$. The following lemma describes its properties.

\begin{lemma}\label{lem:psiperpG}
Equipping $\Psi_{\rmG}^{\perp}$ with the bilinear form induced from $-\langle \cdot,\cdot\rangle_{\rmG}$ \textup{(}note the negative sign\textup{)}, there is a natural isomorphism of lattices
\[\Psi_{\rmG}^{\perp}/\Psi \cong \NS(\rmG_1) \oplus \NS(\rmG_2),\]
where $\NS(\rmG_i)$ is equipped with its usual bilinear form $q(\cdot,\cdot)$.
\end{lemma}
\begin{proof}
Recall that, for each $i \in \{1,2\}$, the point-like element in $\rmG_i$ is given by $r_i(a)$, so $\NS(\rmG_i) := (r_i(a))^{\perp}_{\rmG_i}/r_i(a)$.

We begin by showing that $(r_1(a))^{\perp}_{\rmG_1} \oplus (r_2(a))^{\perp}_{\rmG_2}$ is a sublattice of $\rmG$ in the obvious way. To do this we need to show that the bilinear forms induced on $(r_1(a))^{\perp}_{\rmG_1} \oplus (r_2(a))^{\perp}_{\rmG_2}$ from those on $\rmG_1$ and $\rmG_2$, respectively from $\rmG$, agree.

Let $u,v \in (r_1(a))^{\perp}_{\rmG_1} \oplus (r_2(a))^{\perp}_{\rmG_2}$ and write $u=(u_1,u_2)$, $v = (v_1,v_2)$. Then
\[\langle u,v \rangle_{\rmG} = \langle u_1,v_1\rangle_{\rmG_1} + \langle u_2,v_2\rangle_{\rmG_2} + \langle f_1(u_1), -f_2(v_2)\rangle_{\rmE}.\]
To show that this agrees with the bilinear form induced from those on $\rmG_1$ and $\rmG_2$, we need to show that $\langle f_1(u_1), -f_2(v_2)\rangle_{\rmE} = 0$. As $u_1 \in (r_1(a))^{\perp}_{\rmG_1}$, we have $0 = \langle u_1,r_1(a) \rangle_{\rmG_1} = \langle f_1(u_1),a \rangle_{\rmE}$. By the definition of the bilinear form on $\rmE$, it follows that $f_1(u_1)$ must be an integer multiple of $a$. A similar argument shows that $f_2(v_2)$ must also be an integer multiple of $a$, so $\langle f_1(u_1), -f_2(v_2)\rangle_{\rmE} = c\, \langle a,a \rangle_{\rmE} = 0$, for some $c \in \bZ$, as required.

Next we show that $\Psi_{\rmG}^{\perp} = (r_1(a))^{\perp}_{\rmG_1} \oplus (r_2(a))^{\perp}_{\rmG_2}$ as subsets of $\rmG$. For any $v = (v_1,v_2) \in \rmG$, we have
\begin{align*}
\langle (v_1,v_2),\, (r_1(a),0)\rangle_{\rmG} &= \langle v_1,r_1(a) \rangle_{\rmG_1},\\
\langle (0,r_2(a)),\, (v_1,v_2) \rangle_{\rmG} &= \langle r_2(a),v_2 \rangle_{\rmG_2},
\end{align*}
so $v \in  \Psi_{\rmG}^{\perp}$ if and only if $v_1 \in (r_1(a))^{\perp}_{\rmG_1}$ and $v_2 \in (r_2(a))^{\perp}_{\rmG_2}$.

We thus have an equality of sublattices $\Psi^{\perp}_{\rmG} = (r_1(a))^{\perp}_{\rmG_1} \oplus (r_2(a))^{\perp}_{\rmG_2}$ of $\rmG$. We conclude by quotienting by $\Psi$ and noting that the bilinear form $q(\cdot,\cdot)$ on $\NS(\rmG_i)$ is induced by $-\langle \cdot,\cdot\rangle_{\rmG_i}$.
\end{proof}

As a consequence of Lemma \ref{lem:psiproperties}, we also see that $(r_1(a),0)$ and $(0,r_2(a))$ both descend to the same primitive element $\bp \in \rmM$. 

\begin{proposition} \label{prop:NS} $\rmM$ is a surface-like pseudolattice with point-like element $\bp$. The N\'{e}ron-Severi group of $\rmM$, defined by $\NS(\rmM) := \bp^{\perp}/\bp$, fits into the following exact sequence
\[0 \longrightarrow  \bZ \zeta \longrightarrow \Psi^{\perp}_{\rmK}/\Psi \longrightarrow \NS(\rmM) \longrightarrow 0,\] 
where $\Psi^{\perp}_{\rmK}$ denotes the orthogonal complement of $\Psi$ in $\rmK$ and $\zeta$ is a primitive generator of $\overline{\rmE}/r(a)$. 

If $\Psi^{\perp}_{\rmK}/\Psi$ is equipped with the bilinear form induced by the natural inclusion 
\[\Psi^{\perp}_{\rmK}/\Psi \subset \Psi_{\rmG}^{\perp}/\Psi \cong \NS(\rmG_1) \oplus \NS(\rmG_2),\]
$\NS(\rmM)$ is equipped with its usual bilinear form $q(\cdot,\cdot)$, and $\zeta$ is isotropic \textup{(}$\zeta.\zeta = 0$\textup{)}, then the above sequence becomes an exact sequence of lattices.
\end{proposition}

\begin{proof} Note first that $\zeta$ is well-defined as $r(a)$ is primitive in $\overline{\rmE}$, so $\overline{\rmE}/r(a)$ is a torsion-free $\bZ$-module of rank $1$.

The element $\bp$ is isotropic in $\rmM$, as the bilinear form on $\rmM$ is induced from that on $\rmK$ and $\Psi$ is isotropic in $\rmK$. So, by \cite[Lemma 3.3(a)]{ecslc}, $\rmM$ is surface-like with point-like vector $\bp$.

To prove that the given sequence is exact, we first observe that, by Lemma \ref{lem:psiproperties}, we have an exact sequence
\[0 \longrightarrow \bZ r(a) \longrightarrow \Psi \longrightarrow \bZ \bp \longrightarrow 0,\]
and the restriction of the bilinear forms to the terms of this sequence are all identically zero. Moreover, as $\overline{\rmE}$ is orthogonal to $\Psi$ (by Lemma \ref{lem:r(E)isotrivial}), we have $\overline{\rmE} \subset \Psi^{\perp}_{\rmK}$, so we obtain a second exact sequence
\[0 \longrightarrow \overline{\rmE} \longrightarrow \Psi^{\perp}_{\rmK} \longrightarrow \bp^{\perp} \longrightarrow 0,\]
and this is an exact sequence of lattices if we equip its terms with the induced bilinear forms from $\overline{\rmE}$, $\rmG$, and $\rmM$. Taking the quotient of the second sequence by the first, using the snake lemma, and flipping the signs on the bilinear forms gives the result. 
\end{proof}

To conclude this section, the following lemma describes the image of $\bZ \zeta$ inside $\Psi^{\perp}_{\rmK}/\Psi$.

\begin{lemma}\label{lem:zetaisotrivial}  The inclusion $\bZ \zeta \hookrightarrow \Psi^{\perp}_{\rmK}/\Psi$ realises $\bZ \zeta$ as a totally degenerate rank $1$ sublattice of $\Psi^{\perp}_{\rmK}/\Psi \subset \NS(\rmG_1) \oplus \NS(\rmG_2)$. This sublattice contains the nonzero element $(-K_{\rmG_1}, K_{\rmG_2})$.
\end{lemma}
\begin{proof} The first statement follows immediately from Lemma \ref{lem:r(E)isotrivial}. For the second statement, note that $r(b) \in \overline{\rmE} \subset \Psi_{\rmK}^{\perp}$ and that, by Equation \eqref{eq:r}, we have $r(b) = (r_1T_{f_2}(b),-r_2(b))$.
By Proposition \ref{prop:sphericalsurfacelike}, we have $T_{f_2}(b) = -d_2 a + b$, where $d_2 := q(K_{\mathrm{G}_2},K_{\mathrm{G}_2}) \in \bZ$, so $r(b) = (-d_2 r_1(a) + r_1(b),-r_2(b))$,
which is congruent to $(r_1(b),-r_2(b))$ modulo $\Psi$. The result follows by descending to the quotient and noting that, by Proposition \ref{prop:sphericalsurfacelike}, we have $K_{\rmG_i} = -[r_i(b)]$, where the square brackets $[- ]$ denote the class in $\NS(\rmG_i)$.
\end{proof}

\subsection{Lattice polarisations}\label{sec:latticespolarisingpseudo}

We define various notions of lattice polarisation on our N\'{e}ron-Severi lattices.

\begin{definition} \label{def:Lpol} \begin{enumerate}
\item A \emph{lattice polarisation} on $\NS(\rmM)$ is defined to be a primitive sublattice $L \subset \NS(\rmM)$.
\item A \emph{lifted polarisation} on $\Psi^{\perp}_{\rmK}/\Psi$ is a primitive sublattice $\hat{L} \subset \Psi^{\perp}_{\rmK}/\Psi$ such that $\zeta \in \hat{L}$.
\item Given a lifted polarisation $\hat{L} \subset \Psi^{\perp}_{\rmK}/\Psi$, the \emph{intersection polarisation} on $\NS(\rmG_i)$ is defined to be $L_i := \hat{L} \cap \NS(\rmG_i)$. Note that $L_i$ is primitive in $\NS(\rmG_i)$.
\end{enumerate}
\end{definition}

Our first result shows that lattice polarisations and lifted polarisations are two sides of the same coin.

\begin{lemma} \label{lem:liftinglattices} If $L \subset \NS(\rmM)$ is a lattice polarisation, then the preimage of $L$ under the map $\Psi^{\perp}_{\rmK}/\Psi \to \NS(\rmM)$ is a lifted polarisation. Conversely, if $\hat{L} \subset \Psi^{\perp}_{\rmK}/\Psi$ is a lifted polarisation, then the image of $\hat{L}$ under the map $\Psi^{\perp}_{\rmK}/\Psi \to \NS(\rmM)$ is a lattice polarisation.
\end{lemma}
\begin{proof} The first statement is obvious from the definitions. For the converse, we need to check that the image of $\hat{L}$ is primitive in $\NS(\rmM)$. Let $\varphi$ denote the map $\varphi\colon \Psi^{\perp}_{\rmK}/\Psi \to \NS(\rmM)$. Suppose that $v \in \NS(\rmM)$ is a primitive element such that $mv \in \varphi(\hat{L})$, for some integer $m > 1$. As $\varphi$ is surjective, we may find $w \in \Psi^{\perp}_{\rmK}/\Psi$ such that $\varphi(w) = v$. Then $mw + n\zeta \in \hat{L}$ for some $n \in \bZ$, as $\varphi(mw) = mv \in \varphi(\hat{L})$ and $\zeta$ generates the kernel of $\varphi$. Now, as $\zeta \in \hat{L}$, it follows that $mw \in \hat{L}$. But $\hat{L}$ is primitive, so we must also have $w \in \hat{L}$, and hence $\varphi(w) = v \in \varphi(\hat{L})$ and $\varphi(\hat{L})$ is primitive in $\NS(\rmM)$.
\end{proof}

Next we look at the properties of the intersection polarisation.

\begin{lemma} \label{lem:Kperp} We have
\[\Psi^{\perp}_{\rmK}/\Psi \cap \NS(\rmG_i) = K_{\rmG_i}^{\perp}.\]
In particular, if $\hat{L} \subset \Psi^{\perp}_{\rmK}/\Psi$ is a lifted polarisation, then the corresponding intersection polarisation on $\NS(\rmG_i)$ satisfies ${L}_i \subset K_{\rmG_i}^{\perp}$
\end{lemma}
\begin{proof}
Begin by noting that, by Proposition \ref{prop:sphericalsurfacelike}, we have $K_{\rmG_i} = -[r_i(b)]$, where the square brackets $[- ]$ denote the class in $\NS(\rmG_i)$.

Now, if $v \in \NS(\rmG_i) =  (r_i(a))^{\perp}_{\rmG_i} / r_i(a)$ is any class, let $\hat{v} \in (r_i(a))^{\perp}_{\rmG_i}$ denote a representative of $v$. From the proof of Lemma \ref{lem:psiperpG} we have $\Psi^{\perp}_{\rmG} = (r_1(a))^{\perp}_{\rmG_1} \oplus (r_2(a))^{\perp}_{\rmG_2}$, so $\hat{v} \in \Psi^{\perp}_{\rmG} \subset \rmG$. Therefore, by definition of $f$, we have $f(\hat{v}) = (-1)^{i-1} f_i(\hat{v})$. Thus, in particular, we find
\[q(v,K_{\rmG_i}) = \langle \hat{v},r_i(b) \rangle_{\rmG_i} = \langle f_i(\hat{v}),b\rangle_{\rmE} = (-1)^{i-1} \langle f(\hat{v}),b\rangle_{\rmE}.\]

Now suppose that $v \in \Psi^{\perp}_{\rmK}/\Psi \cap \NS(\rmG_i)$. Then $\hat{v} \in \Psi^{\perp}_{\rmK} \subset \rmK$, so $q(v,K_{\rmG_i}) = (-1)^{i-1} \langle f(\hat{v}),b\rangle_{\rmE} = 0$,
as $\rmK = \ker(f)$. So $\Psi^{\perp}_{\rmK}/\Psi \cap \NS(\rmG_i) \subset K_{\rmG_i}^{\perp}$.

Conversely, suppose that $v \in K_{\rmG_i}^{\perp}$. Then
\begin{align*}
\langle f(\hat{v}),b\rangle_{\rmE} &= (-1)^{i-1} q(v,K_{\rmG_i}) = 0,\\
\langle f(\hat{v}),a \rangle_{\rmE} &= (-1)^{i-1}\langle f_i(\hat{v}), a\rangle_{\rmE} = (-1)^{i-1}\langle \hat{v},r_i(a) \rangle_{\rmG_i} = 0
\end{align*}
as $\hat{v} \in (r_i(a))^{\perp}_{\rmG_i}$. Therefore we must have $f(\hat{v}) = 0$, so $\hat{v} \in \rmK$. As we also have $\hat{v} \in \Psi^{\perp}_{\rmG}$, it follows that $\hat{v} \in \Psi^{\perp}_{\rmK}$, and so $v \in \Psi^{\perp}_{\rmK}/\Psi$. Thus $K_{\rmG_i}^{\perp} \subset \Psi^{\perp}_{\rmK}/\Psi \cap \NS(\rmG_i)$ and the proof is complete.
\end{proof}

Unfortunately, in general the intersection polarisations $L_1$ and $L_2$ are not sufficient to determine the corresponding lattice/lifted polarisation. However, we can recover some information in many cases.

\begin{lemma} \label{lem:latticeinjection} Let $\varphi_i \colon K_{\rmG_i}^{\perp} \to \NS(\rmM)$ denote the morphism induced by the morphism $\Psi^{\perp}_{\rmK}/\Psi \to \NS(\rmM)$ from Proposition \ref{prop:NS}, for $i \in \{1,2\}$, and let $\varphi := \varphi_1 \oplus \varphi_2 \colon K_{\rmG_1}^{\perp} \oplus K_{\rmG_2}^{\perp} \to \NS(\rmM)$. Then the $\varphi_i$ are injective and moreover:
\begin{itemize}
\item If $\deg(\rmG) \neq 0$, then $\varphi$ is injective. Its image is a finite index sublattice of $\NS(\rmM)$.
\item If $\deg(\rmG) = 0$, then $\zeta \in K_{\rmG_1}^{\perp} \oplus K_{\rmG_2}^{\perp}$ and $\ker(\varphi) = \bZ \zeta$.
\end{itemize}
In particular, $\varphi_i$ realises $L_i$ as a sublattice of $L$ for each $i \in \{1,2\}$ and if $\deg(\rmG) \neq 0$, then $\varphi$ realises $L_1 \oplus L_2$ as a sublattice of $L$. 
\end{lemma}
\begin{proof} Injectivity of $\varphi_i$ follows from the fact that the kernel $\bZ\zeta $ of the morphism $\Psi^{\perp}_{\rmK}/\Psi \to \NS(\rmM)$ has trivial intersection with $K_{\rmG_i}^{\perp} $ for each $i$. 

So suppose that $\deg(\rmG) \neq 0$. Let $(v_1,v_2)$ and $(w_1,w_2)$ be any two distinct elements of  $K_{\rmG_1}^{\perp} \oplus K_{\rmG_2}^{\perp}$ and assume for a contradiction that  $\varphi(v_1,v_2) = \varphi(w_1,w_2)$. Then $(v_1,v_2)$ and $(w_1,w_2)$ differ by an integer multiple of $\zeta$ so, by Lemma \ref{lem:zetaisotrivial}, we have $(v_1,v_2) = (w_1 - cK_{\rmG_1}, w_2 + c K_{\rmG_2})$,
for some nonzero $c \in \bQ$. Now, as $v_1,w_1 \in K_{\rmG_1}^{\perp}$, we have
\[0 = q(v_1,K_{\rmG_1}) = q(w_1,K_{\rmG_1}) - cq(K_{\mathrm{G}_1},K_{\mathrm{G}_1}) = -cq(K_{\mathrm{G}_1},K_{\mathrm{G}_1}).\]
But this contradicts the assumption that $\deg(\rmG) \neq 0$. Thus $\varphi$ must be injective. The fact that the image is a finite index sublattice then follows if we note that both $K_{\rmG_1}^{\perp} \oplus K_{\rmG_2}^{\perp}$ and $\NS(\rmM)$ have rank $18$.

If $\deg(\rmG) = 0$, then $q(K_{\rmG_i},K_{\rmG_i}) = 0$ for $i \in \{1,2\}$, so $K_{\rmG_i} \in K_{\rmG_i}^{\perp}$. It follows from Lemma \ref{lem:zetaisotrivial} that $\bZ\zeta \subset K_{\rmG_1}^{\perp} \oplus K_{\rmG_2}^{\perp}$, which is the kernel of the morphism $\Psi^{\perp}_{\rmK}/\Psi \to \NS(\rmM)$.
\end{proof}

\begin{remark} Using the fact that $\NS(\rmM) \cong H \oplus E_8 \oplus E_8$ (Proposition \ref{prop:NSfordegenerations}), which is unimodular, and the discriminant groups of $K_{\rmG_i}^{\perp}$ computed in Lemma \ref{lem:NSproperties}, one may show that the index of $\im(\varphi)$ in $\NS(\rmM)$ is equal to $3$ if $\deg(\rmG) = 9$, equal to $4$ if $\deg(\rmG) = 8'$, and equal to $\deg(\rmG)$ if $1 \leq \deg(\rmG) \leq 8$. If $\deg(\rmG) = 0$, then the quotient $\NS(M) / \im(\varphi)$ is a free $\bZ$-module of rank $1$.
\end{remark} 

\begin{definition} \label{def:coupling} The quotient group $Q(L) := L\, /\, \varphi(L_1 \oplus L_2)$ is called the \emph{coupling group} of $L$.
\end{definition}

The coupling group measures how far the lattice polarisation $L \subset \NS(\rmM)$ is from being induced by the two intersection polarisations $L_i \subset \NS(\rmG_i)$. It consists of elements of $L$ that ``cross between'' or ``couple'' the part of $L$ contained in $\NS(\rmG_1)$ and the part of $L$ contained in $\NS(\rmG_2)$. It fits into an exact sequence of abelian groups
\begin{equation}\label{eq:Qsequence}
L_1 \oplus L_2\stackrel{\varphi}{\longrightarrow} L \longrightarrow Q(L) \longrightarrow 0.
\end{equation}
Note that if $\deg(\rmG) \neq 0$ then $\varphi$ is injective, so in this case sequence \eqref{eq:Qsequence} is also exact on the left.

Our next result shows that the coupling group can be used to measure the failure of orthogonal complements to descend to the intersection polarisation.

\begin{proposition} \label{prop:couplingintersection}
Assume that $\deg(\rmG) \neq 0$. Then the coupling group $Q(L)$ is torsion if and only if $(L_i)^{\perp}_{K_{\rmG_i}^{\perp}} = \hat{L}^{\perp} \cap \NS(\rmG_i)$ for each $i \in \{1,2\}$, where $(L_i)^{\perp}_{K_{\rmG_i}^{\perp}}$ denotes the orthogonal complement of $L_i$ taken in $K_{\rmG_i}^{\perp}$.
\end{proposition}

\begin{proof}
Note first that, as the class $\zeta$ is totally isotropic, without ambiguity we can take $\hat{L}^{\perp}$ to be either the orthogonal complement to $\hat{L}$ in $\Psi^{\perp}_{\rmK}/\Psi$ or the lift of $L^{\perp}$ under $\Psi^{\perp}_{\rmK}/\Psi \to \NS(\rmM)$, as these both give the same sublattice of $\Psi^{\perp}_{\rmK}/\Psi$. Note also that $L_i \subset \hat{L}$, so $\hat{L}^{\perp} \cap \NS(\rmG_i) \subset (L_i)^{\perp}_{K_{\rmG_i}^{\perp}}$.

Now assume that $Q(L)$ is torsion. For simplicity of notation, in this part of the proof we will assume $i=1$; the proof for $i=2$ is identical.  Let $v \in (L_1)^{\perp}_{K_{\rmG_1}^{\perp}} \subset \NS(\rmG_1)$. Then $\varphi(v,0) \in \NS(\rmM)$. Let $w \in L$. Then as $Q(L)$ is torsion, there exists $k \in \bZ$ such that $kw = \varphi(w_1,w_2)$ for some $w_i \in L_i$. Then $q(\varphi(v,0),kw)= q_1(v,w_1) + q_2(0,w_2) = 0$ as $v \in L_i^{\perp}$; here $q$ denotes the bilinear form on $\NS(\rmM)$ and $q_i$ denotes the bilinear form on $\NS(\rmG_i)$. So $\varphi(v,0)  \subset L^{\perp}$ and consequently $(v,0) \in \hat{L}^{\perp}$. This proves that $\hat{L}^{\perp} \cap \NS(\rmG_1) = (L_1)^{\perp}_{K_{\rmG_1}^{\perp}}$

To prove the converse, assume that $Q(L)$ is not torsion. Then $\rank(L_1) + \rank(L_2) < \rank(L)$. Taking orthogonal complements, we find that \[\rank\left((L_1)^{\perp}_{K_{\rmG_1}^{\perp}}\right) + \rank\left((L_2)^{\perp}_{K_{\rmG_2}^{\perp}}\right) > \rank\left(L^{\perp})\right.\] 
But then we cannot have $(L_i)^{\perp}_{K_{\rmG_i}^{\perp}} = \hat{L}^{\perp} \cap \NS(\rmG_i)$ for both $i \in \{1,2\}$.
\end{proof}

We conclude this section with a result about automorphisms of N\'{e}ron-Severi lattices, which will be useful in Section \ref{sec:mirror}.

\begin{proposition} \label{prop:NSMautos} For each $i \in \{1,2\}$, suppose that $\psi_i$ is an isometry of $\NS(\rmG_i)$ which fixes $K_{\rmG_i}$. Then $\psi_1 \oplus \psi_2$ induces an isometry $\psi$ of $\NS(\rmM)$ which satisfies $\psi(K_{\rmG_i}^{\perp}) = K_{\rmG_i}^{\perp}$ for each $i \in \{1,2\}$ (where we identify $K_{\rmG_i}^{\perp}$ with its image under the injective map $\varphi_i$ from Lemma \ref{lem:latticeinjection}).

Conversely, suppose that $\psi$ is an isometry of $\NS(\rmM)$ satisfying $\psi(K_{\rmG_i}^{\perp}) = K_{\rmG_i}^{\perp}$ for each $i \in \{1,2\}$. Then, after possibly multiplying $\psi$ by $(-1)$, for each $i \in \{1,2\}$ there are isometries $\psi_i$ of $\NS(\rmG_i)$ which fix $K_{\rmG_i}$ and induce $\psi$.
\end{proposition}

\begin{proof} To prove the first part we use the matrix representation from the proof of Proposition \ref{prop:qdpisomorphism}. For each $i \in \{1,2\}$, choose bases for $\NS(\rmG_i)$, let $Q_i$ denote the matrix of the intersection form and let $K_i$ denotes the column vector representing $K_{\rmG_i}$ with respect to this basis. Let $P_i$ denote the matrix of $\psi_i$; note that $P_i^TQ_iP_i = Q_i$ and $P_iK_i = K_i$ by assumption. Let $P$ denote the matrix 
\[P = \begin{pmatrix} P_1 & 0 \\ 0 & P_2 \end{pmatrix}\]
of $\psi_1 \oplus \psi_2$.

From the proof of Proposition \ref{prop:qdpisomorphism}, the map $f$ induces the map $F\colon \NS(\rmG_1) \oplus \NS(\rmG_2) \to \rmE$ given by $F(v_1,v_2) = (-K_1^TQ_1v_1+K_2^TQ_2v_2)a$, where $a$ is one of the standard basis elements $(a,b)$ for $\rmE$. An easy matrix calculation then shows that $FP = F$. It follows that $P$ induces an isometry on $\ker(F) = \Psi^{\perp}_{\rmK}/\Psi$. This isometry fixes $\zeta = c(-K_{\rmG_1},K_{\rmG_2})$ (for some $c \in \bQ$), so further descends to an isometry of $\NS(\rmM)$.

To prove the converse, assume first that $\deg(\rmG) \neq 0$. Then $K_{\rmG_1}^{\perp}$ and $K_{\rmG_2}^{\perp}$ are non-degenerate by Lemma \ref{lem:NSproperties}. The map $\psi$ induces isometries of the lattices $K_{\rmG_i}^{\perp}$, which descend to automorphisms of the discriminant groups $\disc(K_{\rmG_i}^{\perp})$ that preserve the quadratic forms. These in turn induce an automorphism $\overline{\psi}$ of $\disc(K_{\rmG_1}^{\perp} \oplus K_{\rmG_2}^{\perp}) \cong \disc(K_{\rmG_1}^{\perp}) \oplus \disc(K_{\rmG_2}^{\perp})$.

Let $A := \NS(\rmM)\, / \, (K_{\rmG_1}^{\perp} \oplus K_{\rmG_2}^{\perp})$. Then $A$ is a subgroup of $\disc(K_{\rmG_1}^{\perp} \oplus K_{\rmG_2}^{\perp})$ which is isotropic with respect to the quadratic form, and the order of $A$ is the index of $K_{\rmG_1}^{\perp} \oplus K_{\rmG_2}^{\perp}$ in $\NS(\rmM)$. By \cite[Proposition 1.4.2]{isbfa}, the automorphism $\overline{\psi}$ must preserve $A$.

Now, using the explicit descriptions of the discriminant groups from Lemma \ref{lem:NSproperties}, one may compute that the only automorphisms of $\disc(K_{\rmG_1}^{\perp} \oplus K_{\rmG_2}^{\perp})$ preserving $A$ and the quadratic forms are $\pm \mathrm{Id}$. After multiplying $\psi$ by $(-1)$ if necessary, we may therefore assume that $\overline{\psi} = \mathrm{Id}$.

In particular, $\psi$ acts trivially on  $\disc(K_{\rmG_i}^{\perp})$ for each $i \in \{1,2\}$. By \cite[Proposition 8.2.5]{cagmv}, it follows that $\psi$ extends to isometries $\psi_i$ of $\NS(\rmG_i)$ which fix $K_{\rmG_i}$, as required.

Finally, assume that $\deg(\rmG) = 0$. Then $K_{\rmG_1}^{\perp} \cap K_{\rmG_2}^{\perp} \cong \bZ$ is preserved by $\psi$. Let $\alpha \in K_{\rmG_1}^{\perp} \cap K_{\rmG_2}^{\perp}$ be a primitive generator. After multiplying by $(-1)$ if necessary, we may assume that $\psi(\alpha)= \alpha$. Consequently, we may assume that $\varphi_1(K_{\rmG_1}) = \varphi_2(K_{\rmG_2}) \in K_{\rmG_1}^{\perp} \cap K_{\rmG_2}^{\perp}$ is fixed by $\psi$.

The class $\alpha$ is totally isotropic in $K_{\rmG_i}^{\perp}$ for each $i \in \{1,2\}$ and, by Lemma \ref{lem:NSproperties},  we have a (non-unique) isometry $K_{\rmG_i}^{\perp} \cong \bZ \alpha \oplus E_8$. After fixing a choice of such an isometry, we may extend a standard basis for $\bZ \alpha \oplus E_8$ to a basis for $\NS(\rmG_i) \cong \rmI_{1,9} \cong \begin{psmallmatrix} 0 & 1 \\ 1 & -1 \end{psmallmatrix} \oplus E_8$ by adding a $(-1)$-class $\beta$, so that $\{\alpha,\beta\}$ forms a standard basis for the $\begin{psmallmatrix} 0 & 1 \\ 1 & -1 \end{psmallmatrix}$ factor. 

To prove the result it suffices to show that any isometry $\psi$ of $K_{\rmG_i}^{\perp}$ with $\psi(\alpha) = \alpha$ extends to an isometry $\psi_i$ of $\NS(\rmG_i)$. To do this it suffices to find $\psi_i(\beta)$. But it is a straightforward linear algebra computation to show that $\psi_i(\beta)$ must be the unique $(-1)$-class which is orthogonal to $\psi_i(E_8)$ and has $q_i(\alpha,\psi_i(\beta)) = 1$ (where $q_i$ denotes the bilinear form on $\NS(\rmG_i)$).
\end{proof}

\section{Tyurin Degenerations of K3 Surfaces}\label{sec:degenerations}

Throughout this section we use the following notation. Let $\pi\colon \calX \to \Delta$ denote a semistable Type II degeneration of K3 surfaces over the open unit disc in $\bC$; note in particular that as part of the Type II condition we assume that $K_{\calX} \sim 0$. Let $X \subset \calX$ denote a general fibre and let $X_p \subset \calX$ denote the fibre over a point $p \in \Delta$.  We assume further that $\pi\colon \calX \to \Delta$ is a \emph{Tyurin degeneration}, meaning that the central fibre $X_0 = V_1 \cup_C V_2$ is a \emph{stable K3 surface of Type II} in the sense of the following definition, originally due to Friedman. We let $i_j\colon C \to V_j$ denote the inclusion.

\begin{definition}\cite[Definition 3.1]{npgttk3s} A \emph{stable K3 surface of Type II} is a surface with normal crossings $X_0 = V_1 \cup_C V_2$  consisting of two smooth rational surfaces $V_1$ and $V_2$ glued along a smooth anticanonical elliptic curve $C$, satisfying $N_{C/V_1} \otimes N_{C/V_2} \cong \calO_C$.
\end{definition}

\begin{remark} \label{rem:canonicalsquare} Note that the normal bundle condition implies that if $X_0 = V_1 \cup_C V_2$ is a stable K3 surface of Type II and $q$ denotes the bilinear form on $\NS(V_i)$, then $q(K_{V_1},K_{V_1}) = -q(K_{V_2},K_{V_2})$. In line with our convention in Section \ref{sec:pseudolattice}, we will henceforth assume that labels have been chosen so that $q(K_{V_1},K_{V_1}) \geq 0$.
\end{remark}

\begin{remark} Tyurin degenerations are in some sense primitive in the class of Type II degenerations: given a Type II degeneration of polarised K3 surfaces, after twisting the polarisation and performing some birational modifications on the central fibre, one may obtain a Tyurin degeneration. We refer the interested reader to \cite[Theorems 2.2 and 2.3]{npgttk3s} for details.
\end{remark}

\subsection{Hodge theory} \label{sec:hodge}

We begin by studying the Hodge theory of such degenerations. The material in this section largely follows from work of Friedman \cite{npgttk3s} and Scattone \cite{cmsak3s}.

In the limiting mixed Hodge structure $H^2_{\lim}(X)$, the weight filtration is defined over $\bZ$ and is given by
\[0 \subset I \subset I^{\perp} \subset H^2(X;\bZ),\]
where $I \subset H^2(X;\bZ)$ is a primitive rank two isotropic sublattice, the choice of which is unique up to isometries of $H^2(X;\bZ)$ (see \cite[Section 2.2]{cmsak3s}). 

\begin{lemma} \label{lem:HE8E8} The quotient $I^{\perp}/I$ is isomorphic to the lattice $H \oplus E_8 \oplus E_8$, where $H$ denotes the even unimodular lattice of signature $(1,1)$ (the hyperbolic plane) and $E_8$ denotes the usual negative definite root lattice.
\end{lemma}
\begin{proof}
$H^2(X;\bZ)$ is isomorphic to the K3 lattice $\Lambda_{\mathrm{K3}} := H \oplus H \oplus H \oplus E_8 \oplus E_8$. Note that each copy of $H$ contains a primitive isotropic vector. As $I\subset H^2(X;\bZ)$ is unique up to isometry, without loss of generality we may take $I$ to be generated by taking one primitive isotropic vector from each of the first two factors of $H$. It follows that the quotient $I^{\perp}/I$ is isomorphic to $H \oplus E_8 \oplus E_8$.
\end{proof}

With this notation, we have the following proposition.

\begin{proposition}\label{prop:CSoverZ}
The Clemens-Schmid exact sequence induces exact sequences over $\bZ$ as follows:
\begin{align*}
0 \longrightarrow H^0(X_0;\bZ) \longrightarrow H^0_{\lim}(X) \longrightarrow 0,\\
0 \longrightarrow \bZ \xi \longrightarrow \xi^{\perp} \longrightarrow I^{\perp}/I \longrightarrow 0,\\
0 \longrightarrow \bZ \eta \longrightarrow H^4(X_0;\bZ) \longrightarrow H^4_{\lim}(X) \longrightarrow 0,
\end{align*}
where 
\begin{itemize}
\item $\xi \in H^2(V_1;\bZ) \oplus H^2(V_2;\bZ)$ is the class of $(C,-C)$,
\item $\eta \in H^4(V_1;\bZ) \oplus H^4(V_2;\bZ)$ is the class of $(p,-p)$, for $p$ a point.
\end{itemize}
\end{proposition}

\begin{proof}
The difficult part of this proposition, namely the second sequence, follows immediately from a result of Friedman \cite[Lemma 3.5]{npgttk3s}. We follow Friedman's approach to prove exactness of the remaining two sequences.

Exactness of these sequences follows from the fact that the Clemens-Schmid sequence is obtained by splicing together the Wang sequence
\[\ldots \to H^{i-1}(X;\bZ) \longrightarrow H^{i}(\calX^*;\bZ) \longrightarrow H^i(X;\bZ) \stackrel{N}{\longrightarrow} H^i(X;\bZ) \to \cdots\]
and the exact sequence of a pair
\[\ldots \to H^i_{X_0}(\calX;\bZ) \longrightarrow H^i(X_0;\bZ) \longrightarrow H^i(\calX^*;\bZ) \longrightarrow H^{i+1}_{V_0}(\calX,\bZ) \to \cdots,\]
where $\calX^* = \calX \setminus X_0$. When $i = 0$ or $i = 4$ we have $H^{i-1}(X,\bZ) = 0$ and $N = 0$, so the Wang sequence gives an isomorphism
$H^{i}(\calX^*,\bZ) \cong H^i(X,\bZ)$. Furthermore, we have $H^i_{X_0}(\calX;\bZ) = 0$ for $i = 0, 1, 5$, which completes the proof of the first sequence. To conclude the proof for the third sequence, note that the image of the map $H^4_{X_0}(\calX;\bZ) \to H^4(X_0;\bZ)$ is generated by the class $\eta$.
\end{proof}

\begin{remark}\label{rem:weightgraded}
Note that, by \cite[Lemma 3.5]{npgttk3s}, the second sequence here arises from the second weight graded piece of the Clemens-Schmid sequence
\[\Gr_2^WH^2_{X_0}(\calX) \longrightarrow \Gr_2^WH^2(X_0) \longrightarrow \Gr_2^WH^2_{\lim}(X) \longrightarrow 0,\]
where the class $\xi$ generates the image of the first map and $\xi^{\perp} = \Gr_2^WH^2(X_0)$.

Using the viewpoint of perverse and weight filtrations $P_{\bullet}$ and $W_{\bullet}$ described in \cite[Sections 4.1 and 5.2.1]{mcss}, the classes $\xi$ and $\eta$ are generators of the graded pieces $\Gr_2^W\Gr_2^PH^2(X_0)$ and $\Gr_4^W\Gr_4^PH^4(X_0)$, respectively.
\end{remark}

We may take the direct sum of the exact sequences from Proposition \ref{prop:CSoverZ} to obtain an exact sequence over $\bZ$
\begin{equation}\label{eq:CS}0 \to \bZ\xi \oplus \bZ \eta \to H^0(X_0)\oplus \xi^{\perp} \oplus H^4(X_0) \to H^0_{\lim}(X) \oplus I^{\perp}/I \oplus H^4_{\lim}(X) \to 0.\end{equation}

\subsection{Clemens-Schmid and pseudolattices over the rational numbers}\label{sec:CSpseudo}

Our next aim is to show how the results of Section \ref{sec:hodge} are naturally recovered from the abstract considerations in Section \ref{sec:pseudolattice}, for quasi del Pezzo homomorphisms $f_i \colon \mathrm{G}_i \to \rmE$ arising from the anticanonical pairs $(V_i,C)$ making up the central fibre. In this subsection we begin by working over the rational numbers $\bQ$; in the subsequent subsections we will discuss to what extent this continues to work over $\bZ$.

For any surface $V$, it is well-known that the Chern character map
\[\mathrm{ch}\colon \mathrm{K}_0^{\mathrm{num}}(\mathbf{D}(V)) \longrightarrow H^{\mathrm{even}}(V;\bQ) := \bigoplus_{i=0}^2 H^{2i}(V;\bQ)\]
is a ring homomorphism. Moreover, it follows easily from \cite[Example 3.5]{ecslc} that this map is injective and its $\bQ$-linear extension induces an isomorphism 
\[\mathrm{ch_{\bQ}}\colon \mathrm{K}_0^{\mathrm{num}}(\mathbf{D}(V)) \otimes \bQ \stackrel{\sim}{\longrightarrow} H^0(V;\bQ) \oplus (\NS(V) \otimes \bQ) \oplus H^4(V;\bQ).\]

The Chern character maps the class of $\calO_V$ to $(1,0,0)$, the class of the structure sheaf $\calO_p$ of a point to $(0,0,1)$, and the class of $\calO_V(D)$, for any divisor $D$ on $V$, to $(1,D,\chi(\calO_V(D)) + \frac{1}{2}q(K_V,D) - \chi(\calO_V))$, where $q(\cdot,\cdot)$ denotes the usual bilinear form on $H^2(V;\bQ)$ induced by the cup product; it follows from \cite[Example 3.5]{ecslc} that these classes span the image of $\mathrm{ch}$.

\begin{remark} Note that the Chern character map is \emph{not} the same as the isomorphism $\mathrm{K}_0^{\mathrm{num}}(\mathbf{D}(V))\cong \bZ \oplus \NS(V) \oplus \bZ$ from \cite[Example 3.5]{ecslc} that we used in the proof of Proposition \ref{prop:qdpisomorphism}
\end{remark}

 Using this, it is easy to show that the Euler form on $ \mathrm{K}_0^{\mathrm{num}}(\mathbf{D}(V))$ is induced by the bilinear form on $H^{\mathrm{even}}(V;\bQ)$ with matrix
\begin{equation} \label{eq:Eulermatrix}\kbordermatrix{ & H^0 & H^2 & H^4 \\ H^0 & \chi(\calO_V) & -\frac{1}{2}K^TQ & 1 \\ H^2 & \frac{1}{2}QK & -Q & 0 \\ H^4 & 1 & 0 & 0 },\end{equation}
where $Q$ denotes the matrix  of the bilinear form $q$ with respect to an appropriate choice of basis for $H^2(V;\bQ)$ and $K$ denotes the column vector representing the canonical divisor $K_V$ in this basis.

Suppose that $X_0 = V_1 \cup_C V_2$ is a stable K3 surface of Type II. Each of the surfaces $V_j$ is a quasi del Pezzo surface. So, by Example \ref{ex:qdp}, the derived pull-back map $i_j^*\colon \mathrm{K}_0^{\mathrm{num}}(\mathbf{D}(V_j)) \to \mathrm{K}_0^{\mathrm{num}}(\mathbf{D}(C))$ is a quasi del Pezzo homomorphism of pseudolattices, with right adjoint given by derived push-forward $i_{j*}$. The class of the structure sheaf of a point defines a point-like element in $\mathrm{K}_0^{\mathrm{num}}(\mathbf{D}(V_j))$ and with respect to this point-like element we have $\NS(\mathrm{K}_0^{\mathrm{num}}(\mathbf{D}(V_j))) \cong \NS(V_j)\cong H^2(V_j;\bZ)$. This isomorphism identifies the canonical classes, so by Remark \ref{rem:canonicalsquare} we have $q(K_{V_1},K_{V_1}) = -q(K_{V_2},K_{V_2})$. It follows that we are in the setting of Section \ref{sec:pseudolattice}, with $\rmG_j = \mathrm{K}_0^{\mathrm{num}}(\mathbf{D}(V_j)))$ and $f_j = i_j^*$. 

Let $f \colon \mathrm{G} \to \mathrm{E}$ denote the spherical homomorphism obtained by gluing the pseudolattices $\rmG_1 = \mathrm{K}_0^{\mathrm{num}}(\mathbf{D}(V_1))$ and $\rmG_2 = \mathrm{K}_0^{\mathrm{num}}(\mathbf{D}(V_2))$ along $  \mathrm{K}_0^{\mathrm{num}}(\mathbf{D}(C)) \cong \mathrm{E}$ via the homomorphisms $i_{1}^*$ and $(-i_{2}^*)$, as in Section \ref{sec:pseudolattice}, and let $r$ be its right adjoint. The results of Section \ref{sec:pseudolattice} give rise to lattices $\overline{\mathrm{E}}$, $\mathrm{K}$, and $\rmM$, related by the short exact sequence \eqref{eq:ses}. 

\begin{lemma} \label{lem:K} The Chern character map induces a natural isomorphism of $\bQ$-modules
\[\kappa\colon \mathrm{G}_{\bQ} \stackrel{\sim}{\longrightarrow} H^{\mathrm{even}}(V_1;\bQ) \oplus H^{\mathrm{even}}(V_2;\bQ),\]
where $\mathrm{G}_{\bQ} := \mathrm{G} \otimes \bQ$. The restriction of $\kappa$ to $\mathrm{K}_{\bQ} := \mathrm{K} \otimes \bQ$ induces an isomorphism
\[\mathrm{K}_{\bQ} \stackrel{\sim}{\longrightarrow} H^0(X_0;\bQ) \oplus \Gr_2^WH^2(X_0;\bQ) \oplus H^4(X_0;\bQ),\]
where $W_{\bullet}$ is the weight filtration arising from Deligne's mixed Hodge structure on a normal crossing variety.
\end{lemma}
\begin{proof}
As $H^2(V_j;\bZ) \cong \NS(V_j)$, the Chern character induces an isomorphism between $\mathrm{K}_0^{\mathrm{num}}(\mathbf{D}(V_j)) \otimes \bQ$ and $H^{\mathrm{even}}(V_j;\bQ)$. Under this identification, the bilinear form $\langle \cdot, \cdot \rangle_{\mathrm{K}_0^{\mathrm{num}}(\mathbf{D}(V_j))}$ is given by Equation \eqref{eq:Eulermatrix} and the maps $i_j^*$ and $i_{j*}$ are given by
\[\kbordermatrix{ & H^0 & H^2 & H^4 \\ {[\calO_p]} & 0 & -K_j^TQ_j & 0 \\ {[\calO_C]} & 1 & 0 & 0} \quad \text{and} \quad \kbordermatrix{ & {[\calO_p]} & {[\calO_C]} \\ H^0 & 0 & 0 \\ H^2 & 0 & -K_j \\ H^4 & 1 & -\frac{1}{2}q(K_{V_j},K_{V_j}) }\]
respectively,  where $Q_j$ denotes the matrix of the usual bilinear form on $H^2(V_j;\bQ)$ with respect to an appropriate choice of basis, and $K_j$ denotes the column vector representing the canonical divisor $K_{V_j}$ in this basis.

 We may thus identify $\mathrm{G} \otimes \bQ$ with $H^{\mathrm{even}}(V_1;\bQ) \oplus H^{\mathrm{even}}(V_2;\bQ)$ as $\bQ$-modules, proving the first statement. The bilinear form on $\mathrm{G}$ induces a bilinear form on $H^{\mathrm{even}}(V_1;\bQ) \oplus H^{\mathrm{even}}(V_2;\bQ)$ with matrix
\[\kbordermatrix{ & H^0(V_1) & H^2(V_1) & H^4(V_1) & H^0(V_2) & H^2(V_2) & H^4(V_2) \\ 
H^0(V_1) &  1 & -\frac{1}{2}K_1^TQ_1 & 1 & 0 & K_2^TQ_2 & 0 \\ 
H^2(V_1) & \frac{1}{2}Q_1K_1 & -Q_1 & 0 & -Q_1K_1 & 0 & 0 \\ 
H^4(V_1) &1 & 0 & 0 & 0 & 0 & 0 \\
H^0(V_2) &0 & 0 & 0 & 1 & -\frac{1}{2}K_2^TQ_2 & 1 \\
H^2(V_2) & 0 & 0 & 0 & \frac{1}{2}Q_2K_2 & -Q_2 & 0 \\ 
H^4(V_2) &0 & 0 & 0 & 1 & 0 & 0 }.\]

Now recall that $f(u_1,u_2) = i_1^*(u_1)  - i_2^*(u_2)$ for $u_i \in \mathrm{G}_i$ and, by the proof of \cite[Proposition 2.23]{pdpslf}, $r(v) = (i_{1*}(v) - i_{1*}i_2^*i_{2*}(v), -i_{2*}(v))$ for $v \in \mathrm{E}$. The maps $f$ and $r$ from Section \ref{sec:pseudolattice} are therefore given respectively by
\[\kbordermatrix{  & H^0(V_1) & H^2(V_1) & H^4(V_1) & H^0(V_2) & H^2(V_2) & H^4(V_2) \\
{[\calO_p]} & 0 & -K_1^TQ_1 & 0 & 0 & K_2^TQ_2 & 0 \\ 
{[\calO_C]} & 1 & 0 & 0 & -1 & 0 & 0 }\]
and
\[\kbordermatrix{ & {[\calO_p]} & {[\calO_C]} \\ 
H^0(V_1) & 0 & 0 \\ 
H^2(V_1) & 0 & -K_1  \\ 
H^4(V_1) & 1 & \frac{1}{2}q(K_{V_1},K_{V_1}) \\
H^0(V_2) & 0 & 0 \\
H^2(V_2) & 0 & K_2 \\ 
H^4(V_2) & -1 & \frac{1}{2}q(K_{V_2},K_{V_2})}.\]
If we write a general vector in $H^{\mathrm{even}}(V_j;\bQ)$ as $(r_j,D_j,s_j)$, it follows that the kernel $\mathrm{K}_{\bQ}$ of $f$ is identified with the subspace of $H^{\mathrm{even}}(V_1;\bQ) \oplus H^{\mathrm{even}}(V_2;\bQ)$ defined by
\[\{(r_1,D_1,s_1,r_2,D_2,s_2) \mid r_1 = r_2 \text{ and } q(K_{V_1},D_1) = q(K_{V_2},D_2)\}.\]

Now, Deligne's mixed Hodge structure on a normal crossing variety gives that
\begin{align*}H^0(X_0) &\cong \{(r_1,r_2) \in H^0(V_1) \oplus H^0(V_2) \mid r_1 = r_2\},\\
\Gr_2^WH^2(X_0) &\cong \{(D_1,D_2) \in H^2(V_1) \oplus H^2(V_2) \mid q(K_{V_1},D_1) = q(K_{V_2},D_2)\}, \\
H^4(X_0) &\cong  H^4(V_1) \oplus H^4(V_2),\end{align*}
so we see that $H^0(X_0) \oplus \Gr_2^WH^2(X_0) \oplus H^4(X_0) \cong \mathrm{K}_{\bQ}$, as required. \end{proof}

\begin{notation}\label{not:6tuple} In light of the isomorphisms from this lemma, we will write elements of $\mathrm{G}_{\bQ}$ as $6$-tuples
\[(r_1,D_1,s_1,r_2,D_2,s_2),\]
where $r_i \in H^{0}(V_i;\bQ)$, $D_i \in H^{2}(V_i,\bQ)$, and $s_i \in H^{4}(V_i;\bQ)$. By the proof of this lemma, we can express elements of the submodule $\mathrm{K}_{\bQ}$ as $6$-tuples as above satisfying the two conditions $r_1 = r_2$ and $q(K_{V_1},D_1) = q(K_{V_2},D_2)$.
\end{notation}

Now suppose that $X_0 = V_1 \cup_C V_2$ occurs as the central fibre of a Tyurin degeneration $\pi \colon \calX \to \Delta$. Then the following result gives an interpretation of the lattice $\rmM$.

\begin{lemma} \label{lem:L} There is a natural isomorphism of $\bQ$-modules 
\[\mu\colon \rmM_{\bQ} \stackrel{\sim}{\longrightarrow} H^0_{\lim}(X) \oplus \Gr_2^WH^2_{\lim}(X) \oplus H^4_{\lim}(X),\]
where $\rmM_{\bQ} := \rmM \otimes \bQ$  and $W_{\bullet}$ is the weight filtration arising from the limiting mixed Hodge structure. Moreover, under the isomorphisms $\kappa$, $\mu$ the map $\mathrm{K}_{\bQ} \to \rmM_{\bQ}$ from the short exact sequence \eqref{eq:ses} agrees with the one induced by the Clemens-Schmid exact sequence \eqref{eq:CS}.
\end{lemma}
\begin{proof} By Proposition \ref{prop:CSoverZ} and Remark \ref{rem:weightgraded}, the Clemens-Schmid exact sequence gives rise to an isomorphism $\varphi_0\colon H^0(X_0) \stackrel{\sim}{\to} H^0_{\lim}(X)$, along with surjective maps $\varphi_2\colon \Gr_2^WH^2(X_0) \to \Gr_2^WH^2_{\lim}(X)$ and $\varphi_4\colon H^4(X_0) \to H^4_{\lim}(X)$. Composing $\varphi := \varphi_0 \oplus \varphi_2 \oplus \varphi_4$ with the isomorphism $\kappa$ from Lemma \ref{lem:K}, we thus obtain a surjective map of $\bQ$-modules
\[\varphi \circ \kappa \colon \mathrm{K}_{\bQ} \longrightarrow H^0_{\lim}(X) \oplus \Gr_2^WH^2_{\lim}(X) \oplus H^4_{\lim}(X).\] 
To prove the lemma, it suffices to show that the kernel of $\varphi \circ \kappa$ is isomorphic to $\overline{\mathrm{E}} \otimes \bQ = \im(r) \otimes \bQ \subset \mathrm{K}_{\bQ}$.

From the proof of Lemma \ref{lem:K}, we see that $\overline{\mathrm{E}} \otimes \bQ \subset \mathrm{K}_{\bQ}$ is generated by the columns of the matrix of the map $r$. Writing a general element of $\mathrm{K}_{\bQ}$ as $(r_1,D_1,s_1,r_2,D_2,s_2)$, and recalling that $q(K_{V_1},K_{V_1}) = -q(K_{V_2},K_{V_2})$, we may simplify this set of generators to $\{(0,0,1,0,0,-1), (0,-K_{V_1},0,0,K_{V_2},0)\}$.

By Proposition \ref{prop:CSoverZ}, we see that $\varphi_0$ is an isomorphism and the kernels of $\varphi_2$ and $\varphi_4$ are generated by the classes $\xi = (-K_{V_1}, K_{V_2}) \in \Gr_2^WH^2(X_0)$ and $\eta = (1,-1) \in H^4(X_0)$ respectively. It thus follows immediately that $\ker(\varphi) = \overline{\mathrm{E}} \otimes \bQ$; this completes the proof.\end{proof}

Putting everything in this subsection together, we have the following result, which gives a relationship between the sequences \eqref{eq:ses} and \eqref{eq:CS} over $\bQ$.

\begin{proposition} \label{prop:CSlattices} Let $\pi \colon \calX \to \Delta$ be a Tyurin degeneration with central fibre $X_0$ a stable K3 surface of Type II. Then there is a commutative diagram of $\bQ$-modules
\[
\begin{tikzcd}[cramped,column sep=1.6em]
0 \ar[r] & \overline{\mathrm{E}}_{\bQ} \ar[d] \ar[r]& \mathrm{K}_{\bQ}\ar[r]\ar[d,"\kappa"] & \rmM_{\bQ} \ar[d,"\mu"]\ar{r} & 0 \\
0 \ar[r] & \bQ\xi \oplus \bQ \eta \ar[r] & H^0(X_0)\oplus \xi^{\perp} \oplus H^4(X_0) \ar[r,"\varphi"] & H^0_{\lim}(X) \oplus I^{\perp}/I \oplus H^4_{\lim}(X) \ar[r] & 0
\end{tikzcd}\]
where the vertical arrows are isomorphisms.
\end{proposition}

\subsection{Integrality}

Next we investigate to what extent Proposition \ref{prop:CSlattices} holds over the integers. We note first that the map $\kappa$ is not integral: indeed, let $D_1 \in \NS(V_1) \cong H^2(V_1;\bZ)$ and $D_2 \in \NS(V_2) \cong H^2(V_2;\bZ)$. Then $(\calO_{V_1}(D_1), \calO_{V_2}(D_2))$ is an integral element of $\mathrm{G}$. Applying the map $\kappa$ gives 
\[\kappa(\calO_{V_1}(D_1), \calO_{V_2}(D_2)) = (1,D_1,s_1,1,D_2,s_2)\]
with
\begin{align*}
s_1 &= \chi(\calO_{V_1}(D_1))+\tfrac{1}{2}q(K_{V_1},D_1) - \chi(\calO_{V_1}),\\
s_2 &= \chi(\calO_{V_2}(D_2))+\tfrac{1}{2}q(K_{V_2},D_2) - \chi(\calO_{V_2}).
\end{align*}
From these expressions we see that $\kappa(\calO_{V_1}(D_1), \calO_{V_2}(D_2))$ is not integral unless $q(K_{V_1},D_1)$ and $q(K_{V_2},D_2)$ are even. The following lemma shows that this is, essentially, the only way in which integrality can fail.

\begin{lemma}\label{lem:KoverZ}
The map $\kappa$ induces an isomorphism between $\mathrm{G}$ and the subset of $H^{\mathrm{even}}(V_1;\bQ) \oplus H^{\mathrm{even}}(V_2;\bQ)$ defined by those $6$-tuples $(r_1,D_1,s_1,r_2,D_2,s_2)$ satisfying
\begin{itemize}
\item $r_i \in H^0(V_i;\bZ)$ for each $i = 1,2$;
\item $D_i \in H^2(V_i;\bZ)$ for each $i = 1,2$;
\item $2s_i \in H^4(V_i;\bZ)$ for each $i = 1,2$, and $s_i \in H^4(V_i;\bZ)$ if and only if $q(K_{V_i},D_i)$ is even.
\end{itemize}
This isomorphism takes $\mathrm{K}$ to the set of $6$-tuples as above satisfying $r_1 = r_2$ and $q(K_{V_1},D_1) = q(K_{V_2},D_2)$.
\end{lemma}
\begin{proof} This is a straightforward consequence of the explicit description of the Chern character map given at the start of Section \ref{sec:CSpseudo}.
\end{proof}

Whilst the isomorphism $\kappa$ is not integral, somewhat surprisingly we find that the isomorphism $\mu$ is.

\begin{lemma} \label{lem:MoverZ}
The map $\mu \colon \rmM \to H^0_{\lim}(X) \oplus I^{\perp}/I \oplus H^4_{\lim}(X)$ is an isomorphism over $\bZ$.
\end{lemma}
\begin{proof}
We first note that the top and bottom rows of the commutative diagram in the statement of Proposition \ref{prop:CSlattices} are defined over $\bZ$, by the results of Sections \ref{sec:pseudolattice} and \ref{sec:hodge}. By commutativity of this diagram, it suffices to show that the composition $\varphi \circ \kappa \colon \mathrm{K} \to H^0_{\lim}(X) \oplus I^{\perp}/I \oplus H^4_{\lim}(X)$ is well-defined and surjective over $\bZ$.

The map $\varphi$ is given explicitly by
\[\varphi(r_1,D_1,s_1,r_2,D_2,s_2) = (r_1,D_1+D_2,s_1+s_2) \in H^0_{\lim}(X) \oplus I^{\perp}/I \oplus H^4_{\lim}(X).\] 
Combining this with the description of the image of $\kappa$ from Lemma \ref{lem:KoverZ}, we find that the image of $\varphi \circ \kappa$ is integral in $H^0_{\lim}(X) \oplus I^{\perp}/I \oplus H^4_{\lim}(X)$, so  $\varphi \circ \kappa$ is defined over $\bZ$ .

It remains to show that $\varphi \circ \kappa$ is surjective. Indeed, as the map $\varphi$ is surjective over $\bZ$, for any $x \in H^0_{\lim}(X) \oplus I^{\perp}/I \oplus H^4_{\lim}(X)$ there exists 
\[y = (r_1,D_1,s_1,r_2,D_2,s_2) \in H^0(X_0;\bZ)\oplus \xi^{\perp} \oplus H^4(X_0;\bZ)\]
with $\varphi(y) = x$. To complete the proof, note that if $y \notin \kappa(\mathrm{K})$, we may use the description $\kappa(\mathrm{K})$ from Lemma \ref{lem:KoverZ} and the fact that
\[\varphi(r_1,D_1,s_1,r_2,D_2,s_2) = \varphi(r_1,D_1,s_1 + \tfrac{1}{2},r_2,D_2,s_2 - \tfrac{1}{2}),\]
to obtain an element of $\kappa(\mathrm{K})$ mapping to $x$.
\end{proof}

As noted in Section \ref{sec:pseudolattice}, the bilinear form on $\mathrm{K}$ descends under the map $\mathrm{K} \to \rmM$ to give a bilinear form on $\rmM$. However, $H^0_{\lim}(X) \oplus I^{\perp}/I \oplus H^4_{\lim}(X)$ also comes equipped with a bilinear form, induced by the bilinear form on $H^{\mathrm{even}}_{\lim}(X) \cong H^{\mathrm{even}}(X;\bZ)$ defined by Equation \eqref{eq:Eulermatrix} (which is integral as $K_X = 0$). The following result shows that these bilinear forms agree.

\begin{lemma} \label{lem:Lagrees} Equipping $\rmM$ and $H^0_{\lim}(X) \oplus I^{\perp}/I \oplus H^4_{\lim}(X)$ with bilinear forms as above, the map $\mu$ is an isomorphism of lattices.
\end{lemma}
\begin{proof} Using Equation \eqref{eq:Eulermatrix} and the explicit computations in the proofs of Lemmas \ref{lem:K} and \ref{lem:L}, one may show that both bilinear forms are given explicitly by the matrix
\[\kbordermatrix{ & H^0_{\lim} & \Gr_2^WH^2_{\lim} & H^4_{\lim} \\ H^0_{\lim} & 2 & 0 & 1 \\ \Gr_2^WH^2_{\lim} & 0 & -Q & 0 \\ H^4_{\lim} & 1 & 0 & 0 }\]
where $Q$ denotes the matrix induced by the cup-product form on $\Gr_2^WH^2_{\lim}(X)$ with respect to an appropriate basis. \end{proof}

Putting this lemma together with the exact sequence \eqref{eq:ses}, we thus obtain the following result.

\begin{proposition}\label{prop:Tyurinexactsequence} In the context of a Tyurin degeneration of K3 surfaces, the exact sequence \eqref{eq:ses} can be realised as
 \[0 \longrightarrow \overline{\mathrm{E}} \longrightarrow \mathrm{K} \longrightarrow H^0_{\lim}(X) \oplus I^{\perp}/I \oplus H^4_{\lim}(X) \longrightarrow 0,\]
 where 
 \begin{itemize}
\item $\overline{\mathrm{E}}$ is isomorphic to $\bZ^2$ with the zero form, 
\item $\mathrm{K}$ is the kernel of the spherical homomorphism $f \colon \mathrm{G} \to \mathrm{E}$ obtained by gluing $\mathrm{K}_0^{\mathrm{num}}(\mathbf{D}(V_1))$ and $\mathrm{K}_0^{\mathrm{num}}(\mathbf{D}(V_2))$ along $\mathrm{E} \cong \mathrm{K}_0^{\mathrm{num}}(\mathbf{D}(C))$, as in Section \ref{sec:pseudolattice}, and
\item the bilinear form on $H^0_{\lim}(X) \oplus I^{\perp}/I \oplus H^4_{\lim}(X)$ is induced by the bilinear form on $H^{\mathrm{even}}_{\lim}(X) \cong H^{\mathrm{even}}(X;\bZ)$ defined by Equation \eqref{eq:Eulermatrix}.
\end{itemize}
\end{proposition}

\subsection{Point-like vectors and the N\'{e}ron-Severi lattice}

Next we ask how the ideas of Section \ref{sec:pointsandNS} arise in this setting.

In this setting, by \cite[Example 3.10]{pdpslf}, $\Psi$ is the sublattice of $\mathrm{K}$ generated by the classes of $(r_1(a),0) = ([\calO_{p_1}],0)$ and $(0,r_2(a)) = (0,[\calO_{p_2}])$, for $p_i$ a point on $X_i$. Using Notation \ref{not:6tuple}, $\Psi$ is generated by $(0,0,1,0,0,0)$ and $(0,0,0,0,0,1)$. By Lemma \ref{lem:psiperpG}, the lattice isomorphisms $\NS(\rmG_i) \cong \NS(V_i) \cong H^2(V_i;\bZ)$ induce an isomorphism of lattices
\[\Psi^{\perp}_{\rmG}/\Psi \cong H^2(V_1;\bZ) \oplus H^2(V_2;\bZ).\]
Using the description of the bilinear form on $\rmG$ in the proof of Lemma \ref{lem:K}, we see that this induces an isomorphism of lattices
\[\Psi^{\perp}_{\rmK}/\Psi \cong \{(D_1,D_2) \in H^2(V_1;\bZ) \oplus H^2(V_2;\bZ) \mid q(K_{V_1},D_1) = q(K_{V_2},D_2)\} = \xi^{\perp},\]
where $\Psi^{\perp}_{\rmG}/\Psi$ and $\Psi^{\perp}_{\rmK}/\Psi$ are equipped with the bilinear forms from Lemma \ref{lem:psiperpG} and Proposition \ref{prop:NS}.

\begin{remark} Note that these isomorphisms are \emph{not} the ones induced by the map $\kappa$ coming from the Chern character (see Lemma \ref{lem:K}). Indeed, by Lemma \ref{lem:KoverZ}, the image of $\kappa$ may even contain some non-integral classes. Instead, they may be thought of as a composition of $\kappa$ with a projection to the subspace $H^2(V_1;\bZ) \oplus H^2(V_2;\bZ)$.
\end{remark}

The map to $\rmM \cong H^0_{\lim}(X) \oplus I^{\perp}/I \oplus H^4_{\lim}(X)$ takes the generators $(0,0,1,0,0,0)$ and $(0,0,0,0,0,1)$ of $\Psi$ to a generator $\bp \in H^4_{\lim}(X)$. By the description of the bilinear form on $\rmM$ from Lemma \ref{lem:Lagrees}, we see that the natural isomorphism  $\mu\colon \rmM \to H^0_{\lim}(X) \oplus I^{\perp}/I \oplus H^4_{\lim}(X)$ induces an isomorphism $\NS(\rmM) = \bp^{\perp}/\bp \cong I^{\perp}/I$. 

Finally, by Lemma \ref{lem:zetaisotrivial}, we see that the kernel of the map $\Psi^{\perp}_{\rmK}/\Psi \to \NS(\rmM)$ is generated by $\zeta = (-K_{V_1},K_{V_2}) = (C,-C) = \xi$. Putting things together, we thus obtain the following result.

\begin{proposition} \label{prop:NSfordegenerations}
If $X_0 = V_1 \cup_C V_2$ is a stable K3 surface of type II, then the lattice isomorphisms $\NS(\rmG_i) \cong \NS(V_i)$ induce an isomorphism of lattices
\[\Psi^{\perp}_{\rmG}/\Psi \cong H^2(V_1;\bZ) \oplus H^2(V_2;\bZ).\] 
If $X_0$ is the central fibre of a Tyurin degeneration $\pi \colon \calX \to \Delta$, then the isomorphism above induces an isomorphism between the exact sequence 
\[0 \longrightarrow \bZ \zeta \longrightarrow \Psi^{\perp}_{\rmK}/\Psi \longrightarrow \NS(\rmM) \longrightarrow 0\] 
from Proposition \ref{prop:NS} and the the exact sequence
\[0 \longrightarrow \bZ \xi \longrightarrow \xi^{\perp} \longrightarrow I^{\perp}/I \longrightarrow 0\]
from Proposition \ref{prop:CSoverZ}. From Lemma \ref{lem:HE8E8}, it follows that $\NS(\rmM) \cong H \oplus E_8 \oplus E_8$.
\end{proposition}

\subsection{Lattice polarisations}\label{sec:degenerationpolarisation}

We now consider how the constructions we have discussed so far work in the presence of a lattice polarisation. We will make use of the following standard definition.

\begin{definition} Let $L$ be a non-degenerate lattice. A class $\delta \in L$ with $\langle \delta,\delta\rangle_L = -2$ is called a \emph{root}. A subset $\Delta(L)^+$ of \emph{positive roots} is a set of roots with the following properties:
\begin{itemize}
\item for every root $\delta \in L$, we either have $\delta \in \Delta(L)^+$ or $-\delta \in \Delta(L)^+$, but not both;
\item if $\delta$ is a root that can be written as a non-negative integer combination of classes from $\Delta(L)^+$, then $\delta \in \Delta(L)^+$.
\end{itemize}
\end{definition}

Now suppose that $L$ is a non-degenerate primitive sublattice of the K3 lattice $\Lambda_{\text{K3}} \cong H \oplus H \oplus H \oplus E_8 \oplus E_8$ of signature $(1,\rank(L) -1)$. We next define the notion of a lattice polarisation on a K3 surface, following Alexeev's and Engel's \cite{cmk3s} modification of Dolgachev's \cite{mslpk3s} original definition.

Begin by fixing a vector $h \in L \otimes \bR$, such that $\langle h,h \rangle_L > 0$ and $h \notin L' \otimes \bR$ for any proper sublattice $L' \subsetneq L$ with $\rank(L') < \rank(L)$; Alexeev and Engel call such vectors \emph{very irrational}. 

\begin{definition} \cite[Definition 2.6]{cmk3s} \label{def:Lquasipol} An \emph{$L$-quasi\-polarised K3 surface} is a K3 surface $X$ along with a primitive lattice embedding $j\colon L \hookrightarrow \NS(X)$ for which $j(h)$ is nef and big.
\end{definition}

Let $\Delta(L)^+$ denote the set of classes
\[\Delta(L)^+ := \{\delta \in L \mid \langle \delta,\delta \rangle_L = -2,\ \langle \delta,h\rangle_L > 0\}.\]
The fact that $h$ is very irrational implies that $\Delta(L)^+$ is a set of positive roots for $L$. Moreover, it follows from \cite[Remark 1.1]{mslpk3s} that if $X$ is an $L$-quasi\-polarized K3 surface, then $j(\Delta(L)^+) \subset \NS(X)$ is precisely the set of classes of effective $(-2)$-curves contained in $j(L)$.


The following definition is based on \cite[Definition 2.1]{flpk3sm}, appropriately modified for compatibility with Definition \ref{def:Lquasipol}. Compare also Alexeev's and Engel's definition of a \emph{divisor model} \cite[Definition 3.11]{cmk3s} in the presence of a lattice polarisation.

\begin{definition} \label{def:polarisedtyurin}
A \emph{$L$-polarised Tyurin degeneration of K3 surfaces} is a Tyurin degeneration $\pi \colon \calX \to \Delta$ such that
\begin{enumerate}
\item (monodromy invariance) over the punctured disc $\Delta^*$ there is a trivial local subsystem $\calL$ of $R^2\pi_*\bZ$ with fibre $L$, so that for each $p \in \Delta^*$ the fibre $\calL_p \subset H^2(X_p;\bZ)$ defines an $L$-quasi\-polarisation of $X_p$,
\item (positivity) there exists an effective, flat, and nef Cartier divisor $A$ on $\calX$  such that the restriction $A_p$ of $A$ to any fibre $X_p$ over $p \in \Delta^*$ is nef and big and satisfies $[A_p] \in \calL_p$. Moreover, the restriction of $A$ to $V_i$ satisfies $(A|_{V_i})^2 > 0$ for each $i \in \{1,2\}$. 
\end{enumerate}
\end{definition}

\begin{remark} \label{rem:SBalgorithm} We note that, given a Tyurin degeneration  $\pi \colon \calX \to \Delta$ along with a Cartier divisor $A$ on $\calX$  such that the restriction $A_p$ of $A$ to any fibre $X_p$ over $p \in \Delta^*$ is nef and big and satisfies $[A_p] \in \calL_p$, by \cite[Theorem 1]{bgd4} one may twist $A$ by a Cartier divisor $Z$ supported on components of $X_0$ and perform a sequence of flops to arrange that $A$ is effective, flat, and nef on $\calX$. Moreover, by \cite[Theorem 2]{bgd4}, after possibly replacing $A$ by a suitable multiple and performing a further twist and sequence of flops, one may additionally arrange that $(A|_{V_i})^2 > 0$ for each $i \in \{1,2\}$.
\end{remark}


The following lemma relates Definition \ref{def:polarisedtyurin} to the description in terms of pseudolattices.

\begin{lemma}\label{lem:tyurinpseudolattice} Let $\pi \colon \calX \to \Delta$ be an $L$-polarised Tyurin degeneration of K3 surfaces and let $0 \subset I \subset I^{\perp} \subset H^2(X;\bZ)$
denote the associated limiting mixed Hodge structure. Then $L$ is orthogonal to $I$ and there is an embedding $L \hookrightarrow I^{\perp}/I$.
\end{lemma}
\begin{proof} From the definition of the limiting mixed Hodge structure, we have $I^{\perp} = \ker(N) \subset H^2(X;\bZ)$, where $N$ is the logarithm of the monodromy. As $L \subset H^2(X;\bZ)$ is primitive and monodromy invariant, we have that $L$ is a primitive sublattice of $\ker(N) = I^{\perp}$; in particular we see that $L$ is orthogonal to $I$. Moreover, $I \subset I^{\perp}$ is a totally degenerate sublattice and $L$ is nondegenerate, so $L \cap I = 0$. It follows that $L$ embeds into $I^{\perp}/I$.
\end{proof}

\begin{definition} \label{def:primitivedegeneration} A Tyurin degeneration of K3 surfaces is called \emph{primitively $L$-polarised} if it is $L$-polarised and the embedding $L \hookrightarrow I^{\perp}/I$ from Lemma \ref{lem:tyurinpseudolattice} is primitive.
\end{definition}

\begin{remark} It is not the case that all $L$-polarised Tyurin degenerations of K3 surfaces are primitively $L$-polarised. Indeed, from Shimada's list of elliptically fibred K3 surfaces \cite{oek3s} there exists an elliptically fibred K3 surface $X$ with N\'{e}ron-Severi lattice $H \oplus D_{16}$, which embeds primitively into $H^2(X;\bZ)$ (see \cite{oek3sarxiv} for the full list, where this case is number 1903 with trivial Mordell-Weil group). Moreover, by \cite[Theorem 1.14.4]{isbfa}, the primitive embedding $H \oplus D_{16} \hookrightarrow \Lambda_{\rmK 3}$ is unique up to isometry, so it follows from \cite[Corollaries 1.6.2 and 1.13.3]{isbfa} that its orthogonal complement is uniquely determined by its signature and discriminant form.  Consequently, one finds that the orthogonal complement of $H \oplus D_{16}$ inside $H^2(X;\bZ)$ is isometric to $H \oplus H(2)$.  

The lattice $H \oplus H(2)$ contains a rank two isotropic sublattice $I$, generated by one primitive isotropic vector from each of the direct summands, which determines a $1$-cusp in the Baily-Borel compactification of the moduli space of $H \oplus D_{16}$-polarised K3 surfaces (see Section \ref{sec:moduli}). A family of Tyurin degenerations corresponding to this $1$-cusp may be constructed following the approach of \cite[Sections 7.4 and 7.5]{cmk3s}, which may be turned into $H \oplus D_{16}$-polarised Tyurin degenerations by choosing a relatively nef and big Cartier divisor over the punctured disc and following the method from  Remark \ref{rem:SBalgorithm}.

However, such Tyurin degenerations cannot be primitively $H \oplus D_{16}$-polarised, as $H \oplus D_{16}$ does not embed primitively into $I^{\perp}/I \cong H \oplus E_8 \oplus E_8$; instead it is a sublattice of index $2$. Some conditions that are sufficient to ensure primitivity are explored in Section \ref{sec:primitivity}.
\end{remark}

Following the discussion of the previous section, composing the embedding $L \hookrightarrow \NS(\rmM)$ from Lemma \ref{lem:tyurinpseudolattice} with  the isomorphism $\NS(\rmM) \cong I^{\perp}/I$ gives an embedding $L \hookrightarrow \NS(\rmM)$. If the Tyurin degeneration is primitively $L$-polarised, then we obtain a lattice polarisation on $\NS(\rmM)$ in the sense of Definition \ref{def:Lpol}(1).

We next define the notion of a lattice polarisation on a  stable K3 surface of type II $X_0 = V_1 \cup_C V_2$; this definition generalises the notion of polarisation by a divisor from \cite[Definition 3.11]{npgttk3s}. Recall from Section \ref{sec:hodge} that we have an isomorphism
\[\Gr^W_2H^2(X_0) \cong \xi^{\perp} =\{(D_1,D_2) \in H^2(V_1) \oplus H^2(V_2) \mid q(K_{V_1},D_1) = q(K_{V_2},D_2)\},\]
where $\xi$ denotes the isotropic class $(-K_{V_1},K_{V_2}) = (C,-C)$.

%

\begin{definition}\label{def:liftedpolarisationstable}
An \emph{$\hat{L}$-polarisation} on a stable K3 surface of Type II is a primitive sublattice $\hat{L} \subset \xi^{\perp} \subset H^2(V_1;\bZ) \oplus H^2(V_2;\bZ)$ satisfying the following conditions:
\begin{enumerate}
\item $\xi \in \hat{L}$;
\item $\hat{L}$ contains the class of a nef Cartier divisor $A$ on $X_0$, such that $(A|_{V_i})^2 > 0$ for each $i \in \{1,2\}$;
\item For every $(-2)$-class $\delta \in \hat{L}$, either $\delta$ or $-\delta$ is congruent modulo $\xi$ to the class of an effective Cartier divisor of self-intersection $(-2)$.
\end{enumerate}
\end{definition}

Following the discussion of the previous section, an $\hat{L}$-polarisation on a stable K3 surface of Type II defines a primitive embedding $\hat{L} \hookrightarrow \xi^{\perp} \cong \Psi^{\perp}_{\rmK}/\Psi$ and the condition that $\xi \in \hat{L}$ corresponds precisely to the condition that the image of this embedding contains the class $\zeta$. We therefore obtain a lifted polarisation on $\Psi^{\perp}_{\rmK}/\Psi$ in the sense of Definition \ref{def:Lpol}(2).

We can relate this to lattice polarised Tyurin degenerations as follows. Let $\pi \colon \calX \to \Delta$ be a primitively $L$-polarised Tyurin degeneration of  K3 surfaces. Then $L$ is a primitive sublattice of $I^{\perp}/I \cong \NS(\rmM)$. Let $\hat{L}$ denote the preimage of $L$ under the map from $\xi^{\perp} \cong \Psi_{\rmK}^{\perp}/\Psi$. Then $\hat{L}$ is the lifted polarisation determined by $L$, in the sense of Lemma \ref{lem:liftinglattices}, and we have the following result.

\begin{theorem} \label{thm:polarisedtyurinstable} With notation as above, if $\pi \colon \calX \to \Delta$ is a primitively $L$-polarised Tyurin degeneration of  K3 surfaces, then the central fibre $X_0$ of $\pi \colon \calX \to \Delta$ is an $\hat{L}$-polarised stable K3 surface of Type II.
\end{theorem}
\begin{proof}  The discussion above shows that there is a primitive embedding of  $\hat{L}$ into $\xi^{\perp}$. The image of this embedding contains $\xi$, as $\xi$ generates the kernel of the map $\xi^{\perp} \to I^{\perp}/I$. Denote this map by $\varphi$.

Let $A$ denote the Cartier divisor from Definition \ref{def:polarisedtyurin}(2). Then the restriction of $A$ to the central fibre $X_0$ is a nef Cartier divisor that lies in  $\hat{L}$, with $(A|_{V_i})^2 > 0$ for each $i \in \{1,2\}$

Finally, Definition \ref{def:polarisedtyurin}(1) implies that elements of $L$ give rise to line bundles on $\calX \setminus X_0$ and the map $\varphi\colon \hat{L} \to L$ corresponds to extensions of these to line bundles on $\calX$. Such extensions are defined up to twisting by $\calO_{\calX}(V_2) \cong \calO_{\calX}(-V_1)$, corresponding to the fact that lifts from $L$ to $\hat{L}$ are only defined up to multiples of $\xi$. Now let $\delta \in \hat{L}$ be a $(-2)$-class and let $\calD \to \calX$ denote the corresponding line bundle. The restriction $\calD|_{X_p}$ defines the $(-2)$-class $\varphi(\delta) \in L \hookrightarrow \NS(X_p)$ for each $p \in \Delta^*$ and, by replacing $\delta$ by $-\delta$ as necessary, we may assume that this class lies in $\Delta(L)^+$ and so corresponds to an effective $(-2)$-curve on $X_p$. We therefore have $h^0(X_p, \calD|_{X_p}) >0$ for all $p \in \Delta^*$ and by semicontinuity it follows that $h^0(X_0, \calD|_{X_0}) >0$.  Then an argument from cohomology and base change shows that $\calD$ corresponds to an effective Cartier divisor on $\calX$. By twisting $\calD$ by an appropriate multiple of $\calO_{\calX}(V_2) \cong \calO_{\calX}(-V_1)$ one may remove any components of $X_0$, leaving a Cartier divisor which is effective and flat over $\Delta$. The restriction of this divisor to $X_0$ is an effective Cartier divisor of self-intersection $(-2)$, which is congruent modulo $\xi$ to either $\delta$ or $-\delta$, as required.
\end{proof}


We conclude this section by showing that the definitions of lattice polarisation above are compatible with the definition of a lattice polarisation on a weak del Pezzo surface from \cite{mslpdps}. 

\begin{definition} \label{def:Lpolwdp} \cite[Definition 4.4]{mslpdps} Let $V$ be a weak del Pezzo surface and let $C \subset V$ be a smooth anticanonical divisor. Let $N$ be a negative definite lattice. An \emph{$N$-polarisation} of $(V,C)$ is a primitive embedding $j\colon N \hookrightarrow \NS(V)$ such that
\begin{itemize}
\item $j(v).C = 0$ for all $v \in N$, and
\item there exists a set of positive roots $\Delta(N)^+ \subset N$ such that $j(\Delta(N)^+)$ is contained in the effective cone in $\NS(V)$.
\end{itemize}
\end{definition}

The following result shows that this is compatible with our definitions.

\begin{theorem} \label{thm:Lpolwdp} Let $X_0 = V_1 \cup_C V_2$  be an $\hat{L}$-polarised stable K3 surface of Type II and suppose that $V_1$ is weak del Pezzo. Then the intersection polarisation $L_1 := \hat{L} \cap \NS(V_1)$ defines an $L_1$-polarisation of $(V_1,C)$, in the sense of Definition \ref{def:Lpolwdp}.
\end{theorem}

\begin{proof} The condition that $j(v).C = 0$ for all $v \in L_1$ follows from Lemma \ref{lem:Kperp}, and negative definiteness follows from the fact that $C^{\perp} = K_{V_i}^{\perp} \subset \NS(V_1)$ is negative definite for any weak del Pezzo $V_1$. To show the second condition from Definition \ref{def:Lpolwdp}, it suffices to show that if $\delta \in L_1$ is a class of self-intersection $(-2)$, then either $\delta$ or $(-\delta)$ is effective. 

So let $\delta \in L_1$ be a class of self-intersection $(-2)$. Write $\delta = (v,0) \in \NS(V_1)\oplus \NS(V_2)$. After possibly exchanging $\delta$ for $(-\delta)$, by Definition \ref{def:liftedpolarisationstable}, there exists $k \in \bZ$ such that $\delta - k \xi = (v-kC,kC)$ is effective. As $C$ is an effective curve in $\NS(V_2)$, this can only occur if $k \geq 0$. But then $v = (v-kC)+kC \in \NS(V_1)$ is a sum of effective classes, so is effective.
\end{proof}

\begin{remark} The condition that $V_1$ is weak del Pezzo holds if and only if $\deg(\rmG) = q(K_{V_1},K_{V_1}) > 0$. As our labelling convention ensures that $q(K_{V_1},K_{V_1}) \geq 0$ and we have $q(K_{V_1},K_{V_1}) = -q(K_{V_2},K_{V_2})$ by Remark \ref{rem:canonicalsquare}, $V_2$ can never be weak del Pezzo, as the canonical divisor on any weak del Pezzo surface always has positive self-intersection.
\end{remark}

Putting the results of this section together we obtain the following.  Let $\pi \colon \calX \to \Delta$ be a primitively $L$-polarised Tyurin degeneration of  K3 surfaces. Then $L$ is a primitive sublattice of $I^{\perp}/I \cong \NS(\rmM)$. If  $\hat{L} \subset \xi^{\perp} \cong \Psi_{\rmK}^{\perp}/\Psi$ is the lifted polarisation induced by $L$, in the sense of Definition \ref{def:Lpol}(2), then by Theorem \ref{thm:polarisedtyurinstable} the central fibre $X_0 = V_1 \cup_C V_2$ of $\pi \colon \calX \to \Delta$ is an $\hat{L}$-polarised stable K3 surface of Type II. Finally, assuming that $V_1$ is weak del Pezzo, if $L_1 \subset \NS(V_1) \cong \NS(\rmG_1)$ is the intersection polarisation, in the sense of Definition \ref{def:Lpol}(3), then by Theorem \ref{thm:Lpolwdp} we obtain an $L_1$-polarisation of $(V_1,C)$, in the sense of \cite[Definition 4.4]{mslpdps}.


\subsection{Primitivity}\label{sec:primitivity}

In this section we explore sufficient conditions under which an $L$-polarised  Tyurin degeneration of K3 surfaces is primitively $L$-polarised. As in the previous section, we let $L$ be a non-degenerate primitive sublattice of the K3 lattice $\Lambda_{\text{K3}} \cong H \oplus H \oplus H \oplus E_8 \oplus E_8$ of signature $(1,\rank({L}) -1)$. Denote the bilinear form on $\Lambda_{\rmK 3}$ by $\langle \cdot,\cdot \rangle_{\rmK 3}$ and the orthogonal complement of $L$ inside $\Lambda_{\text{K3}}$ by $L^{\perp}$.

The following definition is due to Dolgachev \cite{mslpk3s}.

\begin{definition} \label{def:madmissible} Let $L$ be a nondegenerate lattice and $e \in L$ be a primitive isotropic element. Define $\divop(e)$ to be the positive generator of the image of the linear map $\langle e, - \rangle_L \colon L \to \bZ$. Then $e$ is called \emph{$m$-admissible in $L$} if $\divop(e) = m$ and there exists an isotropic class $g \in L$ with $\divop(g) = m$ and $\langle e,g\rangle_{L} = m$.
\end{definition}

Let $\pi \colon \calX \to \Delta$ be an $L$-polarised Tyurin degeneration of K3 surfaces. By Lemma \ref{lem:tyurinpseudolattice} we have that $I$ is a primitive sublattice of $L^{\perp}$. This allows us to state the following definition.

\begin{definition}\label{def:doublyadmissible} Let $L$ be a nondegenerate lattice and let $I \subset L$ be a primitive rank two isotropic sublattice. $I$ is called \emph{doubly admissible in $L$} if there exist primitive generators $e_1$ and $e_2$ for $I$ and isotropic classes $g_1,g_2 \in L$ such that $\langle e_1,g_1 \rangle_{L} = \langle e_2,g_2 \rangle_{L} = 1$ and $\langle e_1,g_2 \rangle_{L} = \langle e_2,g_1 \rangle_{L} = 0$.

An $L$-polarised  Tyurin degeneration of K3 surfaces is called \emph{doubly admissible} if $I$ is doubly admissible in $L^{\perp}$.
\end{definition}

The following lemma gives some equivalent conditions for a rank two isotropic sublattice to be doubly admissible.

\begin{lemma} \label{lem:doublyadmissibleconditions} Let $L$ be a nondegenerate primitive sublattice of $\Lambda_{\rmK 3}$ and let $I$ be a primitive rank two isotropic sublattice of $L^{\perp}$. The following are equivalent:
\begin{enumerate}
\item $I$ is doubly admissible in $L^{\perp}$.
\item The primitive embedding $I \hookrightarrow L^{\perp}$ factors through primitive embeddings $I \hookrightarrow H \oplus H \hookrightarrow L^{\perp}$.
\item The primitive embedding $L \hookrightarrow I^{\perp}$ factors through primitive embeddings $L \hookrightarrow H \oplus E_8 \oplus E_8 \hookrightarrow I^{\perp}$.
\end{enumerate}
\end{lemma}

\begin{proof} $(1) \Rightarrow (2)$ It follows from the definition that $e_1$ is $1$-admissible in $L^{\perp}$. So, by \cite[Lemma 5.4]{mslpk3s}, we have a primitive embedding $i\colon H \hookrightarrow L^{\perp}$ such that $e_1$ is one of the generators of $i(H)$, and we may decompose  $L^{\perp} = i(H) \oplus (i(H))^{\perp}_{L^{\perp}}$. Now define $g_2' = g_2 - \langle g_1,g_2 \rangle_{\rmK 3} e_1$.
Then it is easy to show that $g_2' \in (i(H))^{\perp}_{L^{\perp}}$ is isotropic with $\langle e_2,g_2'\rangle_{\rmK 3} = 1$, so $e_2$ is $1$-admissible in $(i(H))^{\perp}_{L^{\perp}}$. We may therefore apply \cite[Lemma 5.4]{mslpk3s} again to obtain a primitive embedding $j\colon H \hookrightarrow (i(H))^{\perp}_{L^{\perp}}$, where $e_2$ is one of the generators of $j(H)$. Putting everything together gives the required factorisation of $I \hookrightarrow L^{\perp}$.

$(2) \Rightarrow (1)$ Let $\{e_1,e_2\}$ denote a set of primitive generators for $I \subset H \oplus H$. By \cite[Lemma 4.1.2]{cmsak3s}, primitive isotropic vectors in $H \oplus H$ are unique up to automorphism, so we have $e_1^{\perp}/e_1 \cong H$ (where the orthogonal complement is taken in $H \oplus H$). The vector $e_2$ gives rise to a primitive isotropic vector in $e_1^{\perp}/e_1$, so can be taken to be one of the standard generators for $H$. Let $g_2 \in e_1^{\perp}$ be a vector that gives rise to the other standard generator of $H$; then by construction $g_2$ satisfies $\langle e_1,g_2\rangle_{\rmK 3} = 0$, $\langle e_2,g_2\rangle_{\rmK 3} = 1$, and $\langle g_2,g_2 \rangle_{\rmK 3} = 0$.

Now, since $e_1$ is primitive, there exists $f_1 \in H \oplus H$ such that $\langle e_1,f_1 \rangle_{\rmK 3} = 1$. Define
\[g_1 = f_1 - \langle f_1,e_2\rangle_{\rmK 3} g_2 - \tfrac{1}{2}\langle f_1,f_1 \rangle_{\rmK 3} e_1 + \langle f_1,e_2\rangle_{\rmK 3} \langle f_1,g_2\rangle_{\rmK 3}e_1,\]
which is well-defined as $H \oplus H$ is an even lattice. Then $g_1$ satisfies $\langle e_1,g_1\rangle_{\rmK 3} = 1$, $\langle e_2,g_1\rangle_{\rmK 3} = 0$, and $\langle g_1,g_1 \rangle_{\rmK 3} = 0$, so the classes $\{e_1,e_2,g_1,g_2\}$ satisfy the conditions of Definition \ref{def:doublyadmissible}.

$(2) \Leftrightarrow (3)$ By \cite[Proposition 1.6.1]{isbfa}, up to isometry there are unique primitive embeddings $H \oplus H \hookrightarrow \Lambda_{\text{K3}}$ and $H \oplus E_8 \oplus E_8 \hookrightarrow \Lambda_{\text{K3}}$. Therefore the orthogonal complements $(H \oplus H)^{\perp}_{\Lambda_{\text{K3}}}$ and $(H \oplus E_8 \oplus E_8)^{\perp}_{\Lambda_{\text{K3}}}$ do not depend upon the choice of embedding, and are always isomorphic to $H \oplus E_8 \oplus E_8$ and $H \oplus H$ respectively. Given this, equivalence of (2) and (3) is obtained by taking orthogonal complements.
\end{proof}

The next proposition gives some consequences of the doubly admissible condition.

\begin{proposition} \label{prop:doublyadmissibleconsequnces} Let $L$ be a nondegenerate primitive sublattice of $\Lambda_{\rmK 3}$ and let $I$ be a doubly admissible rank two isotropic sublattice of $L^{\perp}$. Then:
\begin{enumerate}
\item The composition $H \oplus E_8 \oplus E_8 \hookrightarrow I^{\perp} \to I^{\perp}/I$ of the primitive embedding from Lemma \ref{lem:doublyadmissibleconditions}(3) and the quotient map is an isomorphism. In particular, the primitive embedding $L \hookrightarrow H \oplus E_8 \oplus E_8$ from Lemma \ref{lem:doublyadmissibleconditions}(3) induces a  primitive embedding 
 $L \hookrightarrow I^{\perp}/I$.
\item There is an orthogonal decomposition $L^{\perp} \cong H \oplus H \oplus \Gamma$, where the embedding $H \oplus H \hookrightarrow L^{\perp}$ is the one from Lemma \ref{lem:doublyadmissibleconditions} and $\Gamma$ is a negative definite lattice.
\end{enumerate}
In particular, if $\pi \colon \calX \to \Delta$ is a doubly admissible $L$-polarised Tyurin degeneration of  K3 surfaces, then $\pi \colon \calX \to \Delta$ is primitively $L$-polarised.
\end{proposition}

\begin{proof} \begin{enumerate}
\item This follows from the proof of $(2) \Leftrightarrow (3)$ in Lemma \ref{lem:doublyadmissibleconditions}. It follows from that argument that we may decompose $\Lambda_{\text{K3}} = (H \oplus H) \oplus (H \oplus E_8 \oplus E_8)$, with $I$ embedding primitively into the first factor $H \oplus H$ and $L$ embedding primitively into the second factor $H \oplus E_8 \oplus E_8$. The quotient $I^{\perp}/I$ is then isomorphic to the second factor $H \oplus E_8 \oplus E_8$ and the result follows.
\item Follows from \cite[Lemma 5.4]{mslpk3s} and the proof of $(1) \Rightarrow (2)$ in Lemma \ref{lem:doublyadmissibleconditions}.
\end{enumerate}
\end{proof}

\subsection{Moduli spaces}\label{sec:moduli}

As in the previous sections, we let $L \subset \Lambda_{\text{K3}}$ be a non-degenerate primitive sublattice of the K3 lattice of signature $(1,\rank(L) -1)$ and fix a vector $h \in L \otimes \bR$ such that $\langle h,h \rangle_L > 0$ and $h \notin L' \otimes \bR$ for any proper sublattice $L' \subsetneq L$. Following \cite[Section 2.3]{cmk3s}, we define a moduli space of $L$-polarised K3 surfaces as follows.



\begin{definition} \cite[Definition 2.10]{cmk3s} An \emph{$L$-polarised K3 surface} is a surface $\overline{X}$ with at worst du Val (ADE) singularities, whose minimal resolution is a smooth K3 surface, together with an isometric embedding $j\colon L \hookrightarrow \Pic(\overline{X})$ for which $j(h)$ is ample.
\end{definition}

Note that one may obtain an $L$-polarised K3 surface $\overline{X}$ from an $L$-quasi\-polarised K3 surface $X$ by contracting all $(-2)$-curves in $j(L)^{\perp} \subset \NS(X)$.

Now define the period domain for $L$-polarised K3 surfaces to be
\[\calD_L := \bP\{v \in \Lambda_{\text{K3}} \otimes \bC \mid \langle v,v\rangle = 0,\ \langle v, \overline{v} \rangle > 0,\ \langle v, L \rangle = 0\}\]
and define $\Gamma_L$ to be the group of isometries of $\Lambda_{\text{K3}}$ that fix $L$,
\[\Gamma_L := \{\phi \in O(\Lambda_{\text{K3}}) \mid \phi(v) = v \text{ for any } v \in L\}.\]
By \cite[Proposition 3.3]{mslpk3s}, there is a natural injective homomorphism $\Gamma_L \hookrightarrow O(L^{\perp})$ which identifies $\Gamma_L$ with the subgroup
\[\widetilde{O}(L^{\perp}) := \ker\{O(L^{\perp}) \to O(\disc(L^{\perp}))\},\]
where $O(\disc(L^{\perp}))$ is the group of automorphisms of $\disc(L^{\perp})$ that preserve the quadratic form. With this setup, the coarse moduli space of $L$-polarised K3 surfaces is given by the quotient $\calD_L / \Gamma_L$.

This moduli space admits a Baily-Borel compactification. Boundary components in this compactification correspond to primitive isotropic sublattices of $L^{\perp}$ of rank $1$ ($0$-cusps) or $2$ ($1$-cusps), up to the action of $\Gamma_L \cong \widetilde{O}(L^{\perp})$ (see \cite[Definition 5.5]{cmk3s}). The following proposition shows how this fits with our existing theory.

\begin{proposition} \label{prop:polarisationIperp} Each $1$-cusp in the Baily-Borel compactification of $\calD_L / \Gamma_L$ uniquely determines an embedding of $L$ into $I^{\perp}/I$, up to the action of $O(I^{\perp}/I)$.
\end{proposition}

\begin{proof} By the discussion above, $1$-cusps in the Baily-Borel compactification are determined by first fixing an embedding $L \subset \Lambda_{\text{K3}}$, up to the action of $O(\Lambda_{\rmK 3})$, then choosing a rank $2$ isotropic sublattice $I \subset L^{\perp}$, up to the action of $\Gamma_L$. It is immediate from the definitions that this is equivalent to choosing a pair $(L,I)$ of sublattices of $\Lambda_{\text{K3}}$, up to the action of $O(\Lambda_{\text{K3}})$, such that $I$ is isotropic of rank $2$ and orthogonal to $L$.

Now we reverse the order in which we choose $L$ and $I$. Begin by fixing a rank $2$ isotropic sublattice $I$ of $\Lambda_{\text{K3}}$; by \cite[Section 2.2]{cmsak3s} this choice is unique up to $O(\Lambda_{\text{K3}})$. It follows from the discussion above that $1$-cusps in the Baily-Borel compactification then correspond to embeddings $L \subset I^{\perp}$, up to the action of 
\[\Stab(I) := \{\phi \in O(\Lambda_{\text{K3}}) \mid \phi(I) = I\}.\]

As $L \subset I^{\perp}$ is non-degenerate we have $L \cap I = \{0\}$, so $L$ descends to a sublattice of $I^{\perp}/I$. Any element of $\Stab(I)$ preserves $I$ and acts as an isometry on $I^{\perp}$, so descends to an element of $O(I^{\perp}/I)$. Moreover, the map $\Stab(I) \to O(I^{\perp}/I)$ is surjective, so every element of $O(I^{\perp}/I)$ arises in this way. We thus see that a choice of embedding $L \subset I^{\perp}$, up to the action of $\Stab(I)$, uniquely determines an embedding of $L$ into $I^{\perp}/I$, up to the action of $O(I^{\perp}/I)$.
\end{proof}

\begin{remark} The correspondence from Proposition \ref{prop:polarisationIperp} is not one-to-one: an embedding of $L$ into $I^{\perp}/I$, up to $O(I^{\perp}/I)$, does not necessarily uniquely determine a $1$-cusp in the Baily-Borel compactification of $\calD_L / \Gamma_L$; see \cite[Section 2.2]{cmsak3s}. Moreover, note that the definition of $\calD_L / \Gamma_L$ requires us to first fix an embedding $L \subset \Lambda_{\text{K3}}$ up to the action of $O(\Lambda_{\rmK 3})$; a priori there may be several embeddings $L \subset \Lambda_{\text{K3}}$ which are not equivalent up to $O(\Lambda_{\rmK 3})$, but where the corresponding embeddings of $L$ into $I^{\perp}/I$ are equivalent up to $O(I^{\perp}/I)$.
\end{remark}

In the Baily-Borel compactification of $\calD_L / \Gamma_L$, Type II degenerations occur over the $1$-cusps and Type III degenerations over the $0$-cusps. For an $L$-polarised Tyurin degeneration of K3 surfaces, the embedding of $L$ into $I^{\perp}/I$ appearing in Proposition \ref{prop:polarisationIperp} is precisely the one from Lemma \ref{lem:tyurinpseudolattice}.

In the remainder of this section we study the consequences of the doubly admissible condition for the structure of the Baily-Borel compactification. We begin with a lemma.

\begin{lemma} \label{lem:HHisometries}
Let $I, I' \subset H \oplus H$ be two primitive isotropic sublattices of rank $2$ and let $e \in I$, $e' \in I'$ be primitive isotropic vectors. Then there exists an isometry $g$ of $H \oplus H$ such that $g(I) = I'$ and $g(e) = e'$.
\end{lemma}
\begin{proof}
Let $e \in I \subset H \oplus H$ be as in the lemma and let $e_1,e_2,e_3,,e_4$ be a standard basis for $H \oplus H$. To prove the lemma, it suffices to show that there is an isometry $g$ of $H \oplus H$ such that $g(e) = e_1$ and $g(I)$ is the isotropic sublattice of rank $2$ generated by $e_1$ and $e_3$.

Indeed, by \cite[Lemma 4.1.2]{cmsak3s}, there exists an isometry $g$ of $H \oplus H$ with $g(e) = e_1$. Let $f \in I$ be a primitive element such that $e$ and $f$ generate $I$. Write $g(f) = \sum_{i=1}^4a_ie_i$. Then $\langle g(f), g(e) \rangle_{H \oplus H} = \langle g(f), e_1 \rangle_{H \oplus H}=  a_2 = 0$ and as $\langle g(f), g(f) \rangle_{H \oplus H} = a_3a_4 = 0$ we have $a_3 = 0$ or $a_4 = 0$. By applying an isometry swapping $e_3$ and $e_4$, if necessary, we may therefore assume that $g(f) = a_1 e_1 + a_3 e_3$. So $g(f-a_1e) = a_3e_3$ and, by primitivity of $g(I)$, we must have that $g(I)$ is the sublattice generated by $e_1$ and $e_3$, as required.
\end{proof}

\begin{proposition} \label{prop:doublyadmissible0cusp} Suppose that $I$ is doubly admissible in $L^{\perp}$. Then the $1$-cusp corresponding to $I$ in the Baily-Borel compactification of $\calD_L / \Gamma_L$ is incident to a unique $0$-cusp.
\end{proposition}
\begin{proof}
Let $e, e' \in I$ be any two primitive isotropic vectors. To show that the $1$-cusp corresponding to $I$ is incident to a unique $0$-cusp, it suffices to show that there is an element of $\Gamma_L \cong \widetilde{O}(L^{\perp})$ taking $e$ to $e'$.

By Proposition \ref{prop:doublyadmissibleconsequnces}, we have a decomposition $L^{\perp} = H \oplus H \oplus \Gamma$, where $I$ embeds primitively into the $H \oplus H$ factor. By Lemma \ref{lem:HHisometries}, there exists an isometry $g$ of $H \oplus H$ with $g(I) = I$ and $g(e) = e'$. We can extend $g$ to an isometry $\hat{g}$ of $L^{\perp}$ by the identity on $\Gamma$. As $\disc(L^{\perp}) = \disc(\Gamma)$ and $\hat{g}$ acts as the identity on $\Gamma$, it follows that $\hat{g} \in \widetilde{O}(L^{\perp})$, as required.
\end{proof}
 
Our final result gives an explicit description of the $1$-cusps in the Baily-Borel compactification which correspond to doubly admissible sublattices $I \subset L^{\perp}$. Here we assume that $L^{\perp}$ contains a $1$-admissible primitive isotropic element; clearly this is a necessary condition for $L^{\perp}$ to contain a doubly admissible sublattice. 
 
 \begin{proposition} Assume that $L^{\perp}$ contains a $1$-admissible primitive isotropic element.
 
Let $(\Gamma,i)$ be a pair consisting of a negative definite even lattice $\Gamma$ of rank $(18-\rank(L))$ along with an isomorphism $i\colon \disc(\Gamma) \to \disc(L^{\perp})$ of discriminant groups which respects the discriminant quadratic forms. We say that two such pairs $(\Gamma,i)$ and $(\Gamma', i')$ are equivalent if there exists an isometry $f \colon \Gamma \to \Gamma'$ such that the composition $i' \circ \overline{f} \circ i^{-1}$ is the identity on $\disc(L^{\perp})$, where $\overline{f} \colon \disc(\Gamma) \to \disc(\Gamma')$ is the isomorphism induced by $f$.
 
Then there is a bijection between $1$-cusps in the Baily-Borel compactification of $\calD_L / \Gamma_L$ corresponding to doubly admissible $I \subset L^{\perp}$ and equivalence classes of pairs $(\Gamma,i)$ as above.
 \end{proposition}
 
 \begin{proof}
 As $L^{\perp}$ contains a $1$-admissible element, by \cite[Lemma 5.4]{mslpk3s}, $L^{\perp}$ can be decomposed as a direct sum of $H$ and another lattice. It then follows from \cite[Corollary 1.13.4]{isbfa} that $L^{\perp}$ is uniquely determined by its signature and discriminant form.
 
From Lemmas \ref{lem:doublyadmissibleconditions} and \ref{lem:HHisometries} and the proof of Proposition \ref{prop:doublyadmissible0cusp}, it follows that  $1$-cusps in the Baily-Borel compactification corresponding to doubly admissible $I \subset L^{\perp}$ are in bijection with embeddings of $H \oplus H \hookrightarrow L^{\perp}$ up to the action of $\widetilde{O}(L^{\perp})$. As $L^{\perp}$ is uniquely determined by its signature and discriminant form, \cite[Proposition 1.15.1]{isbfa} implies that such embeddings are in bijection with pairs $(\Gamma,i)$ as in the statement of the proposition (the lattices $\Gamma$ give the possible orthogonal complements of $H \oplus H$ inside $L^{\perp}$). Moreover, by the same result, the embeddings defined by $(\Gamma,i)$ and $(\Gamma',i')$ are equivalent up to the action of $O(L^{\perp})$ if there exist an isometry $f \colon \Gamma \to \Gamma'$ and an isomorphism $g\colon \disc(L^{\perp}) \to \disc(L^{\perp})$ such that $i'\circ \overline{f} = g \circ i$. The proposition follows by noting that restricting the action to lie in $\widetilde{O}(L^{\perp})$ corresponds to requiring $g$ to be the identity.
 \end{proof}

\section{Elliptic Fibrations on K3 Surfaces}\label{sec:fibrations}

\subsection{Pseudolattices from elliptic fibrations over discs}\label{sec:fibrationsondiscs}

Before discussing the setting of elliptic fibrations on K3 surfaces, we begin with a simpler setting. The definitions in this section broadly follow those in \cite[Section 4]{pdpslf}, although we note that we do not work in precisely the same setting: in particular, \cite[Section 4]{pdpslf} assumes that all singular fibres have Kodaira type $\rmI_1$.

 Let $\Delta := \{z \in \bC: |z| \leq 1\}$ denote the closed complex unit disc and let $\pi_W\colon W \to \Delta$ denote a proper surjective holomorphic map from an oriented nonsingular compact complex surface (with boundary) $W$ to $\Delta$ whose general fibre is a nonsingular elliptic curve. Assume that $\pi_W$ is relatively minimal and has no multiple fibres. Let $\Sigma_W \subset \Delta$ denote the set of points over which the fibres of $W$ are singular; we assume that $\Sigma_W$ is finite and contained in the interior of $\Delta$, so there are no singular fibres over points on the boundary. We call such $\pi_W\colon W \to \Delta$ satisfying these assumptions \emph{elliptic fibrations over discs}.

Let $p \in \partial \Delta$ denote a point on the boundary of $\Delta$ and let $F = \pi^{-1}_W(p)$ denote the fibre over $p$. We consider the relative homology $H_2(W,F;\bZ)$. There is a natural map $\phi\colon H_2(W,F;\bZ) \to H_1(F;\bZ)$ coming from the long exact sequence of a pair; we call this map the \emph{asymptotic charge}. Recall from Section \ref{sec:background} that $H_1(F;\bZ)$, equipped with the usual intersection form, can be identified with the elliptic curve pseudolattice $\mathrm{E}$.

We may define a bilinear form, called the \emph{Seifert pairing}, on $H_2(W,F;\bZ)$. Choose $\varepsilon > 0$ so that $\Sigma \subset \Delta_{\varepsilon} := \{z \in \bC: |z| \leq 1 - \varepsilon\}$ (i.e. there are no singular fibres in an $\varepsilon$-neighbourhood of $\partial \Delta$). Fix a group of diffeomorphisms $\varphi_{\theta}\colon \Delta \to \Delta$, for each $\theta \in \bR$, such that $\varphi$ fixes the disc $\Delta_{\varepsilon}$ and acts as a clockwise rotation through an angle of $\theta$ on $\partial \Delta$. Let $\psi_{\theta}$ denote a family of lifts of these diffeomorphisms to diffeomorphisms of $W$.

\begin{definition}(c.f. \cite[Definition 4.2]{pdpslf}) The \emph{Seifert pairing}
\[\langle \cdot,\cdot\rangle_{\mathrm{Sft}}\colon H_2(W,F;\bZ) \times H_2(W,F;\bZ) \to \bZ\]
is the map sending $(v,w)$ to the negative of the topological intersection product $-\langle v,(\psi_{\pi})_*w\rangle_{\mathrm{top}}$.
\end{definition}

\begin{remark} We note that there is a sign error in \cite[Definition 4.2]{pdpslf}; the negative sign on the topological intersection form is missing. This does not affect the remaining results of \cite{pdpslf}; in particular, the proof of \cite[Theorem 4.14]{pdpslf} uses the correct sign convention, in which the Seifert pairing has the opposite sign to the bilinear form defined by Auroux in the proof of \cite[Lemma 16]{sclf}.
\end{remark}

The first main result of this section is the following.

\begin{proposition}\label{prop:H2pseudolattice} Let $\pi_W\colon W \to \Delta$ be an elliptic fibration over a disc. Then with notation as above:
\begin{enumerate}
\item $H_2(W,F;\bZ)$, equipped with the Seifert pairing, is a unimodular pseudolattice.
\item The Serre operator on $H_2(W,F;\bZ)$ is given by $(\psi_{-2\pi})_*$.
\item The asymptotic charge map $\phi\colon H_2(W,F;\bZ) \to H_1(F;\bZ)$ is a relative $(-1)^0$-Calabi-Yau spherical homomorphism.
\item The twist map $T_{\phi}\colon H_1(F;\bZ) \to H_1(F;\bZ)$ is the action of anticlockwise monodromy around $\partial \Delta$.
\end{enumerate}
\end{proposition}

\begin{remark} With the additional assumption that all of the singular fibres of $\pi_W$ are of Kodaira type $\rmI_1$, this result follows immediately from \cite[Proposition 4.4]{pdpslf} and \cite[Corollary 4.5]{pdpslf}.
\end{remark}

\begin{proof} The main proof proceeds in two steps, which we have separated out as lemmas. In Lemma \ref{lem:decomposition}, we show that $H_2(W,F;\bZ)$ can be decomposed into a set of simpler pieces, and that if (1), (3), and (4) hold for the pieces, then they also hold for $H_2(W,F;\bZ)$. Then, in Lemma \ref{lem:kodaira}, we prove that (1), (3), and (4) hold for the pieces of the decomposition. Combining these results proves statements (1), (3), and (4).

Finally, the proof of (2) follows by the same argument that was used to prove the corresponding statement in \cite[Proposition 4.4]{pdpslf}. Indeed, for any $u_1,u_2 \in H_2(W,F;\bZ)$, we have 
\[\langle u_2 , (\psi_{-2\pi})_* u_1\rangle_{\mathrm{Sft}} = -\langle u_2,(\psi_{-\pi})_*u_1 \rangle_{\mathrm{top}}= -\langle (\psi_{\pi})_*u_2,u_1 \rangle_{\mathrm{top}},\]
where the second equality holds as $\psi_{\pi}$ is a diffeomorphism. Finally, the topological intersection product is symmetric on $4$-manifolds, so  $\langle u_2 , (\psi_{-2\pi})_* u_1\rangle_{\mathrm{Sft}} = -\langle u_1, (\psi_{\pi})_*u_2 \rangle_{\mathrm{top}} = \langle u_1, u_2\rangle_{\mathrm{Sft}}$, as required.
\end{proof}

Before stating the first lemma, we set up some more notation. Label the points in $\Sigma$ as $p_1,\ldots,p_k$ as follows. Connect $p_1,\ldots,p_k$ to $p$ by straight lines in $\Delta$, choosing labels so that these lines emanate from $p$ in a clockwise order; if $m \geq 2$ points lie on the same line through $p$, label them $p_i,p_{i+1},\ldots,p_{i+m-1}$ in order of increasing distance from $p$. Label the fibre $\pi_W^{-1}(p_i)$ over $p_i$ by $F_i$.

\begin{lemma}\label{lem:decomposition}
There are elliptic fibrations over discs $W_i \to \Delta_i$, for $i \in \{1,\ldots,k\}$, each containing precisely one singular fibre of the same Kodaira type as $F_i$, so that
\[H_2(W,F;\bZ) \cong \bigoplus_{i=1}^k H_2(W_i,F;\bZ).\] 

Moreover, if all of the $H_2(W_i,F;\bZ)$ satisfy (1), (3), and (4) from Proposition \ref{prop:H2pseudolattice}, then so does $H_2(W,F;\bZ)$, and we have a decomposition of pseudolattices
\[H_2(W,F;\bZ) \cong \rmG_1 \oright_{\rmE} \rmG_2 \oright_{\rmE} \cdots \oright_{\rmE} \rmG_k,\]
where $\rmG_i := H_2(W_i,F;\bZ)$ and $\rmE := H_1(F;\bZ)$.
\end{lemma}

\begin{remark} When all of the singular fibres in $W$ have Kodaira type $\rmI_1$, this result is essentially \cite[Proposition 4.4(3)]{pdpslf}.
\end{remark}

\begin{proof} It suffices to show that the result holds for a decomposition into $W_1 \to \Delta_1$ and $W_2 \to \Delta_2$, where $W_1 \to \Delta_1$ contains precisely one singular fibre of the same Kodaira type as $F_1$ and $W_2 \to \Delta_2$ contains all other singular fibres; the general result then follows by induction.

Let $\gamma$ be a line through $p$ that passes between $p_1$ and $p_2$ and divides $\Delta$ into two parts; if $p$, $p_1$ and $p_2$ all lie on the same line, deform this line to pass slightly to the right of $p_1$ and slightly to the left of $p_2$, as viewed from $p$. By construction, $\gamma$ divides $\Delta$ into two parts, with $p_1$ on the left and $p_2,\ldots,p_k$ on the right, as viewed from $p$. Label the part containing $p_i$ by $\Delta_i^{\circ}$, for each $i \in \{1,2\}$.

Let $\gamma_{\varepsilon} \subset \Delta$ denote a closed $\varepsilon$-neighbourhood of $\gamma$, where $\varepsilon$ is chosen small enough so that $\gamma_{\varepsilon}$ does not contain any of the $p_i$. Finally, define $\Delta_i = \Delta_i^{\circ} \cup \gamma_{\varepsilon}$ and $W_i = \pi_W^{-1}(\Delta_i)$ for each $i \in \{1,2\}$. Note that, for each $i \in \{1,2\}$, $\Delta_i$ is homeomorphic to a closed disc and $p \in \partial \Delta_i$, so $W_1 \to \Delta_1$ (resp. $W_2 \to \Delta_2$) is an elliptic fibration over a disc containing $F_1$ (resp. $F_2,\ldots,F_k$) and $F \subset W_i$ for each $i \in \{1,2\}$.

With this setup, noting that $\gamma_{\varepsilon}$ is simply connected and the restriction of the fibration $\pi_W$ to $\gamma_{\varepsilon}$ is trivial, the Mayer-Vietoris sequence gives an isomorphism 
\begin{align*}
H_2(W_1,F;\bZ) \oplus H_2(W_2,F;\bZ) & \longrightarrow H_2(W,F;\bZ)\\
(\alpha,\beta) &\longmapsto (\iota_1)_*(\alpha) + (\iota_2)_*(\beta),
\end{align*}
where $\iota_j\colon W_j \to W$ denotes the inclusion.

Given this, the remainder of the proposition will follow from general properties of the pseudolattice sum (Definition \ref{def:gluing}) if we can show that $H_2(W,F;\bZ) \cong \rmG_1 \oright_{\rmE} \rmG_2$ where $\rmG_i := H_2(W_i,F;\bZ)$ and $\rmE := H_1(F;\bZ)$. But this is a straightforward consequence of the construction: the asymptotic charge $\phi\colon H_2(W,F;\bZ) \to H_1(F;\bZ)$ is the sum $\phi_1 + \phi_2$ of the asymptotic charges for $W_1$ and $W_2$, and if $(u_1,u_2), (v_1,v_2) \in H_2(W_1,F;\bZ) \oplus H_2(W_2,F;\bZ)$ then the Seifert pairing on $W$ is given by
\[\begin{pmatrix} \langle u_1,v_1 \rangle_{\mathrm{Sft}} & \langle \phi_1(u_1),\phi_2(v_2)\rangle_{\mathrm{E}} \\ 0 & \langle u_2,v_2 \rangle_{\mathrm{Sft}}\end{pmatrix},\]
which is precisely the definition of the bilinear form on $\rmG_1 \oright_{\rmE} \rmG_2$.
\end{proof}

This lemma reduces the proof of Proposition \ref{prop:H2pseudolattice} to the case of a fibration containing precisely one singular fibre. This setting is dealt with by our second lemma.

\begin{lemma}\label{lem:kodaira} Suppose that $\pi_W\colon W \to \Delta$ is an elliptic fibration over a disc containing precisely one singular fibre. Then $H_2(W,F;\bZ)$, equipped with the Seifert pairing, is a unimodular pseudolattice. Moreover, for some choice of basis $(a,b)$ for $H_1(F;\bZ) \cong \rmE$, we have an isomorphism between $H_2(W,F;\bZ) \to H_1(F;\bZ)$ and $\rmZ(v_1,\ldots,v_r) \to \rmE$ (see Example \ref{ex:multiZ}), for $v_i \in \rmE$ depending upon the Kodaira type of the singular fibre as given in Table \ref{tab:kodairafibres}. In particular, the asymptotic charge map $\phi\colon H_2(W,F;\bZ) \to H_1(F;\bZ)$ is a relative $(-1)^0$-Calabi-Yau spherical homomorphism. Finally, the twist map $T_{\phi}\colon H_1(F;\bZ) \to H_1(F;\bZ)$ is the action of anticlockwise monodromy around $\partial \Delta$.
\end{lemma}

\begin{table}
\begin{tabular}{c|c}
Kodaira type & $H_2(W,F;\bZ) \cong \rmZ(v_1,\ldots,v_r)$ \\
\hline
$\rmI_n$ & $\rmZ(a,a,\ldots,a)$ with $n$ copies of $a$\\
$\rmI_n^*$ & $\rmZ(a,a,\ldots,a,b-a,b+a)$ with $(n+4)$ copies of $a$\\
$\rmI\rmI$ & $\rmZ(a,b+a)$\\
$\rmI\rmI\rmI$ & $\rmZ(a,a,b+a)$\\
$\rmI\rmV$ & $\rmZ(a,a,a,b+a)$\\
$\rmI\rmV^*$ & $\rmZ(a,a,a,a,a,b-a,b+a,b+a)$\\
$\rmI\rmI\rmI^*$ & $\rmZ(a,a,a,a,a,a,b-a,b+a,b+a)$\\
$\rmI\rmI^*$ & $\rmZ(a,a,a,a,a,a,a,b-a,b+a,b+a)$
\end{tabular}
\caption{Pseudolattices $H_2(W,F;\bZ)$ for $W$ an elliptic fibration over a disc containing one singular fibre.}
\label{tab:kodairafibres}
\end{table}

\begin{proof}
This proof is a lengthy explicit case analysis. Essentially, for each type of singular fibre our approach is to use geometric considerations to construct a basis for $H_2(W,F;\bZ)$ and to compute the Seifert pairing for it, then to use a computer algebra system to solve the resulting linear algebra problem and find the required isomorphism. We present a representative sample of cases here to illustrate the method; the other cases are analogous.

We begin with some general setup. Let $p \in \partial \Delta$ denote a point on the boundary of $\Delta$, let $F = \pi_W^{-1}(p)$, and let $q \in \Delta$ be the point over which the singular fibre occurs. As $\pi_W$ has only one singular fibre, which is not multiple, we may assume that $\pi_W$ admits a section $\sigma$ over $\Delta$. Let $\overline{\pi}_W \colon \overline{W} \to \Delta$ denote the quotient of $W$ by the fibrewise elliptic involution defined by this section; note that the double cover $W \to \overline{W}$ is ramified along $\sigma$ and a trisection of $\pi_W$.

The simplest case is $\rmI_n$. Choose a path $\alpha$ from $p$ to $q$. The ramification locus of  $W \to \overline{W}$ contains two disjoint sections, $\sigma_1$ and $\sigma_2$, which meet the same component of the $\rmI_n$ fibre. $\sigma_1$ and $\sigma_2$ intersect $\overline{\pi}_W^{-1}(\alpha)$ in a pair of curves; take a strip in $\overline{\pi}_W^{-1}(\alpha)$ joining them (see Figure \ref{fig:In}). The double cover of this strip under $W \to \overline{W}$ is a thimble $t_{\alpha}$ in $H_2(W,F;\bZ)$; let $a$ denote the corresponding vanishing cycle. Consider the set of classes in $H_2(W,F;\bZ)$ given by $t_{\alpha}$ along with the components $e_1,\ldots,e_{n-1}$ of the $\rmI_n$ fibre which do not meet $\sigma$. One may check that these classes generate $H_2(W,F;\bZ)$ and the Seifert pairing $\chi$ and asymptotic charge map $\phi$ are given by
\[\chi = \kbordermatrix{ & e_{n-1} & e_{n-2} & e_{n-3} & \cdots & e_1 & t_{\alpha} \\ 
e_{n-1} & 2 & -1 & 0 &\cdots & 0 & 0 \\
e_{n-2} & -1 & 2 & -1 & \cdots & 0 & 0 \\
e_{n-3} & 0 & -1 & 2 & \cdots & 0 & 0 \\
\vdots & \vdots & \vdots & \vdots & \ddots & \vdots & \vdots \\
e_{1} & 0 & 0 & 0 &  \cdots & 2 & -1 \\
t_{\alpha} & 0 & 0 & 0 & \cdots & -1 & 1}
\quad \phi = \kbordermatrix{ & e_{n-1} & \cdots & e_1 & t_{\alpha} \\ 
a & 0  & \cdots & 0 & 1 \\
b & 0  & \cdots & 0 & 0 }.\]
As $\det(\chi) = 1$, it follows that the Seifert pairing is unimodular and these classes form a basis for $H_2(W,F;\bZ)$. Using linear algebra, one may then verify that there is a change of basis in $H_2(W,F;\bZ)$ which takes $H_2(W,F;\bZ) \to H_1(F;\bZ)$ to $\rmZ(a,\ldots,a) \to \rmE$ and that the twist map is the usual action $\begin{psmallmatrix} 1 & n \\ 0 & 1
\end{psmallmatrix}$ of anticlockwise monodromy around $\partial \Delta$.

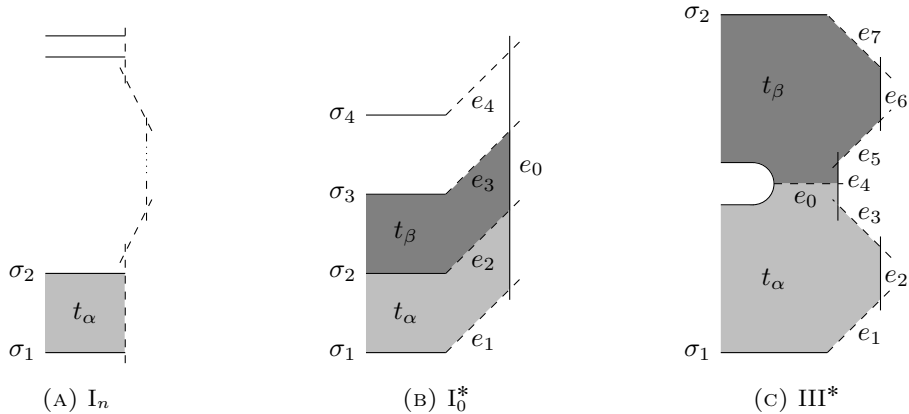
\begin{figure}
  \begin{center}
  \begin{subfigure}{0.3\textwidth}
  \begin{center}
\begin{tikzpicture}[scale=0.7]
\filldraw[lightgray] (0.5,0)--(2,0)--(2,1.5)--(0.5,1.5);
\filldraw[white] (2,0)--(2,1.5)--(2.2,1.5)--(2.2,0);

\draw (0.5,0)--(2,0);
\draw (0.5,1.5)--(2,1.5);
\draw[dashed] (2,-0.2)--(2,2.1);
\draw[dashed] (1.9,1.7)--(2.5,2.9);
\draw[dashed] (2.4,2.5)--(2.4,3.4);
\draw[dotted] (2.4,3.4)--(2.4,3.9);
\draw[dashed] (2.4,4)--(2.4,4.6);
\draw[dashed] (2.5,4.2)--(1.9,5.4);
\draw[dashed] (2,5.1)--(2,6.2);
\draw (0.5,5.6)--(2,5.6);
\draw (0.5,6)--(2,6);

\node[left] at (0.5,0) {$\sigma_1$};
\node[left] at (0.5,1.5) {$\sigma_2$};
\node at (1.25,0.75) {$t_{\alpha}$};

\end{tikzpicture}
\caption{$\rmI_n$}
\label{fig:In}
\end{center}
\end{subfigure}
\begin{subfigure}{0.3\textwidth}
  \begin{center}
\begin{tikzpicture}[scale=0.7]
\filldraw[gray] (0.5,1.5)--(2,1.5)--(3.2,2.7)--(3.2,4.2)--(2,3)--(0.5,3);
\filldraw[lightgray] (0.5,0)--(2,0)--(3.2,1.2)--(3.2,2.7)--(2,1.5)--(0.5,1.5);
\filldraw[white] (0.5,3)--(2,3)--(3.2,4.2)--(3.2,5.7)--(2,4.5)--(0.5,4.5);
\filldraw[white] (2,0)--(3.2,1.2)--(3.2,0);

\draw (0.5,0)--(2,0);
\draw (0.5,1.5)--(2,1.5);
\draw (0.5,3)--(2,3);
\draw (0.5,4.5)--(2,4.5);
\draw[dashed] (2,0)--(3.4,1.4);
\draw[dashed] (2,1.5)--(3.4,2.9);
\draw[dashed] (2,3)--(3.4,4.4);
\draw[dashed] (2,4.5)--(3.4,5.9);
\draw (3.2,1)--(3.2,6);

\node[left] at (0.5,0) {$\sigma_1$};
\node[left] at (0.5,1.5) {$\sigma_2$};
\node[left] at (0.5,3) {$\sigma_3$};
\node[left] at (0.5,4.5) {$\sigma_4$};
\node at (1.25,0.75) {$t_{\alpha}$};
\node at (1.25,2.25) {$t_{\beta}$};
\node[right] at (2.3,0.2) {$e_1$};
\node[right] at (2.3,1.7) {$e_2$};
\node[right] at (2.3,3.2) {$e_3$};
\node[right] at (2.3,4.7) {$e_4$};
\node[right] at (3.2,3.5) {$e_0$};

\end{tikzpicture}
\caption{$\rmI_0^*$}
\label{fig:I0*}
\end{center}
\end{subfigure}
\begin{subfigure}{0.3\textwidth}
  \begin{center}
\begin{tikzpicture}[scale=0.7]
\filldraw[gray] (0,3.2)--(2.2,3.2)--(2.2,3.6)--(3,4.4)--(3,5.4)--(2,6.4)--(0,6.4);
\filldraw[lightgray] (0,0)--(2,0)--(3,1)--(3,2)--(2.2,2.8)--(2.2,3.2)--(0,3.2);
\filldraw[white] (0.6,3.2) circle (0.4);

\draw (0.6,3.2) circle (0.4);
\filldraw[white] (-0.1,3.6)--(0.6,3.6)--(0.6,2.8)--(-0.1,2.8);
\filldraw[white] (3,5.4)--(2,6.4)--(3,6.4);
\filldraw[white] (2.2,3.6)--(3,4.4)--(3,3.6);
\filldraw[white] (3,2)--(2.2,2.8)--(3,2.8);
\filldraw[white] (2,0)--(3,1)--(3,0);

\draw (0,0)--(2,0);
\draw[dashed] (2,0)--(3.2,1.2);
\draw (3,0.8)--(3,2.2);
\draw[dashed] (3.2,1.8)--(2.1,2.9);
\draw (2.2,2.5)--(2.2,3.8);
\draw[dashed] (1,3.2)--(2.3,3.2);
\draw[dashed] (2.1,3.5)--(3.2,4.6);
\draw (3,4.2)--(3,5.6);
\draw[dashed] (3.2,5.2)--(2,6.4);
\draw (0,6.4)--(2,6.4);

\draw (0,3.6)--(0.62,3.6);
\draw (0,2.8)--(0.62,2.8);

\node[left] at (0,0) {$\sigma_1$};
\node[left] at (0,6.4) {$\sigma_2$};
\node at (1,1.4) {$t_{\alpha}$};
\node at (1,5) {$t_{\beta}$};
\node[right] at (2.4,0.3) {$e_1$};
\node[right] at (2.9,1.4) {$e_2$};
\node[right] at (2.4,2.6) {$e_3$};
\node[right] at (2.2,3.2) {$e_4$};
\node at (1.6,2.9) {$e_0$};
\node[right] at (2.4,3.7) {$e_5$};
\node[right] at (2.9,4.8) {$e_6$};
\node[right] at (2.4,6) {$e_7$};

\end{tikzpicture}
\caption{$\mathrm{III}^*$}
\label{fig:III*}
\end{center}
\end{subfigure}
\caption{Construction of strips in $\overline{W}$. The singular fibre is on the right of each diagram and the horizontal lines are sections/multisections. Solid lines denote branch curves of the cover $W \to \overline{W}$ and dashed lines denote components of the singular fibre which are not part of the branch locus. The strips giving rise to $t_{\alpha}$ and $t_{\beta}$ are shaded and labelled.}
\label{fig:mn}
\end{center}
\end{figure}

A case which is more indicative of the general approach is $\rmI_0^*$. In this case we choose two paths $\alpha$ and $\beta$ from $p$ to $q$, intersecting only at $p$ and $q$, with $\alpha$ passing to the left of $\beta$ as viewed from $p$. The ramification locus of  $W \to \overline{W}$ contains four disjoint sections, $\sigma_1,\ldots,\sigma_4$. These sections intersect $\overline{\pi}^{-1}_W(\alpha)$ and $\overline{\pi}^{-1}_W(\beta)$ in four curves each; take a strip in $\overline{\pi}^{-1}_W(\alpha)$ joining $\sigma_1$ to $\sigma_2$ and a strip in $\overline{\pi}^{-1}_W(\beta)$ joining $\sigma_2$ to $\sigma_3$ (see Figure \ref{fig:I0*}). The double cover of these strips under $W \to \overline{W}$ are thimbles $t_{\alpha}$ and $t_{\beta}$ in $H_2(W,F;\bZ)$; let $a$ and $b$ denote the corresponding vanishing cycles. Label the components of the $\rmI_0^*$ fibre such that $e_i$ meets $\sigma_i$ and $e_0$ is the central component. Consider the set of classes in $H_2(W,F;\bZ)$ given by $t_{\alpha}$ and $t_{\beta}$ along with the components $e_0,e_1,e_2,e_4$ of the $\rmI_0^*$ fibre which do not meet $\sigma_3$. One may check that these classes generate $H_2(W,F;\bZ)$ and the Seifert pairing $\chi$ and asymptotic charge map $\phi$ are given by
\[\chi = \kbordermatrix{ & e_0 & e_1 & e_2 & e_4 & t_{\alpha} & t_{\beta} \\ 
e_{0} & 2 & -1 & -1 & -1 & -1 & -1 \\
e_{1} & -1 & 2 & 0 & 0 & 1 & 0 \\
e_{2} & -1 & 0 & 2 & 0 & 1 & 1 \\
e_4 & -1 & 0 & 0 & 2 & 0 & 0 \\
t_{\alpha} & -1 & 1 & 1 & 0 & 1 & 0 \\
t_{\beta} & -1 & 0 & 1 & 0 & 1 & 1}\quad \phi = \kbordermatrix{ & e_{0} & e_1 & e_2 & e_4 & t_{\alpha} & t_{\beta} \\ 
a & 0  & 0 & 0 & 0 & 1 & 0\\
b & 0  & 0 & 0 & 0 & 0 & 1}.\]
Linear algebra can be used to verify that there is a change of basis in $H_2(W,F;\bZ)$ which takes $H_2(W,F;\bZ) \to H_1(F;\bZ)$ to $\rmZ(a,a,a,a,b-a,b+a) \to \rmE$ and that the twist map is the usual action $\begin{psmallmatrix} -1 & 0 \\ 0 & -1
\end{psmallmatrix}$ of anticlockwise monodromy around $\partial \Delta$.

The most complicated case that we will cover here is $\mathrm{III}^*$; the other cases are no more difficult. In this case we again choose two paths $\alpha$ and $\beta$ from $p$ to $q$, intersecting only at $p$ and $q$, with $\alpha$ passing to the left of $\beta$ as viewed from $p$. The ramification locus of  $W \to \overline{W}$ contains two disjoint sections, $\sigma_1$ and $\sigma_2$, along with a bisection. These sections intersect $\overline{\pi}^{-1}_W(\alpha)$ and $\overline{\pi}^{-1}_W(\beta)$ in two disjoint sections and a bisection ramified over $q$. Take a strip in $\overline{\pi}^{-1}_W(\alpha)$ joining $\sigma_1$ to one branch of the bisection and a strip in $\overline{\pi}^{-1}_W(\beta)$ joining $\sigma_2$ to the other branch (see Figure \ref{fig:III*}). The double cover of these strips under $W \to \overline{W}$ are thimbles $t_{\alpha}$ and $t_{\beta}$ in $H_2(W,F;\bZ)$; let $a$ and $b$ denote the corresponding vanishing cycles. Label the components of the $\rmI\rmI\rmI^*$ fibre such that $e_1,\ldots,e_7$ form the chain joining $\sigma_1$ to $\sigma_2$ and $e_0$ is the component meeting the bisection. Consider the set of classes in $H_2(W,F;\bZ)$ given by $t_{\alpha}$ and $t_{\beta}$ along with the components $e_0,\ldots,e_6$ of the $\rmI\rmI\rmI^*$ fibre which do not meet $\sigma_2$. One may check that these classes generate $H_2(W,F;\bZ)$ and the Seifert pairing $\chi$ and asymptotic charge map $\phi$ are given by
\[\chi = \kbordermatrix{ & e_0 & e_1 & e_2 & e_3 &  e_4 & e_5 & e_6 & t_{\alpha} & t_{\beta} \\ 
e_{0} & 2 & 0 & 0 & 0 & -1 & 0 & 0 & 1 & 1 \\
e_{1} & 0 & 2 & -1 & 0 & 0 & 0 & 0 & 1 & 1 \\
e_{2} & 0 & -1 & 2 & -1 & 0 & 0 & 0 & 1 & 1 \\
e_{3} & 0 & 0 & -1 & 2 & -1 & 0 & 0 & 1 & 1 \\
e_{4} & -1 & 0 & 0 & -1 & 2 & -1 & 0 & 1 & 1 \\
e_{5} & 0 & 0 & 0 & 0 & -1 & 2 & -1 & 1 & 1 \\
e_{6} & 0 & 0 & 0 & 0 & 0 & -1 & 2 & 1 & 1 \\
t_{\alpha} & 1 & 1 & -1 & 1 & -1 & 0 & 0 & 1 & 0 \\
t_{\beta} & 1 & 0 & 0 & 0 & -1 & 1 & -1 & 1 & 1}\]
\[\phi = \kbordermatrix{& e_0 & e_1 & e_2 & e_3 & e_4 & e_5 & e_6 & t_{\alpha} & t_{\beta} \\ 
a & 0  & 0 & 0 & 0 & 0 & 0 & 0 & 1 & 0\\
b & 0  & 0 & 0 & 0 & 0 & 0 & 0 & 0 & 1}.\]
Linear algebra can then be used to verify that the twist map is the usual action $\begin{psmallmatrix} 0 & -1 \\ 1 & 0 \end{psmallmatrix}$ of anticlockwise monodromy around $\partial \Delta$ and that there are changes of bases\footnote{It is important to note in the $\mathrm{III}^*$ case that the basis $(a,b) \in \rmE$ defining the pseudolattice $\rmZ(a,a,a,a,a,a,b-a,b+a,b+a)$ is \emph{not} the same as the basis $(a,b)$ of $H_1(F;\bZ)$ constructed here (the twist maps do not match), but the two are related by a change of basis in $H_1(F;\bZ) \cong \rmE$.} in $H_2(W,F;\bZ)$ and $H_1(F;\bZ)$ which take $H_2(W,F;\bZ) \to H_1(F;\bZ)$ to $\rmZ(a,a,a,a,a,a,b-a,b+a,b+a) \to \rmE$.
\end{proof}

The following result is a consequence of Lemma \ref{lem:kodaira}; whilst we will not need it in the sequel, we present it here out of interest.

\begin{proposition} Suppose that $\pi_W\colon W \to \Delta$ is an elliptic fibration over a disc containing precisely one singular fibre. Then there exists a  deformation of $\pi_W$ to a genus $1$ Lefschetz fibration $\tilde{\pi}_W\colon \tilde{W} \to \Delta$ (i.e. an elliptic fibration over a disc in which all singular fibres are of Kodaira type $\rmI_1$), such that there are isomorphisms of pseudolattices $H_2(W,F;\bZ) \cong H_2(\tilde{W},\tilde{F};\bZ)$ and $H_1(F;\bZ) \cong H_1(\tilde{F};\bZ)$ making the following diagram commute
\[
\begin{tikzcd}
H_2(W,F;\bZ) \ar[d] \ar[r,"\sim"]& H_2(\tilde{W},\tilde{F};\bZ)\ar[d] \\
H_1(F;\bZ) \ar[r,"\sim"] & H_1(\tilde{F};\bZ),
\end{tikzcd}\]
where the vertical arrows are the asymptotic charge maps.
\end{proposition}

\begin{proof} For each Kodaira fibre type, the existence of an appropriate deformation to a genus $1$ Lefschetz fibration $\tilde{\pi}_W\colon \tilde{W} \to \Delta$ is given by \cite[Theorem 3.1]{gtsj} and, for such deformations, explicit bases of thimbles for $H_2(\tilde{W},\tilde{F};\bZ)$ have been computed in the string junction literature (see, for example, \cite[Section 3.1]{mwlsj}). From this, the pseudolattice structure on $H_2(\tilde{W},\tilde{F};\bZ)$ and the asymptotic charge map to $H_1(\tilde{F};\bZ)$ can be immediately deduced from \cite[Proposition 4.4]{pdpslf}. The pseudolattices thus obtained are precisely those given in Table \ref{tab:kodairafibres}\footnote{Caution: the reader attempting to replicate the computation in this proof should take careful note of conventions for signs, orientations, and compositions, which are not consistent between the three papers \cite{mwlsj,gtsj,pdpslf}!}.
\end{proof}

To conclude this section, we discuss when $\phi\colon H_2(W,F;\bZ) \to H_1(F;\bZ)$ is a quasi del Pezzo homomorphism. The following assumption will be critical. 

\begin{assumption}\label{ass:quasiLG} There exists a symplectic basis $(a,b)$ for $H_1(F;\bZ)$ such that:
\begin{enumerate}
\item the action of anticlockwise monodromy around $\partial \Delta$ is given by 
\[\begin{pmatrix}1 & e(W)-12 \\ 0 & 1 \end{pmatrix},\]
 where $e(W)$ denotes the topological Euler number of $W$; and
\item $r(a)$ is primitive in $H_2(W,F;\bZ)$, where $r$ is the right adjoint to the asymptotic charge map $\phi\colon H_2(W,F;\bZ) \to H_1(F;\bZ)$.
\end{enumerate}
\end{assumption}

\begin{theorem}\label{thm:qdpfibration} Let $\pi_W\colon W \to \Delta$ be an elliptic fibration over a disc and assume that there exists a symplectic basis $(a,b)$ for $H_1(F;\bZ)$ so that Assumption \ref{ass:quasiLG} holds. Then $\phi\colon H_2(W,F;\bZ) \to H_1(F;\bZ)$ is a quasi del Pezzo homomorphism.
\end{theorem}

\begin{proof} We check the conditions from Definition \ref{def:qdp}. Condition (1) follows immediately from Proposition \ref{prop:sphericalsurfacelike} and condition (2) from Proposition \ref{prop:H2pseudolattice}. Condition (3) follows from Lemmas \ref{lem:decomposition} and \ref{lem:kodaira}, which together give an isomorphism between $H_2(W,F;\bZ)$ and $\rmZ(v_1,\ldots,v_r)$, for some $v_1,\ldots,v_r \in H_1(F;\bZ)$.

Condition (4) is the most difficult to check. The proof of \cite[Theorem 4.14]{pdpslf}, which occupies most of \cite[Section 4.3]{pdpslf}, applies mostly verbatim in our setting. We give a sketch of the argument here, highlighting the main differences.

If $L$ is a lattice of signature $(n_+, n_-)$, define its \emph{$\sigma$-invariant} to be $\sigma(L) := n_+ - n_-$. Let $\sigma(W)$ denote the $\sigma$-invariant of $H_2(W;\bZ)$, equipped with its usual topological intersection pairing. Then the steps in the argument from \cite[Section 4.3]{pdpslf} adapt to this setting as follows.
\begin{itemize}
\item The analogue of \cite[Lemma 4.8]{pdpslf} states that, equipping $\ker(\phi) \subset H_2(W,F;\bZ)$ with the pairing induced by the Seifert pairing, we have $\sigma(W) = -\sigma(\ker(\phi))$. The proof, which is originally due to Auroux \cite{sclf}, still holds in our more general setting. Indeed, if $u_1$ and $u_2$ are classes in $\ker(\phi)$ then, by the long exact sequence of a pair, $u_1$ and $u_2$ are classes from $H_2(W;\bZ)$ with no boundary in $H_1(F;\bZ)$, so we have $\langle u_1,u_2\rangle_{\mathrm{Sft}} = -\langle  u_1,u_2 \rangle_{\mathrm{top}}$.
\item The analogue of \cite[Lemma 4.10]{pdpslf} states that, letting $K$ denote the canonical class in $\NS(H_2(W,F;\bZ))$ (Definition \ref{def:canonical}), we have $\sigma(W) = \sigma(K^{\perp})$. The proof from \cite{pdpslf} only uses the basic properties established in Propositions \ref{prop:sphericalsurfacelike} and \ref{prop:H2pseudolattice}, along with the equality $\sigma(W) = -\sigma(\ker(\phi))$ established above, so applies verbatim in our setting.
\item \cite[Lemma 4.11]{pdpslf} is a standard result in lattice theory which states that
\[\sigma(K^{\perp}) = \begin{cases} \sigma(\NS(H_2(W,F;\bZ)) - 1 & \text{if } q(K,K) > 0,\\
 \sigma(\NS(H_2(W,F;\bZ)) & \text{if } q(K,K) = 0,\\
  \sigma(\NS(H_2(W,F;\bZ)) + 1 & \text{if } q(K,K) < 0.\end{cases}\]
\end{itemize}

Putting these things together it follows that in order to compute $\sigma(\NS(H_2(W,F;\bZ))$, it suffices to compute $\sigma(W)$. This is a direct computation, following the method of \cite[Proposition 4.12]{pdpslf}: Wall's nonadditivity theorem \cite{nas} shows that $\sigma(W)$ may be decomposed as a sum of contributions from each of the singular fibres plus a contribution from the smooth locus. The contribution from the smooth locus can be computed by a result of Mayer \cite[Satz 5]{dsvf} and the contributions from the singular fibres by a result of Matsumoto \cite[Corollary 8.1]{4mft2}; critically, Matsumoto's result holds for fibres of any Kodaira type, not just $\rmI_1$. From this computation one obtains\footnote{Note that the cases $d<0$ and $d > 0$ in the formula from \cite[p. 261]{dsvf} are mistakenly transposed in the proof of \cite[Proposition 4.12]{pdpslf}, which has the effect of swapping the cases $n < 12$ and $n > 12$ in the statement of that result. This error is compensated for in the proof of \cite[Proposition 4.13]{pdpslf}, so the main result \cite[Theorem 4.14]{pdpslf} is unaffected.}
\[\sigma(W) = \begin{cases} 3 - e(W) & \text{if } e(W) < 12,\\
 -8 & \text{if } e(W) = 12,\\
 5 - e(W) & \text{if } e(W) > 12.\end{cases}\]
 
 Finally, from Proposition \ref{prop:sphericalsurfacelike} we obtain $q(K,K) = 12 - e(W)$. Combining this with the results above gives $\sigma(\NS(H_2(W,F;\bZ)) = 4 - e(W)$. Since $\NS(H_2(W,F;\bZ))$ is a nondegenerate lattice of rank $e(W) - 2$, it follows that $\NS(H_2(W,F;\bZ))$ has signature $(1, e(W) - 3)$. This completes the proof of Theorem \ref{thm:qdpfibration}. \end{proof}

\subsection{Elliptically fibred K3 surfaces} \label{sec:ellipticK3}

Now we shift our attention to K3 surfaces. Let $Y$ be a K3 surface which admits an elliptic fibration $\pi\colon Y \to \bP^1$. We do not assume that $\pi$ has a section but we note that, by \cite[Proposition 11.1.6]{lok3s}, $\pi$ cannot have multiple fibres. Let $\Sigma \subset \bP^1$ denote the finite set of points over which the fibres of $\pi$ are singular.

Let $\gamma \subset \bP^1\setminus \Sigma$ be a simple loop and let $p,q \in \gamma$ be two distinct points. Let $F_p = \pi^{-1}(p)$ and $F_q = \pi^{-1}(q)$ denote the fibres over $p$ and $q$. Finally, let $U$ denote the open set $U := Y \setminus F_q$.

 $\gamma$ divides $\bP^1$ into two pieces which are homeomorphic to discs $\Delta_1$ and $\Delta_2$. Let $\pi_1\colon Y_1 \to \Delta_1$ and $\pi_2\colon Y_2 \to \Delta_2$ denote the induced elliptic fibrations over discs. Then Proposition \ref{prop:H2pseudolattice} holds for the asymptotic charge maps $\phi_i\colon H_2(Y_i,F_p;\bZ) \to H_1(F_p;\bZ)$, for $i \in \{1,2\}$.

\begin{definition}\label{def:allowable} A simple loop $\gamma \subset \bP^1 \setminus \Sigma$ is called \emph{allowable} if Assumption \ref{ass:quasiLG} holds for both $\phi_1\colon H_2(Y_1,F_p;\bZ) \to H_1(F_p;\bZ)$ and $\phi_2\colon H_2(Y_2,F_p;\bZ) \to H_1(F_p;\bZ)$.
 \end{definition}

\begin{remark} We note that if Assumption \ref{ass:quasiLG}(1) holds for one of $\pi_1\colon Y_1 \to \Delta_1$ or $\pi_2\colon Y_2 \to \Delta_2$, then it automatically holds for both, as the anticlockwise monodromy around $\partial \Delta_1$ is the inverse of the anticlockwise monodromy around $\partial \Delta_2$ and $e(Y_1) + e(Y_2) = e(Y) = 24$.
\end{remark}

Assuming that $\gamma$ is allowable, Theorem \ref{thm:qdpfibration} implies that the maps $\phi_i\colon H_2(Y_i,F_p;\bZ) \to H_1(F_p;\bZ)$, for $i \in \{1,2\}$, are quasi del Pezzo homomorphisms of pseudolattices. Moreover, by Proposition \ref{prop:sphericalsurfacelike}, if we denote the canonical class for each $i \in \{1,2\}$ by $K_i \in \NS(H_2(Y_i,F_p;\bZ))$, then we have 
\[q(K_1,K_1) = 12-e(Y_1) = e(Y_2) - 12 = -q(K_2,K_2),\]
since $e(Y_1) + e(Y_2) = 24$. We are thus in the setting of Section \ref{sec:pseudolattice}; as in that setting, we will assume labels have been chosen so that $q(K_1,K_1) \geq 0$ and consequently $e(Y_1) \leq 12$. 

Let $f\colon \rmG \to \rmE$ denote the spherical homomorphism obtained by gluing the pseudolattices $\rmG_1 = H_2(Y_1,F_p;\bZ)$ and $\rmG_2 =H_2(Y_2,F_p;\bZ)$ along $\rmE \cong H_1(F_p;\bZ)$ via the map $\phi_1 \oright (-\phi_2)$, as in Section \ref{sec:pseudolattice}, and let $r$ be its right adjoint. Note that $\deg(\rmG) = 12-e(Y_1)$. The results of Section \ref{sec:pseudolattice} give rise to lattices $\overline{\mathrm{E}}$, $\mathrm{K}$, and $\rmM$, related by the short exact sequence \eqref{eq:ses}. 

\begin{lemma}\label{lem:GKforfibrations} $\rmG$ is naturally isomorphic to $H_2(U,F_p;\bZ)$ (as $\bZ$-modules) and the spherical homomorphism $f\colon \rmG \to \rmE$ coincides with the boundary $H_2(U,F_p;\bZ) \to H_1(F_p;\bZ)$ from the long exact sequence of a pair.

Under Poincar\'{e} duality, the lattice $\rmK$ is isomorphic to $H^2_c(U;\bZ)/\bZ F$ (as $\bZ$-modules), where $F$ denotes the class of a fibre in $H^2_c(U;\bZ)$.
\end{lemma}

\begin{remark}\label{rem:perverseleray} From the viewpoint of the perverse Leray filtration $P_{\bullet}$ on $H^2_c(U;\bQ)$, as described in \cite[Sections 2 and 5.1.1]{mcss}, the class $F$ generates $P_1 H^2_c(U;\bQ)$, and the quotient $H^2_c(U;\bQ)/\bQ F$ may be identified with the graded piece $\Gr^P_1 H^2_c(U;\bQ)$.
\end{remark}

\begin{proof} The proof follows the same idea as the proof of Lemma \ref{lem:decomposition}, with a crucial difference (see Remark \ref{rem:orientations}).

Begin by choosing $\varepsilon > 0$ so that the open ball $B_{\varepsilon}(q) \subset \bP^1$ does not contain $p$ or any point of $\Sigma$. Then define $\overline{U} := Y \setminus \pi^{-1}B_{\varepsilon}(q)$. Note that $\overline{U}$ is a closed manifold with boundary and is a deformation retract of $U$, so $H_2(\overline{U},F_p;\bZ) \cong H_2(U,F_p;\bZ)$ and $H_2(\overline{U};\bZ) \cong H_2(U;\bZ) \cong H^2_c(U;\bZ)$, where the last isomorphism is Poincar\'{e} duality.

Now let $\gamma_{\delta}$ denote a closed $\delta$-neighbourhood of $\gamma$ inside $\bP^1\setminus B_{\varepsilon}(q)$, where $0 < \delta < \varepsilon$ is chosen small enough that $\gamma_{\delta}$ does not contain any point of $\Sigma$. Then define $\overline{\Delta_i} := (\Delta_i \cup \gamma_{\delta}) \setminus B_{\varepsilon}(q)$ and $\overline{Y}_i := \pi^{-1}(\overline{\Delta}_i)$. Note that, for each $i \in \{1,2\}$, $\overline{\Delta}_i$ is homeomorphic to a closed disc and $\overline{Y}_i$ is homeomorphic to $Y_i$. Note further that $\overline{Y}_1$ and $\overline{Y}_2$ cover $\overline{U}$, with $\overline{Y}_1 \cap \overline{Y}_2 = \pi^{-1}(\gamma_{\delta})$

With this setup, noting that $\gamma_{\delta}$ is simply connected and the restriction of the fibration $\pi$ to $\gamma_{\delta}$ is trivial, the Mayer-Vietoris sequence gives an isomorphism 
\begin{align*}
\rmG \cong H_2(\overline{Y}_1,F_p;\bZ) \oplus H_2(\overline{Y}_2,F_p;\bZ) & \longrightarrow H_2(\overline{U},F_p;\bZ) \cong H_2(U,F_p;\bZ)\\
(\alpha,\beta) &\longmapsto (\iota_1)_*(\alpha) - (\iota_2)_*(\beta),
\end{align*}
where $\iota_j\colon \overline{Y}_j \to \overline{U}$ denotes the inclusion. Moreover, by construction the boundary map is given by $\phi_1 - \phi_2$, where $\phi_i\colon H_2(Y_1,F_p;\bZ) \to H_1(F_p;\bZ)$ are the asymptotic charge maps. This proves the first part of the statement.

For the second part, note that the long exact sequence of a pair gives
\[0 \longrightarrow H_2(F_p;\bZ) \longrightarrow H_2(U;\bZ) {\longrightarrow} H_2(U,F_p;\bZ) \stackrel{\phi}{\longrightarrow} H_1(F_p;\bZ) \longrightarrow 0.\]
By definition, $\rmK$ is the kernel of the boundary map $\phi$. By exactness, this is isomorphic to the quotient of $H_2(U;\bZ)$ by the image of $H_2(F_p;\bZ)$, which is just the class of the fibre $F_p$. An application of Poincar\'{e} duality gives the result.
\end{proof}

\begin{remark} \label{rem:orientations} We note that this proof differs from the proof of Lemma \ref{lem:decomposition} in the choice of sign in the isomorphism coming from the Mayer-Vietoris sequence. This corresponds to a choice of orientation on the boundaries of the two pieces, which affects how they glue: in Lemma \ref{lem:decomposition}, we glue with the same orientation, whereas in the lemma above we glue with opposite orientations.
\end{remark}

\begin{definition}\label{def:bilinearonU} We may place a bilinear form on $H^2_c(U;\bZ)$ as follows. As in the proof of Lemma \ref{lem:GKforfibrations}, choose $\varepsilon > 0$ so that the open ball $B_{\varepsilon}(q) \subset \bP^1$ does not contain any point of $\Sigma$. Then define $\overline{U} := Y \setminus \pi^{-1}B_{\varepsilon}(q)$.

$\overline{U}$ is a closed manifold with boundary and is a deformation retract of $U$. We therefore have isomorphisms $H^2_c(U;\bZ) \cong H_2(U;\bZ) \cong H_2(\overline{U};\bZ)$, where the first isomorphism is Poincar\'{e} duality. As a manifold with boundary, $H_2(\overline{U};\bZ)$ has a topological intersection product, which defines a symmetric bilinear form on $H^2_c(U;\bZ)$ by the isomorphism above. \end{definition}

The class $F$ is totally degenerate with respect to this bilinear form, so this form descends to the quotient $H^2_c(U;\bZ)/\bZ F$.

\begin{lemma} \label{lem:GKforms} Under the isomorphism of Lemma \ref{lem:GKforfibrations}, the bilinear form induced on $\rmK$ from $\rmG$ coincides with the negative of the bilinear form on $H^2_c(U;\bZ)/\bZ F$ induced by Definition \ref{def:bilinearonU}.
\end{lemma}

\begin{proof} Let $\alpha,\beta \in \rmK$ be any two elements. Expressing $\rmK$ as a sublattice of $\rmG$, by the proof of Lemma \ref{lem:GKforfibrations} we may decompose $\alpha = (\iota_1)_*(\alpha_1) - (\iota_2)_*(\alpha_2)$ and $\beta = (\iota_1)_*(\beta_1) - (\iota_2)_*(\beta_2)$ for $\alpha_j,\beta_j \in H_2(Y_j,F_p;\bZ)$. Note that, as $\alpha,\beta \in \rmK = \ker(\phi_1 \oright (-\phi_2))$, we have $\phi_1(\alpha_1) = \phi_2(\alpha_2)$ and $\phi_1(\beta_1) = \phi_2(\beta_2)$. Now we compute
\begin{align*}
\langle \alpha,\beta \rangle_{\rmG} &= \langle \alpha_1,\beta_1 \rangle_{\mathrm{Sft}} - \langle \phi_1(\alpha_1),\phi_2(\beta_2)\rangle_{\rmE} + \langle \alpha_2,\beta_2 \rangle_{\mathrm{Sft}}\\
&= \langle \alpha_1,\beta_1 \rangle_{\mathrm{Sft}} - \langle \phi_1(\alpha_1),\phi_1(\beta_1)\rangle_{\rmE} + \langle \alpha_2,\beta_2 \rangle_{\mathrm{Sft}}\\
&= \langle \beta_1,\alpha_1 \rangle_{\mathrm{Sft}} + \langle \alpha_2,\beta_2 \rangle_{\mathrm{Sft}},
\end{align*}
where the final equality comes from \cite[Lemma 2.21]{pdpslf}. Now, using the fact that the Serre operator on $H_2(Y_1,F_p;\bZ)$ is given by $(\psi_{-2\pi})_*$ (Proposition \ref{prop:H2pseudolattice}), we obtain
\begin{align*}
\langle \alpha,\beta \rangle_{\rmG} &= \langle \beta_1,\alpha_1 \rangle_{\mathrm{Sft}} + \langle \alpha_2,\beta_2 \rangle_{\mathrm{Sft}}\\
&= \langle \alpha_1, (\psi_{-2\pi})_*\beta_1 \rangle_{\mathrm{Sft}} + \langle \alpha_2,\beta_2 \rangle_{\mathrm{Sft}}\\
&= -\langle \alpha_1, (\psi_{-\pi})_*\beta_1 \rangle_{\mathrm{top}} - \langle \alpha_2,(\psi_{\pi})_*\beta_2 \rangle_{\mathrm{top}}.
\end{align*}
Now, note that $\psi_{-\pi}$ on $Y_1$ and $\psi_{\pi}$ on $Y_2$ agree on the boundary $Y_1 \cap Y_2$, so $(\psi_{-\pi})_*$ and $(\psi_{\pi})_*$ are induced by a homeomorphism $\psi$ on $Y$ that ``rotates the boundary''. Hence
\begin{align*}
\langle \alpha,\beta \rangle_{\rmG} &= -\langle \alpha_1, (\psi_{-\pi})_*\beta_1 \rangle_{\mathrm{top}} - \langle \alpha_2,(\psi_{\pi})_*\beta_2 \rangle_{\mathrm{top}}\\
&= -\langle \alpha, \psi_*\beta \rangle_{\mathrm{top}}\\
&= -\langle \alpha, \beta \rangle_{\mathrm{top}},
\end{align*}
where the topological intersection product involving $\alpha$ and $\beta$ is the one on $H_2(U;\bZ)$ given by Definition \ref{def:bilinearonU}.
\end{proof}

This result allows us to identify the exact sequence \eqref{eq:ses} in this setting.

\begin{proposition} There is a commutative diagram of $\bZ$-modules
\[
\begin{tikzcd}[cramped,column sep=1.6em]
0 \ar[r] & \overline{\mathrm{E}} \ar[d] \ar[r]& \mathrm{K}\ar[r]\ar[d] & \rmM \ar[d]\ar{r} & 0 \\
0 \ar[r] & H^1(F_p;\bZ) \ar[r] & H^2_c(U;\bZ)/\bZ F \ar[r] &  {F}^{\perp}/\bZ F \ar[r] & 0
\end{tikzcd}\]
where the vertical arrows are isomorphisms and  the orthogonal complement $F^{\perp}$ is taken inside $H^2(Y;\bZ)$. 

Moreover, the symmetric bilinear forms on $\mathrm{K}$ and $\rmM$ are identified with the negatives of the bilinear form on $H^2_c(U;\bZ)/\bZ F$ induced by Definition \ref{def:bilinearonU} and the form on $F^{\perp}/\bZ F$ induced by the cup product on $H^2(Y;\bZ)$, respectively.
\end{proposition}

\begin{remark} From the viewpoint of the perverse Leray filtration $P_{\bullet}$ on $H^2(Y;\bQ)$, from \cite[Proposition 2.16]{mcss} we have $F^{\perp} \cong P_2H^2(Y;\bQ)$ and $F$ is a generator of $P_1H^2(Y;\bQ)$, so $F^{\perp}/\bQ F$ is isomorphic to the graded piece $\Gr_2^PH^2(Y;\bQ)$. Putting this together with Remark \ref{rem:perverseleray}, we see that the sequence in the proposition above appears as a perverse graded piece in the mirror Clemens-Schmid sequence; see \cite[Sections 3 and 5.1.1]{mcss}.
\end{remark}

\begin{proof} The long exact sequence of cohomology with compact support in this setting reads
\[0 \longrightarrow H^1(F_q;\bZ) \longrightarrow H^2_c(U;\bZ) \longrightarrow H^2(Y;\bZ) \longrightarrow H^2(F_q;\bZ) \longrightarrow 0\]
and the kernel of the map $H^2(Y;\bZ) \to H^2(F_q;\bZ)$ is $F^{\perp}$. Moreover, by construction, the bilinear form on $H^2_c(U;\bZ)$ is compatible with the cup product on $H^2(Y;\bZ)$, as both are Poincar\'{e} dual to the topological intersection pairing.

Consequently, in light of the result of Lemma \ref{lem:GKforms}, to prove the proposition it suffices to show that $\overline{\rmE} \subset {\rmK}$ coincides with the image of $H^1(F_q;\bZ) \to H^2_c(U;\bZ)$ under the isomorphism of Lemma \ref{lem:GKforfibrations}. By definition, $\overline{E}$ is the saturation of the image of the map $r(v) = (r_1T_{-f_2}(v),-r_2(v))$, which gives the class $\beta \in H_2(U,F_p;\bZ)$ obtained by parallel transporting $v \in \rmE \cong H_1(F_p;\bZ)$ around $\partial \Delta_1$, then back along $\partial \Delta_2$ in the opposite direction. The class $\beta$ is homotopy equivalent to the class obtained by parallel transporting $v \in H_1(F_p;\bZ)$ around a small loop containing $q$. As monodromy around $q$ acts trivially on $H_1(F_p;\bZ)$, the class $\beta$ has no boundary, so lifts to $H_2(U;\bZ)$. But the set of all such classes is Poincar\'{e} dual to the image of $H^1(F_q;\bZ) \to H^2_c(U;\bZ)$.
\end{proof}

\subsection{Point-like vectors and the N\'{e}ron-Severi lattice}\label{sec:NSforfibrations}

Next we ask how the ideas of Section \ref{sec:pointsandNS} arise in this setting. 

Let $r_i\colon H_1(F_p;\bZ) \to H_2(Y_i,F_p;\bZ)$ denote the right adjoint of the asymptotic charge map $\phi_i$. As $\phi_i$ is a quasi del Pezzo homomorphism (by Theorem \ref{thm:qdpfibration}), $H_2(Y_i,F_p;\bZ)$ is surface-like with point-like vector $r_i(a)$. 

We can give a topological description of the class $r_i(a)$ as follows. As the class $a \in H_1(F_p;\bZ)$ is monodromy invariant (by Assumption \ref{ass:quasiLG}), define $\alpha \in H_2(Y_i,F_p;\bZ)$ to be the class of a torus obtained by parallel transporting $a$ around $\partial\Delta_i$. Let $\beta \in H_2(Y_i,F_p;\bZ)$ be any other element. Then, since $\alpha$ is preserved under the diffeomorphism $(\psi_{\pi})_*$, we have
\[\langle \beta, -\alpha \rangle_{\mathrm{Sft}} = \langle \beta, (\psi_{\pi})_*\alpha\rangle_{\mathrm{top}} = \langle \beta, \alpha\rangle_{\mathrm{top}} = \langle \phi_i(\beta), a\rangle_{\rmE},\]
where the last equality follows because we can assume $\beta$ and $\alpha$ intersect only on $F_p$. Therefore, by uniqueness of the right adjoint, we have $r_i(a) = - \alpha$.

The N\'{e}ron-Severi group of $\rmG_i$ is given by
\[\NS(\rmG_i) = r_i(a)^{\perp}/r_i(a) .\]
From the right adjoint property, for all $v \in H_2(Y_i,F_p;\bZ)$ we have $\langle v,r_i(a) \rangle_{\rmG_i} = \langle \phi_i(v),a \rangle_{\rmE}$. Moreover, from the properties of the bilinear form on $\rmE \cong H_1(F_p;\bZ)$, we have $\langle \phi_i(v),a \rangle_{\rmE} = 0$ if and only if $\phi_i(v)$ is a multiple of $a$. This gives an explicit description of the N\'{e}ron-Severi group
\[ \NS(\rmG_i)= \{v \in H_2(Y_i,F_p;\bZ)\mid \phi_i(v) \text{ is a multiple of } a\}/r_i(a).\]

Recall next that, from the proof of Lemma \ref{lem:GKforfibrations}, we have isomorphisms of $\bZ$-modules
\[\rmG \cong H_2(U,F_p;\bZ) \cong H_2(Y_1,F_p;\bZ)\oplus H_2(Y_2,F_p;\bZ).\]
In this description, $\Psi$ can be taken to be the sublattice of $\mathrm{G}$ generated by the classes of $\tau_1:=(-r_1(a),0)$ and $\tau_2:=(0,-r_2(a))$. Under the isomorphism $\rmK \cong H^2_c(U;\bZ)/\bZ F$, these generators are the classes of the tori in $U$ obtained by parallel transporting $a$ around the loop $\gamma$, passing to one side of the puncture $q$ to obtain $\tau_1$ and to the other side for $\tau_2$. We thus obtain an isomorphism
\[\Psi^{\perp}_{\rmK}/\Psi \cong (\bZ \tau_1 \oplus \bZ \tau_2)^{\perp}/(\bZ \tau_1 \oplus \bZ \tau_2),\]
where the orthogonal complement is taken in $\rmK \cong H^2_c(U;\bZ)/\bZ F$.

By Proposition \ref{prop:NS}, both $\tau_1$ and $\tau_2$ descend to the same primitive isotropic element $\overline{\tau}$ of $\rmM \cong {F}^{\perp}/\bZ F$. Topologically, $\overline{\tau}$ is induced by the class of the torus obtained by parallel transporting $a$ around the loop $\gamma$. Let $\tau$ be any lift of $\overline{\tau}$ to $F^{\perp} \subset H^2(Y;\bZ)$. Noting that $\bZ F \oplus \bZ \tau$ is a rank $2$ isotropic sublattice of $H^2(Y;\bZ)$ (and that this sublattice is independent of the choice of lift $\tau$), we obtain an isomorphism
\[\NS(\rmM) \cong (\bZ F \oplus \bZ \tau)^{\perp}/(\bZ F \oplus \bZ \tau),\]
where the orthogonal complement is taken in $H^2(Y;\bZ)$.

Finally, by Lemma \ref{lem:zetaisotrivial}, we see that the kernel of the map $\Psi^{\perp}_{\rmK}/\Psi \to \NS(\rmM)$ is generated by $\zeta$, where the sublattice $\bZ \zeta  \subset \Psi^{\perp}_{\rmK}/\Psi $ contains the nonzero element $(-K_{\rmG_1},K_{\rmG_2}) = (r_1(b),-r_2(b))$. Putting things together, we thus obtain the following result.

\begin{proposition} \label{prop:NSforfibrations}
In the setting above, the exact sequence
\[0 \longrightarrow \bZ \zeta \longrightarrow \Psi^{\perp}_{\rmK}/\Psi \longrightarrow \NS(\rmM) \longrightarrow 0\] 
from Proposition \ref{prop:NS} is realised as
\[0 \longrightarrow \bZ  \zeta \longrightarrow (\bZ \tau_1 \oplus \bZ \tau_2)^{\perp}/(\bZ \tau_1 \oplus \bZ \tau_2) \longrightarrow (\bZ F \oplus \bZ \tau)^{\perp}/(\bZ F \oplus \bZ \tau) \longrightarrow 0,\]
where the orthogonal complement in the middle term is taken in $\rmK \cong H^2_c(U;\bZ)/\bZ F$ and the orthogonal complement in the last term is taken in $H^2(Y;\bZ)$. The symmetric bilinear forms on the last two terms are induced by the bilinear form on $H^2_c(U;\bZ)$ given by Definition \ref{def:bilinearonU} and the cup product on $H^2(Y;\bZ)$, respectively.
\end{proposition}

\subsection{Lattice polarisations} \label{sec:fibrationpolarisation}

Let $\check{L}$ be a nondegenerate primitive sublattice of the K3 lattice $\Lambda_{\text{K3}} \cong H \oplus H \oplus H \oplus E_8 \oplus E_8$ of signature $(1,\rank(\check{L})-1)$. Let $Y$ be a K3 surface which admits an $\check{L}$-quasi\-polarisation $\check{L} \subset \NS(Y)$ (in the sense of Definition \ref{def:Lquasipol}) and an elliptic fibration $\pi \colon Y \to \bP^1$. Assume that $\gamma \subset \bP^1$ is an allowable loop; then we may define a class $\overline{\tau} \subset F^{\perp}/\bZ F$, as in the previous section.

\begin{definition}\label{def:compatible} The $\check{L}$-quasi\-polarisation is \emph{compatible with $\pi$ and $\gamma$} if $\check{L} \subset \NS(Y)$ contains the class $F$ of a fibre of $\pi$, and the orthogonal complement $\check{L}^{\perp} \subset H^2(Y;\bZ)$ contains the class $\tau$ of a lift of $\overline{\tau}$.
\end{definition}

\begin{remark} As $\check{L}$ is non-degenerate and $F \in \check{L}$, if a lift $\tau \in \check{L}^{\perp}$ of $\overline{\tau}$ exists, then it is unique, as any other lift $\tau + nF$, for some $n \neq 0$, will not be orthogonal to $\check{L}$.
\end{remark}

We will henceforth assume that the $\check{L}$-quasi\-polarisation on $Y$ is compatible with $\pi$ and $\gamma$, and we will let $\tau \in \check{L}^{\perp}$ denote the class from Definition \ref{def:compatible}. Let $\Gamma$ denote the negative definite lattice $\Gamma := F^{\perp}_{\check{L}}/\bZ F$, where $F^{\perp}_{\check{L}}$ denotes the orthogonal complement of $F$ taken in $\check{L}$.

\begin{lemma} \label{lem:compatiblepolarisationembedding} There is an embedding $\Gamma \hookrightarrow (\bZ F \oplus \bZ \tau)^{\perp}/(\bZ F \oplus \bZ \tau)$.
\end{lemma}
\begin{proof} By the compatibility condition, we have $F^{\perp}_{\check{L}} \subset (\bZ F \oplus \bZ \tau)^{\perp}$. As $\check{L}$ is nondegenerate and $\tau \in \check{L}^{\perp}$, we cannot have $\tau \in F^{\perp}_{\check{L}}$, so $F^{\perp}_{\check{L}} \cap (\bZ F \oplus \bZ \tau) = \bZ F$. It follows that $\Gamma = F^{\perp}_{\check{L}}/\bZ F$ embeds into $(\bZ F \oplus \bZ \tau)^{\perp}/(\bZ F \oplus \bZ \tau)$.
\end{proof}

By Proposition \ref{prop:NSforfibrations}, we obtain an embedding $\Gamma \hookrightarrow \NS(\rmM) \cong (\bZ F \oplus \bZ \tau)^{\perp}/(\bZ F \oplus \bZ \tau)$. For the remainder of this section, we will assume that this embedding is primitive; some conditions to ensure this are explored in Section \ref{sec:primitivity2}. Under this primitivity assumption, $\Gamma \subset \NS(\rmM)$ defines a lattice polarisation, in the sense of Definition \ref{def:Lpol}(1).

Let $\hat{\Gamma}$ denote the pull-back of $\Gamma$ to $\Psi^{\perp}_{\rmK}/\Psi \cong (\bZ \tau_1 \oplus \bZ \tau_2)^{\perp}/(\bZ \tau_1 \oplus \bZ \tau_2)$. Then $\hat{\Gamma}$ defines a lifted polarisation on $\Psi^{\perp}_{\rmK}/\Psi$, in the sense of Definition \ref{def:Lpol}(2). Finally, let $\Gamma_i := \hat{\Gamma} \cap \NS(\rmG_i)$ denote the intersection polarisation.

Next we define the notion of a lattice polarisation on an elliptic fibration over a disc. This definition is motivated by the definition of a lattice polarisation on a rational elliptic surface, defined in \cite[Section 3.1]{mslpdps}; the relationship between the two definitions will be explored in Section \ref{sec:rationalelliptic}.

\begin{definition} \label{def:fibrationintersectionpol} Let $W \to \Delta$ be an elliptic fibration over a closed disc with fibre $F_p$ over a point $p \in \partial \Delta$, such that $\phi \colon H_2(W,F_p;\bZ) \to H_1(F_p;\bZ)$ is a quasi del Pezzo homomorphism with right adjoint $r$. Assume that the Euler number $e(W) < 12$. Let $N$ be a negative definite lattice. An \emph{$N$-polarisation on $W$} is a primitive embedding $j \colon N \hookrightarrow r(a)^{\perp}/r(a) = \NS(H_2(W,F_p;\bZ))$ such that
\begin{itemize}
\item $q(j(v),[r(b)]) = 0$ for all $v \in N$, where the square brackets $[-]$ denote the class in $r(a)^{\perp}/r(a)$, and
\item there exists a set of positive roots $\Delta(N)^+ \subset N$ such that for every $\delta \in \Delta(N )^+$, we may write $j(\delta) = \sum_{i=1}^k n_i[C_i]$, where $n_i \geq 0$ and $[C_i]$ are classes of irreducible components of singular fibres in $W$.
\end{itemize}
\end{definition}

The following result shows that this is compatible with our other definitions.

\begin{theorem} \label{thm:Gammapolfibration} With notation as above, assume that $Y_1$ satisfies $e(Y_1) < 12$. Then the intersection polarisation $\Gamma_1 := \hat{\Gamma} \cap \NS(\rmG_1)$ defines a $\Gamma_1$-polarisation on $Y_1$, in the sense of Definition \ref{def:fibrationintersectionpol}.
\end{theorem}

\begin{proof} Begin by noting that, by Proposition \ref{prop:sphericalsurfacelike}, $q(K_{\mathrm{G}_1},K_{\mathrm{G}_1}) = 12-e(Y_1) > 0$. Applying Lemma \ref{lem:latticeinjection}, we see that the map $\hat{\Gamma} \to \Gamma$ is injective on $\Gamma_1$ and realises $\Gamma_1$ as a sublattice of $\Gamma$. As $\Gamma$ is negative definite, so is $\Gamma_1$. 

The condition that $q(j(v),[r_1(b)]) = 0$ for all $v \in \Gamma_1$ follows from Lemma \ref{lem:Kperp} and the fact that, by Proposition \ref{prop:sphericalsurfacelike}, we have $K_{\rmG_1} = -[r_1(b)]$.

To show the second condition from Definition \ref{def:fibrationintersectionpol}, it suffices to show that if $\delta \in \Gamma_1$ is a class of self-intersection $(-2)$, then either $\delta$ or $(-\delta)$ is a nonnegative sum of classes of components of singular fibres. Let $\delta \in \Gamma_1$ be any such class. Let $\delta'$ denote the image of $\delta$ under the injective map $\Gamma_1 \to \Gamma = F^{\perp}_{\check{L}}/\bZ F$ and let $\overline{\delta}$ denote a lift of $\delta'$ to $F^{\perp}_{\check{L}} \subset H^2(Y;\bZ)$. Then $\overline{\delta}$ is a $(-2)$-class in $\NS(Y)$ and so, replacing $\delta$ with $(-\delta)$ if necessary, we may assume that $\overline{\delta}$ is the class of an effective divisor \cite[Proposition VIII.3.7]{bpv}. Moreover, as $\overline{\delta}$ is orthogonal to the class $F$ of a fibre,  $\overline{\delta}$ must be a nonnegative sum of classes supported on fibres of $\pi$. As $\overline{\delta}$ is defined modulo $\bZ F$, by subtracting copies of $F$ we may assume that  $\overline{\delta}$ is a nonnegative sum of classes of components of singular fibres of $Y$.

Let $\delta_i$ denote the part of $\overline{\delta}$ supported on fibres over points in $\Delta_i$. Then each $\delta_i$ is a nonnegative sum of classes of components of singular fibres of $Y_i$, and $(\delta,0)$ is congruent modulo $\zeta$ to $(\delta_1,\delta_2)$ in $\NS(\rmG_1) \oplus \NS(\rmG_2)$. Writing $(r_1(b),-r_2(b)) = k\zeta$ for some $k \in \bZ\setminus \{0\}$, we therefore have $k(\delta_1 - \delta,\, \delta_2) = m(r_1(b),-r_2(b))$
for some $m \in \bZ$. Applying $\phi_2$ to the second factor and noting that $\phi_2(\delta_2) = 0$ (as $\delta_2$ is a sum of components of singular fibres), we obtain
\[0 = -m\phi_2(r_2(b)) = (\mathrm{id}_{\rmE} - T_{\phi_2})(-mb) = m(e(Y_2) - 12)b\]
as $\phi_2$ satisfies Assumption \ref{ass:quasiLG}. But as $e(Y_2) - 12 = 12 - e(Y_1) > 0$, we must have $m = 0$. Therefore $\delta_2 = 0$ as well, and so $\delta = \delta_1$. We therefore see that $\delta$ can be written as a nonnegative sum of classes of components of singular fibres of $Y_1$.
\end{proof}

\begin{remark} The condition that $e(Y_1) < 12$ holds if and only if $\deg(\rmG) > 0$. As our labelling convention ensures that $e(Y_1) \leq 12$ and we know that $e(Y_1) + e(Y_2) = 24$, we can never have $e(Y_2) < 12$.
\end{remark}

Finally, if $\pi \colon Y \to \bP^1$ has a section, then there is a particularly nice interpretation of the coupling group $Q(\Gamma)$ (Definition \ref{def:coupling}) in this setting.

\begin{proposition} \label{prop:MWQ}
Assume that $\pi \colon Y \to \bP^1$ admits a section $\sigma$ and that $\check{L} = \NS(Y)$. Then there is a surjective map $\MW(Y) \to Q(\Gamma)$.
\end{proposition}

\begin{proof} Let $H \subset \check{L}$ be the copy of the hyperbolic plane lattice spanned by $[\sigma]$ and $F$. Then by \cite[Lemma 5.4]{mslpk3s}, we have an orthogonal direct sum decomposition $\check{L} = \NS(Y) = H \oplus \Gamma$. Standard results on elliptic fibrations with section \cite[(VII.2.5)]{btes} show that there is an exact sequence
\begin{equation} \label{eq:MW} 0 \longrightarrow R \longrightarrow \Gamma \longrightarrow \MW(Y) \longrightarrow 0,\end{equation}
where $R \subset \NS(Y)$ is the sublattice generated by classes of components of fibres which do not meet $\sigma$. 

Let $\Gamma_i := \hat{\Gamma} \cap \NS(\rmG_i)$ be the intersection polarisation for each $i \in \{1,2\}$. By Lemma \ref{lem:latticeinjection} we have a map $\varphi\colon \Gamma_1 \oplus \Gamma_2 \to \Gamma$. We claim that $R$ is a sublattice of $\Gamma_1 \oplus \Gamma_2$. Indeed, $R$ is generated by classes of components of fibres which do not meet $\sigma$. Let $C$ be any such component. Then $C \subset Y_i$ for some $i \in \{1,2\}$ and it follows from the description of the N\'{e}ron-Severi group in Section \ref{sec:NSforfibrations} that the class $[C]$ of $C$ is in $\NS(\rmG_i)$. So $[C] \in \Gamma_i$. As this holds for all generators of $R$, we see that $R \subset \Gamma_1 \oplus \Gamma_2$. The result then follows by a straightforward diagram chase using sequences \eqref{eq:Qsequence} and \eqref{eq:MW}.
\end{proof}

\subsection{Primitivity}\label{sec:primitivity2}

In this section we explore conditions under which the embedding $\Gamma \hookrightarrow (\bZ F \oplus \bZ \tau)^{\perp}/(\bZ F \oplus \bZ \tau)$ from Lemma \ref{lem:compatiblepolarisationembedding} is primitive. Throughout we will assume that $\pi\colon Y \to \bP^1$ is an elliptically fibred K3 surface, $\gamma \subset \bP^1$ is an allowable loop, and that $Y$ admits an $\check{L}$-quasi\-polarisation that is compatible with $\pi$ and $\gamma$.

\begin{proposition} \label{prop:primitivefibration} Suppose that the class $\tau$ is $1$-admissible in $\check{L}^{\perp}$ (Definition \ref{def:madmissible}). Then the embedding $\Gamma \hookrightarrow (\bZ F \oplus \bZ \tau)^{\perp}/(\bZ F \oplus \bZ \tau)$ from Lemma \ref{lem:compatiblepolarisationembedding} is primitive.
\end{proposition}
\begin{proof} By \cite[Lemma 5.4]{mslpk3s}, $1$-admissibility of $\tau$ implies that we may decompose $\check{L}^{\perp}$ as a direct sum of $H$ and another lattice, with $\tau \in H$. By \cite[Proposition 1.6.1]{isbfa}, up to isometry there is a unique primitive embedding $H \hookrightarrow \Lambda_{\text{K3}}$, so we may decompose $\Lambda_{\rmK 3} = H \oplus (H \oplus H \oplus E_8 \oplus E_8)$, with $\tau$ embedding into the first factor $H$ and $\check{L}$ embedding into the second factor $H \oplus H \oplus E_8 \oplus E_8$.

It then suffices to show that the induced embedding $\Gamma \hookrightarrow F^{\perp}/\bZ F$ is primitive, where the orthogonal complement of $F$ is taken in $H \oplus H \oplus E_8 \oplus E_8$. But this follows from the fact that $F \in \check{L}$. Indeed, suppose that $\Gamma$ is not primitive. Then there exists $\alpha \in F^{\perp}/\bZ F$ and $m > 1$ such that $\alpha \notin \Gamma$ but $m \alpha \in \Gamma$. Let $\hat{\alpha} \in F^{\perp}$ be any lift of $\alpha$. Then $\hat{\alpha} \notin \check{L}$, but $m \hat{\alpha} + nF \in \check{L}$ for some $n \in \bZ$. As $F \in \check{L}$, this implies $m \hat{\alpha} \in \check{L}$, contradicting primitivity of $\check{L}$.
\end{proof}

It is interesting to study the analogue of the doubly admissible condition (Definition \ref{def:doublyadmissible}) in this setting.

\begin{definition} \label{def:fibrationdoublyadmissible} An $\check{L}$-quasi\-polarisation on $Y$ that is compatible with $\pi$ and $\gamma$ is called \emph{doubly admissible} if $\tau$ is $1$-admissible in $\check{L}^{\perp}$ and $F$ is $1$-admissible in $\check{L}$.
\end{definition}

The following result gives some conditions for a quasi\-polarisation to be doubly admissible.


\begin{proposition} Suppose that $\pi \colon Y \to \bP^1$ is an elliptically fibred K3 surface with an $\check{L}$-quasi\-polarisation that is compatible with $\pi$ and $\gamma$. Then:
\begin{enumerate}
\item If the $\check{L}$-quasi\-polarisation is doubly admissible, then $\pi$ admits a section.
\item If $\tau$ is $1$-admissible in $\check{L}^{\perp}$ and $\pi$ admits a section $\sigma$ with class $[\sigma]\in \check{L}$, then the $\check{L}$-quasi\-polarisation is doubly admissible.
\end{enumerate}
\end{proposition}
\begin{proof} To prove (1), by \cite[Lemma 5.4]{mslpk3s}, if the $\check{L}$-quasi\-polarisation is doubly admissible then there exists an embedding $H \hookrightarrow \check{L} \subset \NS(Y)$. By \cite[Remark 11.1.4]{lok3s}, this implies that $\pi$ admits a section. (2) follows by taking $g$ to be the class $g = [\sigma] + F \in \check{L}$ in the definition of $1$-admissible (Definition \ref{def:madmissible}).
\end{proof}

\begin{remark} Note that in part (1) of the above proposition, the class of the section of $\pi$ is not necessarily contained in $\check{L}$.
\end{remark}

The next lemma is the analogue of Proposition \ref{prop:doublyadmissibleconsequnces} in this setting; it will be useful when we come to discuss mirror symmetry in Section \ref{sec:mirror}.

\begin{lemma} \label{lem:doublyadmissibleconsequencesfibration} 
Let $\check{L} \subset \Lambda_{\rmK 3}$ be a nondegenerate primitive sublattice. Let $F$ be $1$-admissible in $\check{L}$ and $\tau$ be $1$-admissible in $\check{L}^{\perp}$.
Then $\check{L}$ admits an orthogonal decomposition $\check{L} = H \oplus \Gamma$, with $F$ primitive in $H$. Moreover, the primitive embedding $\Gamma \hookrightarrow (\bZ \tau \oplus \bZ F)^{\perp}$ factors through primitive embeddings $\Gamma \hookrightarrow H \oplus E_8 \oplus E_8 \hookrightarrow (\bZ \tau \oplus \bZ F)^{\perp}$, so that the composition $H \oplus E_8 \oplus E_8 \hookrightarrow (\bZ \tau \oplus \bZ F)^{\perp} \to (\bZ \tau \oplus \bZ F)^{\perp}/(\bZ \tau \oplus \bZ F)$ of this primitive embedding with the quotient map is an isomorphism.
\end{lemma}

\begin{proof}
Applying \cite[Lemma 5.4]{mslpk3s} to $F$ and $\tau$, we obtain orthogonal direct sum decompositions $\check{L} = H \oplus \Gamma$ and $\check{L}^{\perp} = H \oplus \Gamma'$, for some lattice $\Gamma'$, where the two copies of $H$ contain $F$ and $\tau$ respectively. The direct sum $\check{L} \oplus \check{L}^{\perp}$ embeds into the K3 lattice $\Lambda_{\rmK 3}$ and this embedding is primitive on each of the factors, so we obtain a primitive embedding $H \oplus H \hookrightarrow \Lambda_{\rmK 3}$. By \cite[Proposition 1.6.1]{isbfa}, this embedding is unique up to isometry and, consequently, its orthogonal complement is isometric to $H \oplus E_8 \oplus E_8$. As we have a primitive embedding $\bZ \tau \oplus \bZ F \hookrightarrow H \oplus H$, we obtain primitive embeddings
\[\Gamma \hookrightarrow (H \oplus H)^{\perp}_{\Lambda_{\rmK 3}} \cong H \oplus E_8 \oplus E_8 \hookrightarrow (\bZ \tau \oplus \bZ F)^{\perp}\]
and the final map descends to an isomorphism on the quotient.
\end{proof}

\subsection{Relationship with lattice polarised rational elliptic surfaces}\label{sec:rationalelliptic}

In this section we will show that the theory above is compatible with the notion of lattice polarisation on a rational elliptic surface introduced in \cite{mslpdps}.

We begin by recalling the setup from \cite[Section 3]{mslpdps}. For consistency with the notation in \cite{mslpdps}, \emph{in this section only} we will let $\pi \colon Y \to \bP^1$ be a rational elliptic surface, such that the fibre $D = \pi^{-1}(\infty)$ is of Kodaira type $\rmI_d$ for some $1 \leq d \leq 9$. Write $D = \bigcup_{i=1}^d D_i$ as a union of irreducible components, numbered cyclically. We assume that $\pi$ does not have multiple fibres, which for rational elliptic surfaces is equivalent to the existence of a section. 

Let $U$ denote the open set $U := Y \setminus D$. Let ${\Delta}$ denote the closed set $\Delta := \bP^1 \setminus B_{\varepsilon}(\infty)$, where $B_{\varepsilon}(\infty)  \subset \bP^1$ is an open disc of radius $\varepsilon > 0$ centred at $\infty$ and $\varepsilon$ is chosen small enough that the only singular fibre over the closure $\overline{B_{\varepsilon}(\infty)}$ is $D$, and let $W$ denote the closed set $W := \pi^{-1}(\Delta)$. Note that $W$ is a closed manifold with boundary, so there is an intersection pairing on $H_2(W;\bZ)$, and that $W$ is a deformation retract of $U$. Finally, let $p \in \partial \Delta$ be a point and let $F = \pi^{-1}(p)$ be the corresponding smooth fibre. 

Under this setup, we note that the Euler number $e(W) = e(Y) - d = 12-d$ and that the action of anticlockwise monodromy around $\partial \Delta$ is given by the inverse of the action of anticlockwise monodromy around an $\rmI_d$ fibre. Thus, from Kodaira's classification, there exists a symplectic basis $(a,b)$ for $H_1(F;\bZ)$ such that  the action of anticlockwise monodromy around $\partial \Delta$ is given by
\[\begin{pmatrix}1 & d \\ 0 & 1 \end{pmatrix}^{-1} = \begin{pmatrix}1 & -d \\ 0 & 1 \end{pmatrix} = \begin{pmatrix}1 & e(W)-12 \\ 0 & 1 \end{pmatrix}.\]
Let $\phi\colon H_2(W,F;\bZ) \to H_1(F;\bZ)$ denote the asymptotic charge map and $r$ denote its right adjoint.

\begin{proposition} \label{prop:ratellqdp}
With notation as above, $\phi\colon H_2(W,F;\bZ) \to H_1(F;\bZ)$ is a quasi del Pezzo homomorphism.
\end{proposition}

\begin{proof} If we can prove that $r(a) \in H_2(W,F;\bZ)$ is primitive, then Assumption \ref{ass:quasiLG} will hold in this setting and so, by Theorem \ref{thm:qdpfibration}, $\phi\colon H_2(W,F;\bZ) \to H_1(F;\bZ)$ will be a quasi del Pezzo homomorphism.

Let $\Delta' := B_{\delta}(\infty)$ be an open disc of radius $\delta > \varepsilon$ centred at $\infty \in \bP^1$, where $\delta$ is chosen small enough that the only singular fibre over ${B_{\delta}(\infty)}$ is $D$, and let $W' = \pi^{-1}(\Delta')$. Then $Y = W \cup W'$ and Mayer-Vietoris gives
\[H_2(Y;\bZ) \longrightarrow H_1(W \cap W';\bZ) \longrightarrow H_1(W;\bZ) \oplus H_1(W';\bZ) \longrightarrow 0.\]

Now, $W \cap W'$ deformation retracts onto $\pi^{-1}(\partial \Delta)$, which is a torus bundle over $S^1$ with monodromy $\begin{psmallmatrix} 1 & -d \\ 0 & 1\end{psmallmatrix}$. It is then easy to compute that $H_1(W \cap W';\bZ) = \bZ s \oplus \bZ b \oplus (\bZ/d\bZ) a$, where $s$ is the image of the section of $Y$ under the map $H_2(Y;\bZ) \to H_1(W \cap W';\bZ)$ and $(a,b)$ are the classes from $H_1(F;\bZ)$ pushed-forward by the inclusion $F \hookrightarrow W \cap W'$.

Moreover, $W'$ deformation retracts onto $D$, which is an $\rmI_d$ fibre with $d \geq 1$, so $H_1(W';\bZ) = \bZ b$ (the class $a$ is the vanishing cycle). Thus, by the Mayer-Vietoris sequence above, we see that $H_1(W;\bZ)$ must be a torsion group generated by the push-forward of $a$ under the inclusion $F \hookrightarrow W$.

From the long exact sequence of a pair we have
\[ H_2(W,F;\bZ) \stackrel{\phi}{\longrightarrow}  H_1(F;\bZ) \stackrel{f}{\longrightarrow} H_1(W;\bZ),\]
where $f$ is the push-forward under the inclusion. By the argument above, $H_1(W;\bZ)$ is generated by $f(a)$, so $f(b) = 0$. Thus $b \in \im(\phi)$. Choose any $v \in H_2(W,F;\bZ)$ with $\phi(v) = b$. Then $\langle v, r(a) \rangle_{\mathrm{Sft}} = \langle b, a \rangle_{\rmE} = 1,$
by adjunction. As the Seifert pairing is integral, it follows that $r(a)$ must be primitive.
\end{proof}

Based on \cite[Section 3.1]{mslpdps} we define a lattice polarisation on $Y$ as follows.

\begin{definition} \label{def:ratellpol} Let $\pi \colon Y \to \bP^1$ be a rational elliptic surface with section, let $D \subset Y$ be a fibre of Kodaira type $\rmI_d$, and let $[F] \in H^2(Y;\bZ) \cong \Pic(Y)$ be the class of a fibre of $\pi$. Let $N$ be a negative definite lattice. An \emph{$N$-polarisation on $(Y,D)$} is a primitive embedding $j\colon N \hookrightarrow [F]^{\perp}/\bZ [F]$ such that
\begin{itemize}
\item $j(\beta).D_i = 0$ for all $\beta \in N$ and all irreducible components $D_i \subset D$ (where $1 \leq i \leq d$), and
\item there exists a set of positive roots $\Delta(N)^+ \subset N$ such that every $\delta \in j(\Delta(N)^+)$ is the class of an effective divisor from $H^2(Y;\bZ)$.
\end{itemize}
\end{definition}

\begin{remark} This definition differs from \cite[Definition 3.6]{mslpdps} in that we embed $N$ into $[F]^{\perp}/\bZ [F]$, rather than  into $H^2(Y;\bZ) \cong \Pic(Y)$. These definitions are not equivalent: an embedding $N \hookrightarrow H^2(Y;\bZ)$ satisfying \cite[Definition 3.6]{mslpdps} induces an embedding $N \hookrightarrow [F]^{\perp}/\bZ [F]$ as above, but this latter embedding is not necessarily primitive. Conversely, an embedding $N \hookrightarrow [F]^{\perp}/\bZ [F]$ as above does not uniquely determine an embedding $N \hookrightarrow H^2(Y;\bZ)$ satisfying \cite[Definition 3.6]{mslpdps}. However, it follows from the results of \cite[Section 3]{mslpdps} that there is still a well-defined period mapping and corresponding moduli theory for lattice polarised rational elliptic surfaces in the sense of Definition \ref{def:ratellpol}.
\end{remark}

The aim of this section is to show that Definition \ref{def:ratellpol} is compatible with Definition \ref{def:fibrationintersectionpol}. We begin with a pair of lemmas that relate the lattices that we are interested in, before proving the main result (Theorem \ref{thm:ratellcompatibility}).

\begin{lemma} \label{lem:ratellembedding1} There is a primitive lattice embedding
\[H_2(W;\bZ)/\bZ[F] \hookrightarrow r(a)^{\perp} \subset H_2(W,F;\bZ),\]
where $H_2(W;\bZ)$ is equipped with the bilinear form induced by the topological intersection pairing, $[F]$ is the totally degenerate class of the fibre $F$ in $H_2(W;\bZ)$, and $r(a)^{\perp} \subset H_2(W,F;\bZ)$ is equipped with the negative of the Seifert pairing (which is symmetric on $r(a)^{\perp}$ as $r(a)$ is point-like).

Moreover, the totally degenerate class $r(a) \subset r(a)^{\perp}$ is in the image of this embedding, so we have a primitive lattice embedding
\[H_2(W;\bZ)/(\bZ r(a) \oplus \bZ[F]) \hookrightarrow r(a)^{\perp}/r(a),\]
where $r^{\perp}(a)/r(a)$ is the N\'{e}ron-Severi lattice of $H_2(W,F;\bZ)$. The image of this embedding is those classes $v$ in the N\'{e}ron-Severi lattice with $q(v,[r(b)]) = 0$.
\end{lemma}

\begin{proof}
From the long exact sequence of a pair we have
\[H_2(F;\bZ) \longrightarrow H_2(W;\bZ) \stackrel{f}{\longrightarrow} H_2(W,F;\bZ) \stackrel{\phi}{\longrightarrow} H_1(F;\bZ),\]
and the map $f$ takes the topological intersection product on  $H_2(W;\bZ)$ to the negative of the Seifert pairing on $H_2(W,F;\bZ)$. The image of the first map is $\bZ[F]$. To prove the first part of the lemma, it thus suffices to show that the image of the map $f$ is a primitive sublattice of $r(a)^{\perp}$. Let $v \in \im(f)$ be any class. Then $\langle v,r(a) \rangle_{\mathrm{Sft}} = \langle \phi(v),a \rangle_{\rmE} = 0$, as $\phi(v) = 0$ by exactness, so $\im(f) \subset r(a)^{\perp}$. Primitivity follows from the fact that $H_1(F;\bZ)$ is torsion-free.

To prove the second part, by exactness $\im(f) = \ker(\phi)$. As the map $\phi$ is a quasi del Pezzo homomorphism, by Lemma \ref{lem:riproperties} we have $\phi \circ r(a) = 0$ and thus $r(a) \in \ker(\phi) = \im(f)$. Finally, note that $v \in \im(f) = \ker(\phi)$ if and only if $\langle \phi(v), a\rangle_{\rmE} = \langle \phi(v),b \rangle_{\rmE} = 0$. A class $v \in r(a)^{\perp}$ satisfies this if and only if $\langle v,r(b) \rangle_{\mathrm{Sft}} = -q(v,[r(b)]) = 0$.
\end{proof}

\begin{lemma} \label{lem:ratellembedding2} There is a primitive lattice embedding
\[H_2(W;\bZ)/\bZ r(a) \hookrightarrow [F]^{\perp} \subset  H^2(Y;\bZ) \cong \Pic(Y),\]
where $H_2(W;\bZ)$ is equipped with the bilinear form induced by the topological intersection pairing and $[F]$ is the class of a fibre $F$.

Moreover, the totally degenerate class $[F]$ is in the image of this embedding, so we have a primitive lattice embedding
\[H_2(W;\bZ)/(\bZ r(a) \oplus \bZ[F]) \hookrightarrow [F]^{\perp}/\bZ [F].\]
The image of this embedding is those classes $\beta$ with $\beta.D_i = 0$ for all irreducible components $D_i \subset D$ (where $1 \leq i \leq d$).
\end{lemma}

\begin{proof}  Note first that, under Poincar\'{e} duality, we have an isomorphism of lattices $H^2(Y;\bZ) \cong H_2(Y;\bZ)$, where $H_2(Y;\bZ)$ is equipped with the topological intersection pairing $\langle \cdot , \cdot \rangle_{\mathrm{top}}$. 

Set $D_0 := D_1 \cap D_d$; note that $D_0$ is a point. Define $U_0 = Y$ and  $U_{i+1} = U_{i} \setminus D_i$ for each $0 \leq i \leq d$; note that $U_{d+1} = U$. Then the embeddings $U_{i+1} \hookrightarrow U_{i}$ induce pushforward maps $H_2(U_{i+1};\bZ) \to H_2(U_{i};\bZ)$, which fit into the following long exact sequence that is Poincar\'{e} dual to the exact sequence of cohomology with compact support
\[H_1(U_{i} \cap D_i;\bZ) \longrightarrow H_2(U_{i+1};\bZ) \longrightarrow H_2(U_{i};\bZ) \longrightarrow H_0(U_{i} \cap D_i;\bZ).\]
Noting that $U_{i} \cap D_i$ deformation retracts onto a point if $0 \leq i \leq d-1$ and a loop $S^1$ homotopic to the vanishing cycle $a$ if $i=d$, we see that the map $H_2(U_{i+1};\bZ) \to H_2(U_{i};\bZ)$ is injective if $0 \leq i \leq d-1$ and has kernel $\bZ \alpha$ generated by the class $\alpha$ of a torus obtained by parallel transporting $a$ around $\partial \Delta$ if $i=d$. Moreover, the image of this map consists of those classes $\beta \in H_2(U_{i};\bZ)$ with $\langle \beta, D_i\rangle_{\mathrm{top}} = 0$. 

Composing these maps for $0 \leq i \leq d$, we see that the map in homology $H_2(U;\bZ) \to H_2(Y;\bZ)$ induced by the inclusion $U \to Y$ has kernel $\bZ \alpha$ and image consisting of those classes $\beta \in H_2(Y;\bZ)$ with $\langle \beta, D_i\rangle_{\mathrm{top}} = 0$ for all $1 \leq i \leq d$; in particular the image is contained in $[F]^{\perp}$. Moreover, by deformation retraction we have an isomorphism $H_2(W;\bZ) \cong H_2(U;\bZ)$, which identifies $\alpha$ with $-r(a)$ (see Section \ref{sec:NSforfibrations}). 

Putting everything together, we therefore obtain an injective map of lattices $H_2(W;\bZ)/\bZ r(a) \hookrightarrow [F]^{\perp}$, proving the first part of the lemma. The image of this map consists of those classes $\beta \in [F]^{\perp}$ with $\beta.D_i = 0$ for all $1 \leq i \leq d$. As the class $[F]$ satisfies $[F].D_i = 0$ for all $1 \leq i \leq d$, this class is in the image, and the second statement follows.
\end{proof}

\begin{theorem} \label{thm:ratellcompatibility} Let $N$ be a negative definite lattice. With notation as above, if $j\colon N \hookrightarrow [F]^{\perp}/\bZ [F]$ is an $N$-polarisation on $(Y,D)$, in the sense of Definition \ref{def:ratellpol}, then $j$ induces an $N$-polarisation on $W$, in the sense of Definition \ref{def:fibrationintersectionpol}.

Conversely, if $j'\colon N \hookrightarrow r^{\perp}(a)/r(a)$ is an $N$-polarisation on $W$, in the sense of Definition \ref{def:fibrationintersectionpol}, then $j'$ induces an $N$-polarisation on $(Y,D)$, in the sense of Definition \ref{def:ratellpol}.
\end{theorem}

\begin{proof} Note first that, by Lemmas \ref{lem:ratellembedding1} and \ref{lem:ratellembedding2}, we have embeddings
\[ 
\begin{tikzcd}
r(a)^{\perp}/r(a) \arrow[r, hookleftarrow, "\varphi"] & H_2(W;\bZ)/(\bZ r(a) \oplus \bZ[F]) \arrow[r, hookrightarrow, "\psi"] & {[F]}^{\perp}/\bZ {[F]}.
\end{tikzcd}
\]

 By Lemma \ref{lem:ratellembedding2}, if $j\colon N \hookrightarrow [F]^{\perp}/\bZ [F]$ is an $N$-polarisation on $(Y,D)$, then $j(N)$ is in the image of the embedding $\psi$, so there is a primitive embedding $N \hookrightarrow H_2(W;\bZ)/(\bZ r(a) \oplus \bZ[F])$. Composing with the primitive embedding $\varphi$, we obtain an embedding $N \hookrightarrow r(a)^{\perp}/r(a)$, such that $q(v,[r(b)]) = 0$ for all $v \in N$.

If $\delta \in  \Delta(N)^+$, then since $j(\delta).D_i = 0$ for all components $D_i \subset D$ and $j(\delta)$ is the class of an effective divisor, $j(\delta)$ can be represented as an effective sum of classes of components of fibres in $H_2(W;\bZ)$. Quotienting by $\bZ[F]$, we see that $j(\delta)$ can be expressed as a nonnegative sum of classes of components of singular fibres of $W$. We thus obtain an $N$-polarisation on $W$, in the sense of Definition \ref{def:fibrationintersectionpol}.

Conversely, by Lemma \ref{lem:ratellembedding1}, if $j\colon N \hookrightarrow r(a)^{\perp}/r(a)$ is an $N$-polarisation on $W$, then $j(N)$ is in the image of the embedding $\varphi$, so there is a primitive embedding $N \hookrightarrow H_2(W;\bZ)/(\bZ r(a) \oplus \bZ[F])$. Composing with the primitive embedding $\psi$,  we obtain an embedding $N \hookrightarrow [F]^{\perp}/\bZ [F]$, such that $v.D_i=0$ for all $v \in N$ and all $1 \leq i \leq d$.

If $\delta \in  \Delta(N)^+$, then $j(\delta)$ is a nonnegative sum of classes of components of singular fibres in $H_2(W;\bZ)$. But every such class is also the class of an effective divisor from $H^2(Y;\bZ)$. We thus obtain an $N$-polarisation on $(Y,D)$, in the sense of Definition \ref{def:ratellpol}.
\end{proof}

\section{Mirror Symmetry}\label{sec:mirror}

The DHT philosophy, introduced in \cite{mstdfcym}, postulates a mirror correspondence between Tyurin degenerations and codimension $1$ fibrations on Calabi-Yau manifolds. The aim of this section is to make this correspondence precise in the case of K3 surfaces.

\subsection{Mirror Pairs}\label{sec:mirrorpairs}

We begin by defining when a Tyurin degeneration of K3 surfaces and an elliptically fibred K3 surface form a \emph{mirror pair}.

Let $L$ and $\check{L}$ be two nondegenerate sublattices of $\Lambda_{\rmK 3}$ of signatures $(1,\rank(L)-1)$ and $(1,\rank(\check{L})-1)$, respectively. Note that we do not, a priori, assume any relationship between $L$ and $\check{L}$.

Let $\calX \to \Delta$ be an $L$-polarised Tyurin degeneration of K3 surfaces, with general fibre $X$ and central fibre $X_0 = V_1 \cup_C V_2$. Then we have an embedding $L \hookrightarrow I^{\perp}/{I}$, by Lemma \ref{lem:tyurinpseudolattice}. Moreover, by Lemma \ref{lem:latticeinjection} and Proposition \ref{prop:NSfordegenerations} there is a map
\[\varphi_X \colon K_{V_1}^{\perp} \oplus K_{V_2}^{\perp} \longrightarrow I^{\perp}/I\]
which is injective on each of the factors $K_{V_i}^{\perp} \subset \NS(V_i)$.

Let $\pi \colon Y \to \bP^1$  be an elliptically fibred K3 surface which admits an $\check{L}$-quasi\-polarisation. Assume that $\gamma \subset \bP^1$ is an allowable loop, such that the $\check{L}$-quasi\-polarisation is compatible with $\pi$ and $\gamma$. Let $p \in \gamma$ be a point, let $F_p = \pi^{-1}(p)$ be the fibre over $p$, and let $F$ denote the class of $F_p$ in $\check{L}$. If $\Gamma := F^{\perp}_{\check{L}}/\bZ F$, then by Lemma \ref{lem:compatiblepolarisationembedding} there is an embedding $\Gamma \hookrightarrow (\bZ F \oplus \bZ \tau)^{\perp} / (\bZ F \oplus \bZ \tau)$, where  $\tau \in \check{L}^{\perp}$ is the class from Definition \ref{def:compatible}. 
Then $\gamma$ induces a splitting $\bP^1 = \Delta_1 \cup_{\gamma} \Delta_2$ of $\bP^1$ into a pair of discs; let $\pi_i\colon Y_i \to \Delta_i$ denote the induced elliptic fibrations over discs. By Lemma \ref{lem:latticeinjection} and Proposition \ref{prop:NSforfibrations} there is a map
\[\varphi_Y \colon [r_1(b)]^{\perp} \oplus [r_2(b)]^{\perp} \longrightarrow (\bZ F \oplus \bZ \tau)^{\perp} / (\bZ F \oplus \bZ \tau)\]
which is injective on each of the factors $[r_i(b)]^{\perp} \subset \NS(H_2(Y_i,F_p;\bZ))$.

\begin{definition} \label{def:mirrorpair}
A Tyurin degeneration $\calX \to \Delta$ and elliptically fibred K3 surface $\pi\colon Y \to \bP^1$ as above are called a \emph{mirror pair} if the following conditions hold.
\begin{enumerate}
\item The embeddings $L \hookrightarrow I^{\perp}/{I}$ and $\Gamma \hookrightarrow (\bZ F \oplus \bZ \tau)^{\perp} / (\bZ F \oplus \bZ \tau)$ are primitive.
\item There is an isometry $\psi\colon (\bZ F \oplus \bZ \tau)^{\perp} / (\bZ F \oplus \bZ \tau)  \stackrel{\sim}{\longrightarrow} I^{\perp}/I$ satisfying 
\[\psi\circ \varphi_Y([r_i(b)]^{\perp}) = \varphi_X(K_{V_i}^{\perp}),\] 
for each $i \in \{1,2\}$.
\item $\psi$ can be lifted to an isometry $\hat{\psi}\colon H^2(Y;\bZ) \to H^2(X;\bZ)$ taking $\tau$ and $F$ to primitive generators $e_1$ and $e_2$ of $I$, respectively.
\item For some $m \in \bN$, either
\begin{enumerate}
\item $\tau$ is $m$-admissible in $\check{L}^{\perp}$ and $\hat{\psi}(\check{L}^{\perp}) = H(m) \oplus L$, with $e_1 = \hat{\psi}(\tau) \in H(m)$, or
\item $e_1$ is $m$-admissible in $L^{\perp}$ and $\hat{\psi}^{-1}(L^{\perp}) = H(m) \oplus \check{L}$, with $\tau = \hat{\psi}^{-1}(e_1) \in H(m)$.
\end{enumerate}
\end{enumerate}
\end{definition}

\begin{remark}By Theorem \ref{thm:qdpclassification} and Propositions \ref{prop:qdpisomorphism} and \ref{prop:NSMautos}, an isometry satisfying condition (2) always exists if the corresponding pseudolattices $\rmG$ on both sides have the same degree. In all cases except degrees $8$ and $8'$ this reduces to checking that $12 - e(Y_1) = K_{V_1}^2$.\end{remark}

Using conditions (1), (3) and (4) from Definition \ref{def:mirrorpair}, it is straightforward to show that $\psi(\Gamma) = L^{\perp}$, where the orthogonal complement is taken in $I^{\perp}/I$. Our next result shows that if our polarisations are doubly admissible, then the converse also holds. This allows for a dramatic simplification of Definition \ref{def:mirrorpair} in the doubly admissible setting.

\begin{proposition} \label{prop:mirrordoublyadmissible} Suppose that the $L$-polarised Tyurin degeneration $\calX \to \Delta$ and the $\check{L}$-quasi\-polarisation on $Y$ are both doubly admissible, in the sense of Definitions \ref{def:doublyadmissible} and \ref{def:fibrationdoublyadmissible}. Then $\calX \to \Delta$ and $\pi\colon Y \to \bP^1$ are a mirror pair, in the sense of Definition \ref{def:mirrorpair}, if and only if there is an isometry $\psi\colon (\bZ F \oplus \bZ \tau)^{\perp} / (\bZ F \oplus \bZ \tau)  \stackrel{\sim}{\longrightarrow} I^{\perp}/I$ satisfying $\psi(\Gamma) = L^{\perp}$ and $\psi\circ \varphi_Y([r_i(b)]^{\perp}) = \varphi_X(K_{V_i}^{\perp})$, for each $i \in \{1,2\}$.
\end{proposition}

\begin{proof} The forward direction is immediate from the discussion above; it suffices to prove the converse. So assume that there is an isometry $\psi\colon (\bZ F \oplus \bZ \tau)^{\perp} / (\bZ F \oplus \bZ \tau)  \stackrel{\sim}{\longrightarrow} I^{\perp}/I$ satisfying $\psi(\Gamma) = L^{\perp}$ and $\psi\circ \varphi_Y([r_i(b)]^{\perp}) = \varphi_X(K_{V_i}^{\perp})$, for each $i \in \{1,2\}$; we check the four conditions for $\calX \to \Delta$ and $\pi\colon Y \to \bP^1$ to be a mirror pair from Definition \ref{def:mirrorpair}. 

By Propositions \ref{prop:doublyadmissibleconsequnces} and \ref{prop:primitivefibration}, condition (1) is a consequence of the doubly admissible assumption. Moreover, condition (2) holds by assumption.

By the proof of Lemma \ref{lem:doublyadmissibleconditions}, there is an isometry $\phi_X \colon H^2(X;\bZ) \to H \oplus H \oplus H \oplus E_8 \oplus E_8$, taking primitive generators $e_1,e_2 \in I$ to primitive isotropic vectors in each of the first two factors of $H$ and taking $L$ to a primitive sublattice of the remaining $H \oplus E_8 \oplus E_8$. Similarly, writing $\check{L} = H \oplus \Gamma$ with $F \in H$ primitive isotropic, by the proof of Lemma \ref{lem:doublyadmissibleconsequencesfibration} there is an isometry $\phi_Y \colon H^2(Y;\bZ) \to H \oplus H \oplus H \oplus E_8 \oplus E_8$ taking $\tau$ to a primitive isotropic vector in the first factor of $H$,  taking the copy of $H$ in the decomposition  $\check{L} = H \oplus \Gamma$ to the second factor of $H$, and taking $\Gamma$ to a primitive sublattice of $H \oplus E_8 \oplus E_8$. As primitive isotropic vectors in $H$ are unique up to isometry, we may arrange that $\phi_Y(\tau) = \phi_X(e_1)$ and $\phi_Y(F) = \phi_X(e_2)$. Moreover, by Proposition \ref{prop:doublyadmissibleconsequnces} and Lemma \ref{lem:doublyadmissibleconsequencesfibration}, there are isomorphisms $H \oplus E_8 \oplus E_8 \cong I^{\perp}/I$ and $H \oplus E_8 \oplus E_8 \cong (\bZ F \oplus \bZ \tau)^{\perp} / (\bZ F \oplus \bZ \tau)$ compatible with the quotients; consequently the isometry $\psi$ may be lifted to an isometry $\overline{\psi}$ of $H \oplus H \oplus H \oplus E_8 \oplus E_8$ which fixes the first two factors of $H$. Then $\phi_X^{-1} \circ \overline{\psi} \circ \phi_Y$ is the required isometry for condition (3). 

Finally, it follows from the definition of doubly admissible that $\tau$ is $1$-admissible in $\check{L}^{\perp}$ and $e_1$ is $1$-admissible in $L^{\perp}$. By the construction above and the assumption that $\psi(\Gamma) = L^{\perp}$, we have: 
\begin{enumerate}[(a)]
\item $\overline{\psi} \circ \phi_Y(\check{L}^{\perp}) = \overline{\psi} \circ \phi_Y((H \oplus \Gamma)^{\perp}) = \phi_X(H \oplus L)$, where the copy of $H$ in $H \oplus \Gamma$ (resp. $H \oplus L$) is the one containing $F$ (resp. $e_1$). Part (a) of condition (4) follows.
\item $\overline{\psi} \circ \phi_Y(H \oplus \check{L}) = \overline{\psi} \circ \phi_Y(H \oplus H \oplus \Gamma) = \phi_X(L^{\perp})$, where the two copies of $H$ are the ones containing $\tau$ and $F$ respectively. Part (b) of condition (4) follows.
\end{enumerate}
\end{proof}

\subsection{Compatibility with other forms of mirror symmetry} \label{sec:mirrorcompatibility}

In this section we explore the compatibility between the notion of mirror pairs established in Definition \ref{def:mirrorpair} and some established notions of mirror symmetry.

\subsubsection{Lattice polarised mirror symmetry for K3 surfaces}\label{sec:dolgachevnikulin} 

Given a mirror pair in the sense of Definition \ref{def:mirrorpair}, condition (4) in the definition immediately implies that a general fibre $X$ of the Tyurin degeneration $\calX \to \Delta$ and the elliptically fibred K3 surface $Y$ are mirror as lattice polarised K3 surfaces, in the sense of \cite{mslpk3s}. 

Note that if $m = 1$ in condition (4), then (4a) holds if and only if (4b) does and the mirror correspondence is a duality; see \cite[Section 6]{mslpk3s}.

\subsubsection{Homological mirror symmetry for quasi del Pezzo surfaces and quasi LG models}

In \cite{pdpslf} it is shown that if $V$ is a quasi del Pezzo surface with smooth anticanonical divisor $C$, then there exists a genus $1$ Lefschetz fibration $W \to \Delta$ over a disc with a smooth fibre $F_p$ over a point $p$ on the boundary, so that there are isomorphisms of pseudolattices making the following diagram commute:
\[\begin{tikzcd}
H_2(W,F_p;\bZ) \ar{d}{\phi} \ar{r}{\sim} & \rmK_0^{\mathrm{num}}(\mathbf{D}(V))  \ar{d}{i^*} &  \\
H_1(F_p;\bZ) \ar{r}{\sim} & \rmK_0^{\mathrm{num}}(\mathbf{D}(C)) 
\end{tikzcd}\]
where $i\colon C \hookrightarrow V$ is the inclusion and $\phi$ is the asymptotic charge map. 

To see how this is a homological mirror symmetry statement, let $\mathbf{F}(F_p)$ denote the derived Fukaya category of $F_p$ and $\mathbf{F}(W)$ denote the derived Fukaya-Seidel category of $W$. Then there is an isomorphism $\rmK_0^{\mathrm{num}}(\mathbf{F}(F_p)) \cong H_1(F_p;\bZ)$ which takes the Euler pairing to the intersection form, and an isomorphism $\rmK_0^{\mathrm{num}}(\mathbf{F}(W)) \cong H_2(W,F_p;\bZ)$ which takes the Euler pairing to the Seifert pairing \cite[Section 6.1]{fscslf}. Homological mirror symmetry predicts a derived equivalence between $\mathbf{D}(V)$ and $\mathbf{F}(W)$, compatible with the derived equivalence between $\mathbf{D}(C)$ and $\mathbf{F}(F_p)$ established in \cite{cmsec,hms4t}; the statement above can be thought of as a shadow of this on the level of numerical Grothendieck groups.

Our notion of mirror pairs enjoys the following compatibility with the theory above.

\begin{theorem} \label{thm:HMS}
Given a mirror pair in the sense of Definition \ref{def:mirrorpair}, for each $j \in \{1,2\}$ there are isomorphisms of pseudolattices which make the following diagram commute
\[\begin{tikzcd}
 H_2(Y_j,F_p;\bZ) \ar{d}{\phi_j} \ar{r}{\sim} & \rmK_0^{\mathrm{num}}(\mathbf{D}(V_j))  \ar{d}{i^*_j} & \\
 H_1(F_p;\bZ) \ar{r}{\sim} &  \rmK_0^{\mathrm{num}}(\mathbf{D}(C)).
\end{tikzcd}\]
\end{theorem}

\begin{proof}
Condition (2) in Definition \ref{def:mirrorpair} and Proposition \ref{prop:NSMautos} imply that the isometry $\psi$ lifts to isometries $\psi_j\colon \NS(H_2(Y_j,F_p;\bZ)) \to \NS(V_j)$ such that $\psi_j([r_j(b)]) = -K_{V_j}$.  Existence of the isomorphisms in the theorem then follows from Proposition \ref{prop:qdpisomorphism}.
\end{proof}

\begin{remark} The pseudolattice isomorphisms $H_2(Y_j,F_p;\bZ) \cong \rmK_0^{\mathrm{num}}(\mathbf{D}(V_j))$  do not induce a lattice isometry between $H^2_c(U;\bZ)/\bZ F$ and $H^0(X_0;\bZ) \oplus \xi^{\perp} \oplus H^4(X_0;\bZ)$, as the Chern character is not defined over $\bZ$ (see Lemma \ref{lem:KoverZ}). However, they do induce a lattice isometry between $F^{\perp}/\bZ F$ and $H^0_{\lim}(X) \oplus I^{\perp}/I \oplus H^4_{\lim}(X)$ (see Lemma \ref{lem:Lagrees}), and (up to sign) this isomorphism induces $\psi$ on the N\'{e}ron-Severi lattices.
\end{remark}

\subsubsection{Mirror symmetry for lattice polarised del Pezzo surfaces}

In \cite{mslpdps}, Doran and the second author conjectured a mirror symmetric relationship between lattice polarised weak del Pezzo surfaces and lattice polarised rational elliptic surfaces. This should be thought of as a lattice polarised enhancement of Auroux's, Katzarkov's, and Orlov's \cite{msdpsvccs} description of Landau-Ginzburg models of del Pezzo surfaces as open sets inside rational elliptic surfaces. The aim of this section is to explore the relationship between this conjecture and the theory presented here.

In more detail, let $V$ be a weak del Pezzo surface with smooth anticanonical divisor $C$, such that $(V,C)$ admits an $N$-polarisation by a negative definite lattice $N$, in the sense of Definition \ref{def:Lpolwdp}. Let $\check{N}$ denote the orthogonal complement of $N$ in $K_V^{\perp}$; note that $\check{N}$ is also negative definite, as $V$ is weak del Pezzo (see Lemma \ref{lem:NSproperties}). Then, roughly speaking, \cite[Conjecture 5.1]{mslpdps} postulates that the mirror to an $N$-polarised weak del Pezzo surface $(V,C)$ with $C^2=d$ should be given by the open set $Y \setminus D$ in an $\check{N}$-polarised rational elliptic surface $(Y,D)$ (in the sense of Definition \ref{def:ratellpol}), where $D \subset Y$ is a fibre of Kodaira type $\rmI_d$. For a more detailed statement we refer the reader to \cite{mslpdps}; in particular, we note that there is some subtlety when $d=8$.

By Proposition \ref{prop:ratellqdp}, a rational elliptic surface $(Y,D)$ as above determines a quasi del Pezzo homomorphism $\phi\colon H_2(W,F_p;\bZ) \to H_1(F_p;\bZ)$, where $W$ is obtained by removing an open tubular neighbourhood around $D$ and $F_p$ is a fibre on the boundary of $W$. Moreover, by Theorem \ref{thm:ratellcompatibility}, specifying an $\check{N}$-polarisation on $(Y,D)$, in the sense of Definition \ref{def:ratellpol}, is equivalent to specifying an $\check{N}$-polarisation on $W$, in the sense of Definition \ref{def:fibrationintersectionpol}. These observations allow us to reformulate \cite[Conjecture 5.1]{mslpdps} in the language of this paper, as follows.

\begin{conjecture}[Mirror symmetry for lattice polarised weak del Pezzo surfaces] \label{con:wdpmirror}  The mirror to an $N$-polarised weak del Pezzo surface $(V,C)$ is given by an elliptic fibration $W \to \Delta$ over a closed disc, with fibre $F_p$ over a point $p \in \partial \Delta$, such that
\begin{enumerate}
\item $\phi\colon H_2(W,F_p;\bZ) \to H_1(F_p;\bZ)$ is a quasi del Pezzo homomorphism which is isomorphic to $\rmK_0^{\mathrm{num}}(\mathbf{D}(V)) \to \rmK_0^{\mathrm{num}}(\mathbf{D}(C))$, and
\item $W$ is $\check{N}$-polarised, in the sense of Definition \ref{def:fibrationintersectionpol}, where $\check{N}$ denotes the orthogonal complement of $N$ in $K_V^{\perp}$.
\end{enumerate}
\end{conjecture}

Now we explore the relationship between this conjecture and the notion of a mirror pair from Definition \ref{def:mirrorpair}. In particular, if we have a mirror pair with $12 - e(Y_1) = K_{V_1}^2 > 0$, so $V_1$ is weak del Pezzo, we would like to know when the intersection polarisations on $(V_1,C)$ and $Y_1$ satisfy Conjecture \ref{con:wdpmirror}.

From Proposition \ref{prop:couplingintersection}, we expect the answer to be related to the coupling group. The following lemma will be helpful.

\begin{lemma} \label{lem:couplingtorsion}
Suppose we have a mirror pair, in the sense of Definition \ref{def:mirrorpair}, and assume that $12 - e(Y_1) = K_{V_1}^2 > 0$. Then the coupling group $Q(L)$ is torsion if and only if the coupling group $Q(\Gamma)$ is torsion.
\end{lemma}

\begin{proof} This is a straightforward consequence of Proposition \ref{prop:couplingintersection} and the fact that, for a mirror pair, we have $\psi(\Gamma) = L^{\perp}$.
\end{proof}

Now we state the main theorem of this section.

\begin{theorem} \label{thm:wdpcompatibility}
Suppose that we have a mirror pair, in the sense of Definition \ref{def:mirrorpair}, and assume that $V_1$ is weak del Pezzo. Then $(V_1,C)$ is an $L_1$-polarised weak del Pezzo surface (in the sense of Definition \ref{def:Lpolwdp}) and $Y_1 \to \Delta_1$ is a $\Gamma_1$-polarised fibration over a disc (in the sense of Definition \ref{def:fibrationintersectionpol}), where $L_1$ and $\Gamma_1$ denote the intersection polarisations induced by $L$ and $\Gamma$ respectively. 

Moreover, if the coupling groups $Q(L)$ and $Q(\Gamma)$ are torsion, then $(V_1,C)$ and $Y_1\to \Delta_1$ satisfy Conjecture \ref{con:wdpmirror}, with $N = L_1$ and $\check{N} = \Gamma_1$.
\end{theorem}

\begin{remark}
By Lemma \ref{lem:couplingtorsion}, if one of $Q(L)$ and $Q(\Gamma)$ is torsion, then both are. By Proposition \ref{prop:MWQ}, if $\check{L} = \NS(Y)$ then a sufficient condition for this to hold is that $\MW(Y)$ is torsion.
\end{remark}

\begin{proof}
$(V_1,C)$ is an $L_1$-polarised weak del Pezzo surface by Theorem \ref{thm:Lpolwdp} and $Y_1$ is a $\Gamma_1$-polarised fibration over a disc by Theorem \ref{thm:Gammapolfibration} (noting that $12 - e(Y_1) = K_{V_1}^2 > 0$).

Given this, condition (1) from Conjecture \ref{con:wdpmirror} follows from Theorem \ref{thm:HMS}. Finally, to prove condition (2) we need to show that the orthogonal complement of $L_1$ in ${K_{V_1}^{\perp}}$ is isomorphic to $\Gamma_1$; this follows from  Proposition \ref{prop:couplingintersection} and the fact that, for a mirror pair, we have $\psi(\Gamma) = L^{\perp}$.
\end{proof}

\subsection{The DHT conjecture} \label{sec:dht}

The DHT philosophy \cite{mstdfcym} postulates that there is a mirror correspondence between Tyurin degenerations and codimension $1$ fibrations on Calabi-Yau manifolds. Moreover, it suggests that it should be possible to divide the base of the fibration into pieces which ``look like'' the Landau-Ginzburg models of the components of the Tyurin degeneration.

The K3 case of this idea is explored somewhat in \cite[Section 4]{mstdfcym}, which presents some evidence for this idea but does not state a concrete conjecture. Using the theory developed in this paper, we are now able to state a precise version of the DHT conjecture for K3 surfaces.

\begin{conjecture}[DHT for K3 surfaces] \label{conj:DHT}
Let $\calX \to \Delta$ be an $L$-polarised Tyurin degeneration of K3 surfaces, with general fibre $X$. Assume that there is a primitive $e \in I$ which is $m$-admissible in $L^{\perp}$ and let $\check{L} = e^{\perp}_{L^{\perp}}/\bZ e$ denote the mirror lattice to $L$. Let $Y$ be an $\check{L}$-quasi\-polarised K3 surface. Then $Y$ admits an elliptic fibration $\pi\colon Y \to \bP^1$ and there exists an allowable loop $\gamma \subset \bP^1$ such that the $\check{L}$-quasi\-polarisation is compatible with $\pi$ and $\gamma$, so that $\calX \to \Delta$ and $\pi\colon Y \to \bP^1$ are a mirror pair in the sense of Definition \ref{def:mirrorpair}.

Conversely, let $Y$ be an $L'$-quasi\-polarised K3 surface which admits an elliptic fibration  $\pi\colon Y \to \bP^1$. Suppose that there is an allowable loop $\gamma \subset \bP^1$ such that the $L'$-quasi\-polarisation is compatible with $\pi$ and $\gamma$. Assume that the class $\tau$ is $m$-admissible in $(L')^{\perp}$ and let $\check{L}' = \tau^{\perp}_{(L')^{\perp}} / \bZ\tau$ denote the mirror lattice to $L'$. Then there exists an $\check{L}'$-polarised Tyurin degeneration of K3 surfaces $\calX \to \Delta$, so that $\calX \to \Delta$ and $\pi\colon Y \to \bP^1$ are a mirror pair in the sense of Definition \ref{def:mirrorpair}.
\end{conjecture}

\begin{remark}
The difficult part of this conjecture is the choice of the loop $\gamma$ or, conversely, the choice of birational model for the Tyurin degeneration. The examples in \cite{nftdk3spr18l} show that knowledge of the class $\tau \in \check{L}^{\perp}$ is not sufficient to determine $\gamma$ and, conversely, knowledge of the rank $2$ isotropic sublattice $I \subset (\check{L}')^{\perp}$ is not sufficient to determine the birational model for the Tyurin degeneration.
\end{remark}

\begin{remark} We expect the conjecture above to relate doubly admissible Tyurin degenerations to doubly admissible elliptically fibred K3 surfaces.

Indeed, suppose that $\calX \to \Delta$ is a doubly admissible $L$-polarised Tyurin degeneration. Choose primitive generators $e_1,e_2 \in I$ as in Definition \ref{def:doublyadmissible}. By Proposition \ref{prop:doublyadmissibleconsequnces} and the proof of Lemma \ref{lem:doublyadmissibleconditions}, we may decompose $L^{\perp} = H \oplus H \oplus \Gamma$, where $e_1$ (resp. $e_2$) is a primitive element of the first (resp. second) factor of $H$.

By Proposition \ref{prop:doublyadmissible0cusp}, the $1$-cusp in the Baily-Borel compactification of the moduli space of $L$-polarised K3 surfaces corresponding to $I \subset L^{\perp}$ is incident to a unique $0$-cusp which, up to isometry, we may assume corresponds to the isotropic vector $e_1$. Let $\check{L} = (e_1)^{\perp}_{L^{\perp}}/\bZ e_1$ be the mirror lattice. Then $\check{L} \cong H \oplus \Gamma$ and $e_2$ descends to a primitive isotropic vector in the factor $H \subset \check{L}$. It follows that $e_2 \in \check{L}$ and $e_1 \in \check{L}^{\perp}$ are both $1$-admissible. Moreover, if $Y$ is a K3 surface with $\NS(Y) \cong \check{L}$, it follows from \cite[Remark 8.2.13]{lok3s} that $Y$ is elliptically fibred and that we may choose the isometry $\check{L} \to \NS(Y)$ so that $e_2$ is taken to the class of a fibre. If Conjecture \ref{conj:DHT} holds, we expect that the class $\tau \in \check{L}^{\perp}$ defined by the loop $\gamma \subset \bP^1$ should equal $e_1$; if this is the case then the $\check{L}$-quasi\-polarisation on $Y$ will be doubly admissible and compatible with $\pi$ and $\gamma$.

Conversely, if $\pi\colon Y \to \bP^1$ is an elliptically fibred K3 surface with a doubly admissible $L'$-quasi\-polarisation, then $F \in L'$ and $\tau \in (L')^{\perp}$ are both $1$-admissible. If $\check{L}' = \tau^{\perp}_{(L')^{\perp}} / \bZ\tau$ is the mirror lattice to $L'$, then $F$ and $\tau$ give rise to a doubly admissible isotropic sublattice $I \subset (\check{L}')^{\perp}$. If Conjecture \ref{conj:DHT} holds, we expect this $I \subset (\check{L}')^{\perp}$ to be the rank $2$ isotropic sublattice corresponding to the mirror $\check{L}'$-polarised Tyurin degeneration.
\end{remark}

\bibliography{publications,preprints}
\bibliographystyle{amsalpha}
\end{document}